\documentclass{aims}
\usepackage{amsmath}
  \usepackage{paralist}
  \usepackage{graphics} %% add this and next lines if pictures should be in esp format
  \usepackage{epsfig} %For pictures: screened artwork should be set up with an 85 or 100 line screen
\usepackage{graphicx}  \usepackage{epstopdf}%This is to transfer .eps figure to .pdf figure; please compile your paper using PDFLeTex or PDFTeXify.
 \usepackage[colorlinks=true]{hyperref}
\hypersetup{urlcolor=blue, citecolor=red}

\def\currentvolume{X}
 \def\currentissue{X}
  \def\currentyear{200X}
   \def\currentmonth{XX}
    \def\ppages{X--XX}
     \def\DOI{10.3934/xx.xx.xx.xx}

\newtheorem{theorem}{Theorem}[section]
\newtheorem{corollary}{Corollary}

\newtheorem{lemma}[theorem]{Lemma}
\newtheorem{proposition}{Proposition}

\newtheorem{algorithm}[theorem]{Algorithm}
\theoremstyle{definition}
\newtheorem{definition}[theorem]{Definition}
\newtheorem{remark}{Remark}
\newtheorem*{notation}{Notation}

\numberwithin{equation}{section}

\title[On the construction of Lyapunov functions]{On the construction of Lyapunov functions\\ with computer assistance}

\author[Kaname Matsue, Tomohiro Hiwaki and Nobito Yamamoto]{}


 \email{kmatsue@ism.ac.jp}
 \email{h1341007@edu.cc.uec.ac.jp}
 \email{yamamoto@im.uec.ac.jp}

\thanks{The first author is supported by Coop with Math Program, MEXT, Japan.}

\thanks{$^\ast$ Corresponding author: Kaname Matsue}

\begin{document}
\maketitle

% Enter the first author's name and address:
\centerline{\scshape Kaname Matsue$^\ast$}
\medskip
{\footnotesize
% please put the address of the first author
 \centerline{The Institute of Statistical Mathematics}
 \centerline{10-3, Midori-Cho} 
  \centerline{Tachikawa, 190-8562, Tokyo, Japan}
   
} % Do not forget to end the {\footnotesize by the sign }

\medskip

\centerline{\scshape Tomohiro Hiwaki}
\medskip
{\footnotesize
 % please put the address of the second  and third author
 \centerline{Graduate School of Informatics and Engineering}
   \centerline{The University of Electro-Communications}
  \centerline{1-5-1 Chofugaoka, Chofu, Tokyo 182-8585, Japan}
}

\medskip

\centerline{\scshape Nobito Yamamoto}
\medskip
{\footnotesize
 % please put the address of the second  and third author
 \centerline{Graduate School of Informatics and Engineering}
   \centerline{The University of Electro-Communications}
  \centerline{1-5-1 Chofugaoka, Chofu, Tokyo 182-8585, Japan}
%  \par
% \centerline{JST, CREST}
%   \centerline{The University of Electro-Communications}
%  \centerline{1-5-1 Chofugaoka, Chofu, Tokyo 182-8585, Japan}
}

\bigskip

% The name of the associate editor will be entered by an editorial staff
% "Communicated by the associate editor name" is not needed for special issue.
 \centerline{(Communicated by the associate editor name)}

%The abstract of your paper
\begin{abstract}
Computer assisted procedures of Lyapunov functions defined in given neighborhoods of fixed points for flows and maps are discussed.
We provide a systematic methodology for constructing explicit ranges where quadratic Lyapunov functions exist in two stages; negative definiteness of associating matrices and direct approach.
We note that the former is equivalent to the procedure of cones describing enclosures of the stable and the unstable manifolds of invariant sets, which gives us flexible discussions of asymptotic behavior not only around equilibria for flows but also fixed points for maps.
Additionally, our procedure admits a re-parameterization of trajectories in terms of values of Lyapunov functions.
Several verification examples are shown for discussions of applicability.
\end{abstract}

%The title of your section 1
\section{Introduction}
In this paper, we provide a systematic method to validate the domain of Lyapunov functions around hyperbolic fixed points both for continuous and for discrete dynamical systems with computer assistance.
%We also give the construction of Lyapunov functions around periodic orbits as fixed points of Poincar\'{e} maps as well as validation examples.
\par
\bigskip
First of all, we recall the definition of Lyapunov functions for flows in the typical sense.
Let $\varphi : \mathbb{R}\times \mathbb{R}^n \to \mathbb{R}^n$ be a flow on $\mathbb{R}^n$.
\begin{definition}[e.g. \cite{Rob}]\rm
\label{dfn-Lyapunov-typical}
Let $U\subset \mathbb{R}^n$ be an open subset. 
Consider the differential equation
\begin{equation}
\label{autonomous-intro}
\dot {\bf x} = f({\bf x}),\quad f\in C^1(\mathbb{R}^n, \mathbb{R}^n).
\end{equation}
A {\em Lyapunov function} $L:U\to \mathbb{R}^n$ (for the flow) is a $C^1$-function satisfying the following conditions.
\begin{enumerate}
\item $(dL/dt)(\varphi(t,{\bf x}))\mid_{t=0}\leq 0$ holds for each solution orbit $\{\varphi(t,{\bf x})\}$ through ${\bf x} \in U$.
\item $(dL/dt)(\varphi(t,{\bf x}))\mid_{t=0}= 0$ implies $\varphi(t,{\bf x})\equiv \bar {\bf x}\in U$, where $\bar {\bf x}$ is an equilibrium of (\ref{autonomous-intro}).
\end{enumerate}
\end{definition}

%\begin{definition}[e.g. \cite{Rob}]\rm
%\label{dfn-Lyapunov-typical}
%Let $\varphi:\mathbb{R}\times \mathbb{R}^n\to \mathbb{R}^n$ be a smooth flow on $\mathbb{R}^n$. 
%A {\em Lyapunov function for $\varphi$} is a $C^1$-function $L: \mathbb{R}^n \to \mathbb{R}$ satisfying
%\begin{enumerate}
%\item $(dL/dt)(\varphi(t,x))\mid_{t=0}\leq 0$ holds for each solution orbit $\{\varphi(t,x)\mid t\in \mathbb{R}\}$ through $\varphi(0,x) = x\in \mathbb{R}^n$.
%\item $(dL/dt)(\varphi(t,x))\mid_{t=0}= 0$ implies $\varphi(t,x)\equiv \bar x$, where $\bar x$ is an equilibrium of $\varphi$.
%\end{enumerate}
%Next let $\psi:\mathbb{R}^n \to \mathbb{R}^n$ be a diffeomorphism and consider iterations of $\psi$. 
%A {\em Lyapunov function for $\psi$} is a $C^1$-function $L: \mathbb{R}^n \to \mathbb{R}$ satisfying
%\begin{enumerate}
%\item $L(\psi(x)) \leq L(x)$ holds for each $x\in \mathbb{R}^n$.
%\item $L(\psi(x)) = L(x)$ implies $\psi(x) = x$, i.e., $x$ is a fixed point of $\psi$.
%\end{enumerate}
%\end{definition}
Lyapunov functions are heart of gradient dynamical systems, which ensure that all trajectories behave so that Lyapunov functions decrease monotonously. 
Such behavior is expected locally around hyperbolic fixed points or general invariant sets.
Once we construct Lyapunov functions locally around invariant sets, 
they let us easy to understand local dynamics in terms of level sets of Lyapunov functions. 
In Definition \ref{dfn-Lyapunov-typical}, a Lyapunov function $L$ is, if exists, defined in an open set $U\subset \mathbb{R}^n$, but it does not tell us how large $U$ can be chosen.
Furthermore, explicit constructions of Lyapunov functions themselves are not easy because of nonlinearity of dynamical systems. 
For typical systems which possess Lyapunov functions with explicit forms, such systems themselves are fully determined by these functions, which are often called {\em gradient systems}; namely, $\frac{d}{dt}{\bf x} = \nabla L({\bf x})$.
Although there is an abstract result concerning with the existence of Lyapunov functions, which is known as {\em Conley's Fundamental Theorem of Dynamical Systems} (e.g. \cite{Con, Rob}), detections of the concrete form of $L$ and the concrete shape of $U$ near invariant sets remain open and depend on individual dynamical systems.
Despite the great importance for dynamical systems, construction of Lyapunov functions in concrete systems remains open.
%Nevertheless, local analysis around fixed points with interval arithmetics contribute the explicit construction of locally defined Lyapunov functions, e.g. \cite{Mat}. 
%Similarly, such analyses also contribute the construction of {\em cones}, e.g. \cite{ZCov}, which describes local dynamics in the stable and unstable directions around invariant sets \cite{Rob}.
%
\par
There are several preceding works for validating functionals like Lyapunov functions within explicit domains (e.g. \cite{C, KWZ2007, Mat2}), many of which apply functionals called {\em cones} satisfying {\em cone conditions} (e.g. \cite{ZCov}) to understanding asymptotic behavior around invariant sets .
Cones with cone conditions restrict the behavior of points in terms of differences between two points. 
In particular, these concepts describe stable and unstable manifolds of invariant sets as graphs of Lipschitzian (or smooth) functions.
On the other hand, there is a preceding study for constructing Lyapunov functions in the sense of Definition \ref{dfn-Lyapunov-typical} in explicitly given neighborhoods of hyperbolic equilibria \cite{Mat}. 
There Lyapunov functions have very simple forms, and a sufficient condition for validating Lyapunov functions in given domains as well as hyperbolicity of equilibria for (semi)flows is proposed.

In these two directions, there are several similarities. 
Firstly, functionals (cones, Lyapunov functions) describing asymptotic behavior have {\em quadratic forms} around equilibria or fixed points.
Secondly, validations of functionals are done via negative definiteness of matrices associated with functionals.
Thirdly, computer assisted analysis via {\em interval arithmetics} are applied to validating given quadratic forms being cones or Lyapunov functions in explicitly given domains.

\bigskip
This paper aims a systematic procedure of Lyapunov functions around fixed points both for continuous and discrete dynamical systems by simple forms in given domains with computer assistance.
This procedure gives us a general implementation of Lyapunov functions around fixed points, which can be applied to various dynamical systems, as shown in preceding works (e.g. \cite{C, KWZ2007, Mat2, W, WZ}).
We also discuss relationships between preceding works; cones with cone conditions, Lyapunov functions in \cite{Mat}, and present studies, as well as a new aspect concerning with re-parameterization of trajectories.
We believe that presenting arguments lead to a comprehensive understanding of Lyapunov function validations.
\par
Our central target function $L$ is a quadratic function as preceding works. 
The determination of domains $L$ being a Lyapunov function, called {\em Lyapunov domain}, consists of following two stages. 
\begin{itemize}
\item {\bf Stage 1: Negative definiteness of the matrix associated with $dL/dt$}.
\end{itemize}
For a domain containing an equilibrium, we verify the negative definiteness of specific matrix associated with $dL/dt(\varphi(t,{\bf x}))$ along solution orbits with computer assistance, which gives us a sufficient condition so that $L$ is a Lyapunov function in the domain.
\begin{itemize}
\item {\bf Stage 2: Direct calculations of the negativity of $dL/dt$ along trajectories}.
\end{itemize}
For domains which do not contain equilibria, we calculate $dL/dt(\varphi(t,{\bf x}))$ directly with computer assistance and verify if it is negative.

If such criteria pass for given domains, the quadratic function $L$ is validated as a Lyapunov function on the domain.
Combination of validations in these two stages with computer assistance gives us a useful tool to validate explicit Lyapunov domains of given quadratic functions as large as possible, which will be the basis for constructing locally defined Lyapunov functions in various dynamical systems.

\bigskip
Our paper is organized as follows.
%\par
In Section \ref{section-cont}, we discuss the construction of Lyapunov functions for flows (continuous dynamical systems).
We provide a systematic construction of Lyapunov functions being quadratic.
We also generalize Lyapunov functions to {\em ${\bf m}$-Lyapunov functions}, which is a generalization of {\em $m$-cones} in \cite{Mat2}.
In Section \ref{section-aspect}, several geometric and algebraic aspects containing mathematical validity of numerically computed objects associated with Lyapunov functions, as well as relevance to preceding works (e.g. \cite{C, Mat, W, WZ, ZCov}) are discussed.
%\par
In Section \ref{section-discrete}, we discuss the construction of Lyapunov functions for maps (discrete dynamical systems).
The basic strategy for construction is similar to flows.
A discrete dynamical systems' version of ${\bf m}$-Lyapunov functions is also derived.
%\par
In Section \ref{section-Lyapunov-periodic}, 
we briefly review the computation of Poincar\'{e} maps for flows and their differentials so that we apply arguments in Section \ref{section-discrete} to Poincar\'{e} maps, which yields the construction of Lyapunov functions for periodic orbits.
%\par
In Section \ref{section-example-cont} and \ref{section-example-discrete}, several validation examples are shown for demonstrating applicability of arguments in previous sections.
The former refers to flows and the latter refers to Poincar\'{e} maps.

\begin{notation}
For scalars $t$ or vectors ${\bf x} = (x_1,\cdots, x_n)$, $[t]$ and $[{\bf x}]$ denote interval enclosures containing $t$ and ${\bf x}$, respectively. 
More precisely, $[{\bf x}]$ is a vector whose $i$-th entry is $[x_i]$ for $i=1,\cdots, n$.
For a function $f$ and objects $x$ (scalars or vectors), the set $f([x])$ is defined as 
$f([x]) = \{f(z)\mid z\in [x]\}$.
An interval matrix $[A]$ denotes a matrix whose entries are intervals. Namely, $[A] = ([a_{ij}])_{i,j=1}^n$, $[a_{ij}] = [a_{ij}^-, a_{ij}^+]\subset \mathbb{R}$.
A matrix-valued function $F(D) = (F_{ij}(D))_{i,j=1}^n$ of a set $D$, say rectangular domains, is defined by $F_{ij}(D) = \{F_{ij}({\bf x})\mid {\bf x}\in D\}$.
These notations are used in discussions with interval arithmetics.
\end{notation}

%The title of your section 2
\section{Lyapunov functions for continuous dynamical systems}
\label{section-cont}
In this section, we consider dynamical systems generated by ordinary differential equations, in particular, an autonomous system
\begin{equation}
\label{autonomous}
\displaystyle \frac{d}{dt}{\bf x} = f({\bf x}),\quad  {\bf x} \in \mathbb{R}^n,
\end{equation}
where $f : \mathbb{R}^n\to \mathbb{R}^n$ is a smooth map.
Let ${\bf x^\ast}$ be an equilibrium of (\ref{autonomous}). 
Also, assume that ${\bf x^\ast}$ is hyperbolic. 
Namely, all eigenvalues of the Jacobian matrix $Df({\bf x^\ast})$ are away from the imaginary axis.
\par
A fundamental result on dynamical systems called {\em Hartman-Grobman's Theorem} (e.g. \cite{Rob}) claims that 
the dynamics around a hyperbolic equilibrium ${\bf x}^\ast$ is topologically conjugate to the dynamics generated by the linearized matrix $Df({\bf x}^\ast)$ at ${\bf x}^\ast$ in a small neighborhood of ${\bf x}^\ast$. 

As for linear vector fields, we can easily construct Lyapunov functions. Indeed, any linear vector fields can be generically written as, under the change of coordinates,
\begin{equation}
\label{aut-linear}
\displaystyle \frac{d}{dt}{\bf x} = \Lambda {\bf x}, \quad \Lambda = {\rm diag}(\lambda_1,\cdots, \lambda_n).
\end{equation}
In this case, ones easily see that the functional 
\begin{equation*}
L({\bf x}) := -\sum_{i=1}^n {\rm sgn}({\rm Re}(\lambda_i)) x_i^2
\end{equation*}
is a Lyapunov function for (\ref{aut-linear}).
Under this observation, it is natural to consider that there are {\em locally defined} Lyapunov functions around, at least, hyperbolic equilibria.
Moreover, Lyapunov functions can be quadratic around hyperbolic equilibria. 
%With this in mind, we generalize the definition of Lyapunov functions as follows.

Our aim here is to provide a systematic procedure for constructing quadratic Lyapunov functions defined not only in a small neighborhood of hyperbolic equilibria but also in their explicitly given neighborhoods.
The latter can be realized with computer assistance such as interval arithmetics.

% New subsection
\subsection{Construction of quadratic functions}
\label{section-cont-construction}

An observation in the previous subsection gives us an expectation that Lyapunov functions can be locally constructed as {\em the quadratic function}.
Now we provide a basic strategy for constructing Lyapunov functions. 
Let ${\bf x}^\ast$ be an equilibrium for (\ref{autonomous}).

\begin{itemize}
\item[1.] 
Let $Df^\ast$ be the Jacobian matrix $Df({\bf x^\ast})$ of $f$ at ${\bf x}={\bf x}^\ast$.
For simplicity, assume that $Df^\ast$ is diagonalizable by a nonsingular matrix $X$, which is generically valid: 
\begin{equation*}
\Lambda = X^{-1}Df^\ast X,
\end{equation*}
where $\Lambda = {\rm diag}(\lambda_1,\lambda_2,\cdots,\lambda_n)$. 
Note that computations of $\Lambda$ and $X$ do not need to apply interval arithmetics.
Assume that, in the sense of floating-point arithmetics, ${\rm Re}(\lambda_i)\not = 0$ for all $i=1,\cdots, n$.
\item[2.] 
	Let $I^\ast$ be the diagonal matrix $I^\ast = {\rm diag}(i_1,i_2,\cdots,i_n)$, where
\begin{equation}
\label{sign-change-cont}
i_k = \left\{ \begin{array}{c} 
                1, \quad \mbox{if} \quad \mbox{Re}(\lambda_k) < 0, \\
                -1, \quad \mbox{if} \quad \mbox{Re}(\lambda_k) > 0. 
                \end{array}\right.
\end{equation}
Note that ${\rm Re}(\lambda_k)\not = 0$ since $x^\ast$ is hyperbolic.
\item[3.] 
Calculate the real symmetric matrix $Y = (Y_{ij})_{i,j=1}^n$ as follows:
\begin{equation}
\label{Y-cont}
\hat{Y} :=  X^{-H} I^\ast X^{-1}, \quad Y := \mbox{Re}( \hat{Y}),
\end{equation}
          where $X^{-H}$ denotes the inverse matrix of the Hermitian transpose $X^H$ of $X$, which is sufficient to be calculated by floating-point arithmetics.
\item[4.] 
Define the quadratic function $L$ by
\begin{equation}
\label{Lyapunov-flow}
L({\bf x}) := ({\bf x}-{\bf x}^\ast)^T Y ({\bf x}-{\bf x}^\ast),
\end{equation}
which is our candidate of Lyapunov function around $x^\ast$. If we deal with $L$ with interval arithmetics, we replace $Y$ by $(Y+Y^T)/2$ or set $Y_{ji} = Y_{ij}$ so that we keep the symmetry of $Y$. 
\end{itemize}

Note that the matrix $Y$ is determined to be {\em real and symmetric}.
The operation $Y\mapsto (Y+Y^T)/2$ avoids a delicate case that $Y$ is not actually symmetric due to, say, numerical errors.
It is valid for our arguments to provide various properties of $L$ as a Lyapunov function, which can be seen below.

% New subsection
\subsection{Validity of $L({\bf x})$ in (\ref{Lyapunov-flow})}
Here we find a sufficient condition such that the function $L({\bf x})$ in (\ref{Lyapunov-flow}) is indeed a Lyapunov function in a given  neighborhood of ${\bf x^\ast}$. 
%Moreover, we prove that such a sufficient condition is indeed satisfied in a sufficiently small neighborhood of ${\bf x^\ast}$.

%Summarizing the above observations,
%\begin{proposition}
%Let ${\bf x}^\ast$ be a hyperbolic fixed point of (\ref{autonomous}). Then there is a neighborhood of $U$ of ${\bf x}^\ast$ such that $A({\bf z})$ given by (\ref{matrix-neg-def}) is negative definite for all points in $U$. In particular, $L({\bf x})$ is a Lyapunov function on $U$.
%\end{proposition}

Note that quadratic forms with respect to Hermitian matrices take real values.
In particular, 
\begin{equation*}
{\bf z}^T H {\bf z} = {\bf z}^T \mbox{Re}(H) {\bf z}
\end{equation*}
holds for a Hermitian matrix $H$ and a real vector $z$. 

%\begin{proposition}
%\label{prop-sufficient-cont}
%Let $D_L$ be a compact, star-shaped domain centered at an equilibrium ${\bf x}^\ast$ of (\ref{autonomous}). 
%Define a matrix $A({\bf z})$ by
%\begin{equation}
%\label{matrix-neg-def}
%A({\bf z}) = Df({\bf z})^T Y + Y Df({\bf z}).
%\end{equation}
%Assume that the matrix $A({\bf z})$ is strictly negative definite for all ${\bf z}\in D_L$. 
%Then $L({\bf x})$ is a Lyapunov function on $D_L$.
%\end{proposition}

Now we have the following theorem, which indicates that our procedure validates 
(i) Lyapunov functions on a compact, star-shaped domain $D_L$, and
(ii) the uniqueness of equilibria in $D_L$
at the same time.

\begin{theorem}
\label{thm-Lyapunov-cont}
Let $D_L$ be a compact, star-shaped domain centered at an equilibrium ${\bf x}^\ast$ of (\ref{autonomous}). 
Define a matrix $A({\bf z})$ by
\begin{equation}
\label{matrix-neg-def}
A({\bf z}) = Df({\bf z})^T Y + Y Df({\bf z}).
\end{equation}
Assume that the matrix $A({\bf z})$ is strictly negative definite for all ${\bf z}\in D_L$. 
Then $L({\bf x})$ is a Lyapunov function on $D_L$. 
Moreover, ${\bf x}^\ast$ is the unique equilibrium in $D_L$.
\end{theorem}

In our arguments, Lyapunov functions are considered in compact domains. Nevertheless, thanks to the continuous dependence of $A({\bf z})$, we can easily extend functions to open neighborhoods of such compact domains, and hence our arguments with respect to Lyapunov functions in the sense of Definition \ref{dfn-Lyapunov-typical} still make sense.

\begin{proof}
Let ${\bf x}(t)$ be a solution orbit of (\ref{autonomous}) with ${\bf x}(0) = {\bf x}$. Differentiating $L({\bf x}(t))$ with respect to $t$ along ${\bf x}(t)$, we obtain
\begin{equation*}
\displaystyle \frac{d}{dt}L({\bf x})_{t=0} = 
f({\bf x})^T Y ({\bf x}-{\bf x}^\ast) + 
({\bf x}-{\bf x}^\ast)^T Y f({\bf x}).
\end{equation*}
Now we consider the function
\begin{equation*}
g(s) = f({\bf x}^\ast+s({\bf x} - {\bf x}^\ast)).
\end{equation*}
Since
\begin{equation*}
\displaystyle \frac{dg}{ds}(s) =
     Df({\bf x}^\ast+s({\bf x} - {\bf x}^\ast))({\bf x} - {\bf x}^\ast)
\end{equation*}
and $f({\bf x}^\ast)={\bf 0}$, we obtain
\begin{equation*}
f({\bf x}) = \displaystyle 
\int_0^1 Df({\bf x}^\ast+s({\bf x} - {\bf x}^\ast)) ds ({\bf x} - {\bf x}^\ast),
\end{equation*}
where $Df({\bf x})$ is the Jacobian matrix of $f$ at ${\bf x}$.
The $t$-differential of $L({\bf x}(t))$ is thus represented by a quadratic function as follows:
\begin{equation}
\label{t-differential}
\displaystyle \frac{d}{dt}L({\bf x}) =
({\bf x} - {\bf x}^\ast)^T \displaystyle 
\int_0^1 \{ Df({\bf x}^\ast+s({\bf x} - {\bf x}^\ast))^T Y + 
  Y Df({\bf x}^\ast+s({\bf x} - {\bf x}^\ast)) \}ds\, ({\bf x} - {\bf x}^\ast).
\end{equation}

Let ${\bf z}={\bf x}^\ast+s({\bf x} - {\bf x}^\ast)$, and $A({\bf z})$ be a real symmetric matrix given by (\ref{matrix-neg-def}).
If $A({\bf z})$ is negative definite on $D_L$, then it is also negative definite at any point ${\bf z}$ on a segment connecting ${\bf x}^\ast$ and ${\bf x}$, since $\{{\bf x}^\ast+s({\bf x} - {\bf x}^\ast)\mid s\in [0,1]\}\subset D_L$.
Thus we get $\displaystyle \frac{d}{dt}L({\bf x}) < 0$ for any ${\bf x}\neq {\bf x^\ast}$.
\par
Obviously, $\displaystyle \frac{d}{dt}L({\bf x}) < 0$ for all ${\bf x} \not = {\bf x}^\ast$ and $\displaystyle \frac{d}{dt}L({\bf x}^\ast) = 0$.
Therefore, $L({\bf x})$ is a Lyapunov function on $D_L$.
\par
\bigskip
The form $dL/dt({\bf x}) = ({\bf x}-{\bf x}^\ast)^T \int_0^1 A({\bf x}^\ast + s({\bf x} - {\bf x}^\ast))ds  ({\bf x}-{\bf x}^\ast)$ and the strict negative definiteness of $A({\bf z})$ in $D_L$ imply that $dL/dt(\varphi(t,{\bf x}))_{t=0} < 0$ in $D_L\setminus \{{\bf x}^\ast\}$.
It implies that ${\bf x}^\ast$ is the unique equilibrium in $D_L$.
\end{proof}

In this theorem, we assume the hyperbolicity of ${\bf x}^\ast$ not in the rigorous sense, but just in the numerical sense.
Nevertheless, an additional assumption guarantees the hyperbolicity of ${\bf x}^\ast$ in the rigorous sense during the construction of Lyapunov functions, which is discussed in Section \ref{section-hyp-cont}.

\begin{definition}\rm
We shall call the domain $D_L$ where the Lyapunov function $L$ exists {\em a Lyapunov domain} (for $L$). 
\end{definition}

\begin{remark}\rm
In preceding works (e.g. \cite{Mat, ZCov}), $D_L$ is assumed to be convex.
Our theorem weakens the geometric condition of $D_L$.
\end{remark}

% New subsection
\subsection{Verification of Lyapunov domains with interval arithmetics}

%\subsubsection{Stage 1: A subdomain contains ${\bf x}^\ast$.}
\subsubsection{Stage 1: Negative definiteness of $A({\bf z})$.}
\label{section-stage1-cont}
Theorem \ref{thm-Lyapunov-cont} claims that $L({\bf x}) = ({\bf x}-{\bf x}^\ast)^T Y ({\bf x}-{\bf x}^\ast)$ is a Lyapunov function on a compact, star-shaped domain $D_L$ centered at an equilibrium ${\bf x}^\ast$, if the real symmetric matrix $A({\bf z})$ in (\ref{matrix-neg-def}) is negative definite. 
\par
If the domain $D_L$ itself or a piece of divided subdomains contains an equilibrium, we enclose it an interval vector $[{\bf x}^\ast]$ and verify the negative definiteness of the matrix enclosure $A([{\bf x}^\ast]) = \{A({\bf x})\mid {\bf x} \in [{\bf x}^\ast]\}$ on domains in our considerations.
An example of such a procedure is the following:

\begin{algorithm}
\label{alg-cont}
An algorithm for validating Lyapunov domains containing an equilibrium ${\bf x}^\ast$ based on the strict negative definiteness of $A({\bf z})$ consists of the following.
\begin{itemize}
\item[(1)]
	Calculate $A({\bf x}^\ast)$ with floating-point arithmetics. 
	If necessary, diagonalize $A({\bf x}^\ast)$ approximately after ensuring the symmetry of $A({\bf x}^\ast)$. Namely, calculate a matrix $X^\ast$ such that
          \begin{equation*}
          \Lambda^\ast = (X^\ast)^{-1}A({\bf x}^\ast) X^\ast.
          \end{equation*}
          Note that $X^\ast$ can be chosen as a real matrix since $A({\bf x}^\ast)$ is real and symmetric.
\item[(2)]
	Calculate the interval matrix $(X^\ast)^{-1}A([{\bf x}^\ast]) X^\ast$ with interval arithmetics. Let $[a^\ast]_{ij}$ the $(i,j)$-th entry of the matrix, which takes the interval value.
\item[(3)] 
	Apply the Gershgorin Circle Theorem to the interval matrix, namely, verify 
          \begin{equation*}
          [a^\ast]_{ii}+\displaystyle \sum_{j\neq i}|[a^\ast]_{ij}| < 0
          \end{equation*}
	with interval arithmetics for $i=1,\cdots,m$.
\end{itemize}
\end{algorithm}
This procedure works to prove the existence of Lyapunov functions as long as the above algorithm passes successfully, even if domains do not contain equilibria.
Note that there is a more effective verification method using Cholesky decomposition by Rump, which can be applied on INTLAB \cite{INTLAB}.
\par
\bigskip
We also note that the negative definiteness of $A({\bf z})$ does not involve the integral in (\ref{t-differential}).
This fact indicates that the verification of negative definiteness of $A({\bf z})$ on $D_L$ by Algorithm \ref{alg-cont} can be done independently for each subdomain $D_k$ of $D_L$ with the decomposition $D_L = \bigcup_{k=1}^K D_k$.
Namely,
\begin{corollary}
\label{cor-subdivision-cont}
Let $D_L$ be a compact, star-shaped domain centered at an equilibrium ${\bf x}^\ast$ of (\ref{autonomous}). 
Define a matrix $A({\bf z})$ by (\ref{matrix-neg-def}).
Let $D_L=\bigcup_{k=1}^K D_k$ be a decomposition of $D_L$ into subdomains.
We enclose each $D_k$ by an interval vector $[{\bf x}_k]$ and define an interval matrix $A([{\bf x}_k])$ in the similar way to the above arguments.
Assume that, for all $k=1,\cdots, K$, the interval matrix matrix $A([{\bf x}_k])$ is strictly negative definite
in the sense that the matrix $A({\bf z})$ is strictly negative definite for all ${\bf z}\in [{\bf x}_k]$.
Then $L({\bf x})$ is a Lyapunov function on $D_L$. 
Moreover, ${\bf x}^\ast$ is the unique equilibrium in $D_L$.
\end{corollary}
The proof can be done by the same arguments as the first part of Theorem \ref{thm-Lyapunov-cont} for ${\bf x}\in D_k$ for each $k=1,\cdots, K$.

\subsubsection{Stage 2: The case where a subdomain does not contain ${\bf x}^\ast$.}
\label{section-stage2-cont}
Even if the strict negative definiteness of $A({\bf x})$ violates, if a divided subdomain does not contain equilibria, there is another way to verify $dL/dt < 0$ on it, namely, verifying
\begin{equation*}
\displaystyle \frac{d}{dt}L({\bf x})\mid_{t=0} =
f({\bf x})^T Y ({\bf x}-{\bf x}^\ast) + 
({\bf x}-{\bf x}^\ast)^T Y f({\bf x})
\end{equation*}
being negative directly by interval arithmetics. 
\begin{corollary}
Let $D_L\subset \mathbb{R}^n$ be a compact and star-shaped set and ${\bf x}^\ast$ be an equilibrium of (\ref{autonomous}). Assume that the vector
\begin{equation*}
f({\bf x})^T Y ({\bf x}-{\bf x}^\ast) + ({\bf x}-{\bf x}^\ast)^T Y f({\bf x})
\end{equation*}
is strictly negative for all ${\bf x}\in D_L$. 
Then $L({\bf x})$ given by (\ref{Lyapunov-flow}) is a Lyapunov function on $D_L$.
\end{corollary}
In practical computations, for an interval vector $[{\bf x}]$ containing the subdomain, we calculate $\displaystyle \frac{d}{dt}L([{\bf x}])$ with interval arithmetics and verify if it is negative. 
We discuss the effectiveness of this procedure in Section \ref{section-example-cont}.

\bigskip
If either of the above procedures (Stage 1 and 2) passes in $D_L$, it turns out that the function $L({\bf x})$ is a Lyapunov function in $D_L$. 
In other words, $D_L$ is a Lyapunov domain.

% New subsection
\subsection{${\bf m}$-Lyapunov functions}
\label{section-m-Lyapunov-cont}
A typical choice of the matrix $Y$ transform the original coordinate to the new one such that (\ref{autonomous}) is approximately described by
\begin{equation*}
\frac{d}{dt}y_i = \lambda_i y_i + o(|{\bf y}|),\quad i=1,\cdots, n, 
\end{equation*}
which gives a Lyapunov function of the following form in a small neighborhood of the equilibrium ${\bf y}^\ast = (y_1^\ast, \cdots, y_n^\ast)$ (cf. Section \ref{section-hyp-cont}):
\begin{equation}
\label{Lyapunov-Mat}
L({\bf y}) = -\sum_{i=1}^n {\rm sgn}({\rm Re}(\lambda_i)) (y_i-y_i^\ast)^2.
\end{equation}
Note that all coefficients in (\ref{Lyapunov-Mat}) have modulus $1$.
It is thus natural to consider the similar quadratic functions {\em with arbitrary coefficients},
which leads to the idea of ${\bf m}$-Lyapunov functions; the main discussions in this subsection.

The coefficients in (\ref{Lyapunov-Mat}) is in fact determined by the matrix $I^\ast$ in the definition of $Y$.
Now we consider the following diagonal matrix $M^\ast = {\rm diag}(i_1,i_2,\cdots,i_m)$ , where
\begin{equation*}
i_j = \left\{ \begin{array}{c} 
                m_j, \quad \mbox{if} \quad \mbox{Re}(\lambda_j) < 0, \\
                -m_j, \quad \mbox{if} \quad \mbox{Re}(\lambda_j) > 0 
                \end{array}\right.
\end{equation*}
instead of the diagonal matrix $I^\ast$. Here ${\bf m} = \{m_j\}_{j=1}^n$ denotes a sequence of given positive numbers.
Define $L_{M^\ast}({\bf x})$ as follows: 
\begin{equation}
\label{Lyapunov-M-cont}
L_{M^\ast}({\bf x}) := ({\bf x} - {\bf x}^\ast)^T Y_{M^\ast}({\bf x} - {\bf x}^\ast),\qquad Y_{M^\ast} = {\rm Re}(X^{-H}M^\ast X^{-1}), 
\end{equation}
which is nothing but $L$ replacing $I^\ast$ by $M^\ast$. 
We can then prove negative definiteness of the matrix $A$ corresponding to the differential $(dL_{M^\ast}/dt)({\bf x})$ at ${\bf x}={\bf x}^\ast$ by similar arguments to Theorem \ref{thm-Lyapunov-cont}.
Consequently, we obtain the following corollary.

\begin{corollary}
\label{cor-m-Lyap-cont}
Let $D_L$ be a compact, star-shaped domain containing an equilibrium ${\bf x}^\ast$ of (\ref{autonomous}). 
Define a matrix $A_{M^\ast}({\bf z})$ by $A$ in (\ref{matrix-neg-def}) replacing $I^\ast$ by $M^\ast$ for a given sequence of positive numbers ${\bf m}$.
Assume that the replaced matrix $A_{M^\ast}({\bf z})$ is strictly negative definite for all ${\bf z}\in D_L$. 
Then $L_{M^\ast}({\bf x})$ given in (\ref{Lyapunov-M-cont}) is a Lyapunov function on $D_L$. 
Moreover, ${\bf x}^\ast$ is the unique equilibrium in $D_L$.
\end{corollary}

We shall call the function $L_{M^\ast}$ satisfying assumptions in Corollary \ref{cor-m-Lyap-cont} an {\em ${\bf m}$-Lyapunov function}.
Since positive numbers ${\bf m}=\{m_j\}_{j=1}^n$ can be chosen arbitrarily, we can arrange these distributions, depending on situations.
In particular, we can adjust the enclosure of stable and unstable directions in a neighborhood of a saddle ${\bf x}^\ast$ by arranging ${\bf m}$ as follows.
\begin{itemize}
\item If we choose the unstable cone $L_{M^\ast}^{-1}(-\infty,0)\cap D_L$ of ${\bf x}^\ast$ as sharp as possible, we choose positive numbers $m_j$ corresponding to $\mbox{Re}(\lambda_j) < 0$ as large as possible, compared with those corresponding to $\mbox{Re}(\lambda_j) > 0$.
\item If we choose the stable cone $L_{M^\ast}^{-1}(0,\infty)\cap D_L$ of ${\bf x}^\ast$ as sharp as possible, we choose positive numbers $m_j$ corresponding to $\mbox{Re}(\lambda_j) > 0$ as large as possible, compared with those corresponding to $\mbox{Re}(\lambda_j) < 0$.
\end{itemize}
Note that the geometry of $D_L$ depends on the choice of positive numbers $\{m_j\}_{j=1}^n$, which is discussed in Section \ref{section-example-cont}.

% New section
\section{Several aspects of Lyapunov domains}
\label{section-aspect}
We have seen that the strict negative definiteness of $A({\bf z})$ in (\ref{matrix-neg-def}) intrinsically provides us with Lyapunov domains containing unique equilibria.
Moreover, our Lyapunov functions are quadratic. 
Lyapunov functions with explicit forms and domains of definition help us with the study of asymptotic behavior around equilibria {\em in given domains}. 
At least, a Lyapunov function $L$ admits the following geometric and algebraic aspects.
Let $D_L$ be a Lyapunov domain.

\subsection{Lyapunov functions around hyperbolic equilibria}

In Theorem \ref{thm-Lyapunov-cont}, we focused on the strict negative definiteness of $A({\bf z})$ in a given domain containing ${\bf x}^\ast$.
The negative definiteness of $A({\bf z})$ reflects the hyperbolicity of $x^\ast$. 
Indeed, for a {\em rigorous} hyperbolic equilibrium ${\bf x}^\ast$, the matrix $A({\bf x}^\ast)$ is strictly negative definite.

\begin{proposition}
\label{prop-hyp-cont}
Let ${\bf x}^\ast$ be a hyperbolic equilibrium for (\ref{autonomous}). 
Then there is a neighborhood of $U$ of ${\bf x}^\ast$ such that $A({\bf z})$ given by (\ref{matrix-neg-def}) is negative definite for all points in $U$. 
Moreover, $L({\bf x})$ is a Lyapunov function on $U$.
\end{proposition}

\begin{proof}
It is sufficient to check the negative definiteness of the Hermitian matrix
\begin{equation*}
A^\ast = (Df({\bf x}^\ast))^H \hat{Y} + \hat{Y} Df({\bf x}^\ast)
\end{equation*}
instead of $A({\bf x}^\ast)$ in order to prove the negative definiteness of $A({\bf x}^\ast)$, which follows from the fact that ${\bf z}^T \hat Y {\bf z} = {\bf z}^T \mbox{Re}(\hat Y) {\bf z}$ for any real vector ${\bf z}$.
\par
By the definition of the Hermitian matrix $\hat{Y}$, we get
\begin{align*}
A^\ast &= (Df^\ast)^H X^{-H} I^\ast X^{-1} + X^{-H} I^\ast X^{-1}Df^\ast \\
    &= X^{-H} \Lambda^H I^\ast X^{-1} + X^{-H} I^\ast \Lambda X^{-1}\\
    &= X^{-H}(2\,\mbox{Re}(\Lambda) I^\ast)X^{-1}\\
    &= -2\, X^{-H}|\mbox{Re}(\Lambda)| X^{-1},
\end{align*}
where $|\mbox{Re}(\Lambda)|$ is the matrix whose entry is the absolute value of the corresponding entry of $\mbox{Re}(\Lambda)$.
This yields that $A^\ast$ is negative definite Hermitian matrix. 
Therefore, in a sufficiently small neighborhood of $x^\ast$, say a ball $\{{\bf x}\in \mathbb{R}^n \mid |{\bf x} - {\bf x}^\ast| < \epsilon\}$ for sufficiently small $\epsilon > 0$, we obtain
\begin{equation*}
\displaystyle \frac{d}{dt}L({\bf x}) < 0, \qquad \forall {\bf x}\in \{{\bf x}\in \mathbb{R}^n \mid 0< |{\bf x} - {\bf x}^\ast| < \epsilon\}
\end{equation*}
following the proof of Theorem \ref{thm-Lyapunov-cont}, which implies that $L({\bf x})$ is a Lyapunov function in such a neighborhood since $\displaystyle \frac{d}{dt}L({\bf x^\ast})=0$ also holds. 
\end{proof}

This proposition indicates that, like Stable Manifold Theorem, {\em all hyperbolic equilibria locally admit Lyapunov functions of the form (\ref{Lyapunov-flow})}.

\subsection{$L$ determines the eigenstructure of $Y$.}

The symmetric matrix $Y$ is typically determined by the matrix consisting of eigenvectors at a point ${\bf x}^\ast$: an equilibrium for example.
If ${\bf x}^\ast$ is hyperbolic and $Y$ is constructed by eigenvectors at ${\bf x}^\ast$, then $Y$ is nonsingular by definition.
In many practical situations, by contrast, $Y$ is determined only numerically, or constructed by a matrix $X$ close to the eigenmatrix, which imply that there may be a case that $Y$ is singular.
Conversely, if $L({\bf x})=({\bf x}-{\bf x}^\ast)^TY({\bf x}-{\bf x}^\ast)$ is a Lyapunov function on $D_L$ for {\em some} real symmetric matrix $Y$, then $L$ determines the algebraic structure of $Y$, which is stated as follows.

Recall that the {\em stable manifold} of a hyperbolic equilibrium ${\bf x}^\ast$ is the set 
\begin{equation*}
W^s({\bf x}^\ast) = \{{\bf x}\in \mathbb{R}^n \mid \varphi(t,{\bf x})\to {\bf x}^\ast \text{ as }t\to +\infty\}.
\end{equation*}
Similarly, the {\em unstable manifold} of ${\bf x}$ is the set 
\begin{equation*}
W^u({\bf x}^\ast) = \{{\bf x}\in \mathbb{R}^n \mid \varphi(t,{\bf x})\to {\bf x}^\ast \text{ as }t\to -\infty\}.
\end{equation*}
Without the loss of generality, we may assume that ${\bf x}^\ast$ is the origin in $\mathbb{R}^n$ by translating ${\bf x}^\ast$ to the origin.

\begin{proposition}
\label{prop-Y-cont}
Consider a functional $L({\bf x}) = {\bf x}^TY{\bf x}$ for some real symmetric matrix $Y$.
Assume that the origin ${\bf 0}$ is a hyperbolic equilibrium of (\ref{autonomous}) such that $Df({\bf 0})$ has $u$ eigenvalues of positive real part and $s = n-u$ eigenvalues of negative real part.
Assume that, for a compact star-shaped domain $D_L$ with ${\bf 0}\in {\rm int}D_L$, the following inequality holds:
\begin{equation*}
\frac{dL}{dt}(\varphi(t,{\bf x}))_{t=0} < 0\quad \forall {\bf x}\in D_L\setminus \{{\bf 0}\}.
\end{equation*}
Then $Y$ is non-singular and has $u$ negative eigenvalues and $s = n-u$ positive eigenvalues.
\end{proposition}

\begin{proof}
Suppose that $Y$ has a null eigenvector ${\bf x}\not = {\bf 0}$. 
Since the norm $\|{\bf x}\|$ can be chosen arbitrarily, by taking $\|{\bf x}\|$ sufficiently small, we may assume that the vector ${\bf x}$ is identified with a point ${\bf x}\in D_L$.
Then
\begin{equation*}
\frac{dL}{dt}(\varphi(t,{\bf x}))_{t=0} = f({\bf x})^T Y{\bf x} + {\bf x}^T Y f({\bf x}) = 0
\end{equation*}
since $Y$ is real and symmetric, which contradicts $(dL/dt)(\varphi(t,{\bf x}))_{t=0} < 0$.
The matrix $Y$ is thus nonsingular.
\par
Next, choose ${\bf x}$ on $W^s({\bf 0})\setminus \{{\bf 0}\}$, which can be achieved by choosing ${\bf x}$ sufficiently close to ${\bf 0}$, if necessary, by the Stable Manifold Theorem with respect to the origin.
We then have
\begin{equation*}
\frac{dL}{dt}(\varphi(t,{\bf x}))_{t=0} < 0\quad (\forall t > 0)\quad \text{ and }\quad \lim_{t\to +\infty} L(\varphi(t,{\bf x})=0,
\end{equation*}
which implies $L({\bf x}) > 0$.
Similarly, if ${\bf x}$ is on $W^u({\bf 0})\setminus \{{\bf 0}\}$, then $L({\bf x}) < 0$.
Arguments of the Rayleigh quotient thus yield that $Y$ has, at least, one positive eigenvalue and one negative eigenvalue.
\par
Let $\{{\bf u}_1, \cdots, {\bf u}_{u'}\}$ be eigenvectors of $Y$ associated with negative eigenvalues, and $\{{\bf v}_1, \cdots, {\bf v}_{s'}\}$ be eigenvectors of $Y$ associated with positive eigenvalues.
We choose these vectors so that any two vectors in $\{{\bf u}_1, \cdots, {\bf u}_{u'}\}$ and $\{{\bf v}_1, \cdots, {\bf v}_{s'}\}$ are orthogonal to each other.
Regularity of $Y$ implies that $u'+s'=n$.
Our claim here is $u'=u$ and $s'=s$.
\par
\bigskip
Let $\{{\bf u}_1^\ast, \cdots, {\bf u}_{u}^\ast\}$ be eigenvectors of $Df({\bf x}^\ast)$ associated with eigenvalues whose real parts are positive, and $\{{\bf v}_1^\ast, \cdots, {\bf v}_s^\ast\}$ be eigenvectors of $Df({\bf x}^\ast)$ associated with eigenvalues whose eigenvalues are negative.
Now suppose that $u'<u$. 
There are a nontrivial collection of coefficients $(\alpha_1, \cdots, \alpha_u)$ such that the vector ${\bf u}^\ast := \sum_{j=1}^u \alpha_j {\bf u}^\ast_j$ is orthogonal to ${\bf u}_j$ for all $j=1,\cdots, u'$.
This implies that ${\bf u}^\ast \in {\rm Span}\{{\bf v}_1, \cdots, {\bf v}_{s'}\}$. 
Therefore $L({\bf u}^\ast) \equiv l^\ast > 0$ holds from the property of $L$ as the quadratic form and the fact that eigenvalues associated with ${\bf v}_j$ are positive.
\par
We choose a local coordinate $(x_1,\cdots, x_u, y_1,\cdots, y_s)$ so that any point ${\bf x}$ close to the origin is represented as
\begin{equation*}
{\bf x} = \sum_{j=1}^u x_j {\bf u}_j^\ast + \sum_{j=1}^s y_j {\bf v}_j^\ast.
\end{equation*}
We identify a point ${\bf x}$ with its coordinate representation $(x_1,\cdots, x_u, y_1,\cdots, y_s)$.
By the Stable Manifold Theorem (e.g. \cite{Rob}), there are a sufficiently small $\epsilon > 0$ and a smooth map $\Psi : B_u({\bf 0},\epsilon)\to \mathbb{R}^s$ such that the unstable manifold $W^u({\bf 0})$ is represented by the graph of $\Psi$ :
\begin{align*}
&(y_1,\cdots, y_s) = \Psi(x_1,\cdots, x_u),\\
&(x_1,\cdots, x_u, y_1,\cdots, y_s) \in W^u(0) \text{ with } (x_1,\cdots, x_u)\in B_u({\bf 0},\epsilon),
\end{align*}
and that $\|\Psi(x_1,\cdots, x_u)\| = O(\|(x_1,\cdots, x_u)\|^2)$ as $(x_1,\cdots, x_u)\to (0,\cdots, 0)$, where $B_u({\bf x},\epsilon)$ denotes the $u$-dimensional closed ball with the radius $\epsilon$ centered at a point ${\bf }x$.
\par
Let ${\bf z}^\ast = \delta {\bf u}^\ast + {\bf v}^\ast$ be a vector with a positive $\delta >0$.
If $\delta > 0$ is sufficiently small, we can uniquely determine the vector ${\bf v}^\ast$ so that ${\bf z}^\ast \in W^u({\bf 0})$.
Let ${\bf v}^\ast = \Psi (\delta \alpha_1, \cdots, \delta \alpha_u)\equiv \Psi (\delta {\bf u}^\ast)$ be such a vector.
Then 
\begin{align*}
({\bf z}^\ast)^T Y {\bf z}^\ast &= \delta^2 ({\bf u}^\ast)^T Y {\bf u}^\ast + \delta ({\bf u}^\ast)^T Y {\bf v}^\ast + \delta ({\bf v}^\ast)^T Y {\bf u}^\ast + ({\bf v}^\ast)^T Y {\bf v}^\ast\\
&= \delta^2 ({\bf u}^\ast)^T Y {\bf u}^\ast + \delta ({\bf u}^\ast)^T Y \Psi (\delta {\bf u}^\ast) + \delta \Psi (\delta {\bf u}^\ast)^T Y {\bf u}^\ast + \Psi (\delta {\bf u}^\ast)^T Y \Psi (\delta {\bf u}^\ast)\\
 &= \delta^2 (l^\ast + O(\delta)) > 0
\end{align*}
if $\delta > 0$ is sufficiently small; i.e., ${\bf z}^\ast$ is sufficiently close to the origin, which follows from the fact $\|\Psi(x_1,\cdots, x_u)\| = O(\|(x_1,\cdots, x_u)\|^2)$ as $(x_1,\cdots, x_u)\to (0,\cdots, 0)$.
This contradicts the fact that $L({\bf z}^\ast) < 0$, since ${\bf z}^\ast \in W^u({\bf 0})$.
\par
Similar argument yields $s'\geq s$, which concludes that $u'=u$ and $s'=s$.
\end{proof}

% New subsection
\subsection{Validation of Lyapunov functions and hyperbolicity of equilibria}
\label{section-hyp-cont}
Now we discuss the negative definiteness of $A({\bf z})$ from the different viewpoint.
Let $\bar{\bf x}$ be a point such that $Df(\bar {\bf x})$ is numerically diagonalized:
\begin{equation*}
\Lambda = X^{-1}Df(\bar {\bf x})X,\quad \Lambda = {\rm diag}(\lambda_1, \cdots, \lambda_n),
\end{equation*}
and that ${\rm Re}(\lambda_i)\not = 0$ for all $i=1,\cdots, n$.
The nonsingular matrix $X$ transforms (\ref{autonomous}) into the following system:
\begin{equation}
\label{diagonal-cont}
\frac{d}{dt}{\bf y} = X^{-1}f({\bf x}^\ast + X{\bf y}),\quad {\bf x} = {\bf x}^\ast + X{\bf y},
\end{equation}
which is close to the diagonalized system in a small neighborhood of ${\bf x}^\ast$.
We then define a quadratic form $\tilde L$ by
\begin{equation}
\label{Lyapunov-simple-cont}
\tilde L({\bf y}) := {\bf y}I^\ast {\bf y},\quad I^\ast = {\rm diag}(i_1,\cdots, i_n),
\end{equation}
where $i_k$'s are given by (\ref{sign-change-cont}).
Let ${\bf x}$ be a point close to ${\bf x}^\ast$. 
%The line segment $\{ {\bf x}^\ast + s({\bf x} - {\bf x}^\ast)\mid s\in [0,1]\}$ corresponds one-to-one to
Direct calculations yield
\begin{equation*}
\displaystyle \frac{d}{dt}\tilde L({\bf y}) =
{\bf y}^T \displaystyle 
\int_0^1 \{  X^T Df({\bf x}^\ast+s({\bf x} - {\bf x}^\ast))^T X^{-T}I^\ast + 
  I^\ast X^{-1}Df({\bf x}^\ast+s({\bf x} - {\bf x}^\ast)) X\}ds\, {\bf y}.
\end{equation*}

In general, the following lemma holds from general linear algebra.
\begin{lemma}
\label{lem-matrix}
Let $A,B$ be $n\times n$ real matrices. Assume that $B$ is diagonal. Then $(AB)^T = BA^T$.
\end{lemma}
\begin{proof}
This lemma directly follows from $(AB)^T = B^TA^T$ and $B^T=B$.
\end{proof}
Lemma \ref{lem-matrix} implies
\begin{align*}
\displaystyle \frac{d}{dt}\tilde L({\bf y}) &=
{\bf y}^T  \displaystyle 
\int_0^1 \{ (X^T Df({\bf x}^\ast+s({\bf x} - {\bf x}^\ast))^T X^{-T})I^\ast + 
  I^\ast X^{-1}Df({\bf x}^\ast+s({\bf x} - {\bf x}^\ast))X \}ds\, {\bf y}\\
&=
{\bf y}^T  \displaystyle 
\int_0^1 \{ (X^{-1} Df({\bf x}^\ast+s({\bf x} - {\bf x}^\ast)) X)^T I^\ast + 
  I^\ast X^{-1}Df({\bf x}^\ast+s({\bf x} - {\bf x}^\ast))X \}ds\, {\bf y}
\end{align*}
Since $X^{-1}Df({\bf x}^\ast)X$ is diagonalizable, The matrix-valued sets
\begin{equation*}
\{X^{-1} Df({\bf x}^\ast+s({\bf x} - {\bf x}^\ast)) X \mid s\in [0,1]\}
\end{equation*}
can be regarded as a short segment with an endpoint $\Lambda = \Lambda^T$ in the set of $n\times n$ matrices.
The Gershgorin Circle Theorem can be applied to the persistence of $\{{\rm sgn} ({\rm Re}(\lambda_i))\}_{i=1}^n$ for matrices on the sets.
Combining this observation and arguments in the proof of Theorem \ref{thm-Lyapunov-cont}, we obtain the following result. 

\begin{theorem}[cf. \cite{Mat}]
\label{thm-hyp-cont}
Let ${\bf x}^\ast$ be an equilibrium for (\ref{autonomous}), and $D_L$ be a compact, star-shaped neighborhood of ${\bf x}^\ast$ in $\mathbb{R}^n$. Define a matrix $\tilde A({\bf z})$ by
\begin{equation*}
\tilde A({\bf z}) = X^{-1} Df({\bf z}) X \equiv \Lambda + V({\bf z}), 
\end{equation*}
where $\Lambda = {\rm diag}(\lambda_1,\cdots, \lambda_n)$ given by $\Lambda = X^{-1}Df(\bar {\bf x})X$ at a point $\bar {\bf x}\in D_L$ and ${\rm Re}(\lambda_i) \not = 0$ for all $i=1,\cdots, n$.
Assume that
\begin{equation}
\label{neg-def-strong-cont}
|{\rm Re}(\lambda_i)| > \sum_{j=1}^n |V({\bf z})_{ij}|\quad \text{ and }\quad |{\rm Re}(\lambda_i)| > \sum_{j=1}^n |V({\bf z})_{ji}|
\end{equation}
holds for all ${\bf z}\in D_L$ and $i=1,\cdots, n$.

Then ${\bf x}^\ast$ is hyperbolic. Moreover, the functional $\tilde L$ given by (\ref{Lyapunov-simple-cont}) is a Lyapunov function on $D_L$.
Finally, ${\bf x}^\ast$ is the unique equilibrium in $D_L$.
\end{theorem}

\begin{proof}
Gershgorin's theorem and the assumption (\ref{neg-def-strong-cont}) imply that the sign of ${\rm Re}(\lambda_i) = {\rm Re}(\lambda_i({\bf z}))$ is identical in $D_L$. This yields that ${\rm sgn}({\rm Re}(\lambda_i({\bf x}^\ast))) = {\rm sgn}({\rm Re}(\lambda_i))$ for all $i$ and hence ${\bf x}^\ast$ is hyperbolic.

Applying Gershgorin's theorem again to the matrix $\tilde A({\bf z})^T I^\ast + I^\ast \tilde A({\bf z})$, one sees that all eigenvalues have negative real parts for all ${\bf z}\in D_L$. Since $Df({\bf z})$ is a real matrix, the matrix $\tilde A({\bf z})^T I^\ast + I^\ast \tilde A({\bf z})$ is strictly negative definite in $D_L$. Thus Theorem \ref{thm-Lyapunov-cont} can be applied to proving that $\tilde L$ is a Lyapunov function in $D_L$ and ${\bf x}^\ast$ is the unique equilibrium in $D_L$.
\end{proof}

\begin{remark}\rm
The Lyapunov function $\tilde L$ is exactly the form discussed in \cite{Mat}.
Note that forms of Lyapunov functions $L$ and $\tilde L$ are exactly the same, while the condition in Theorem \ref{thm-hyp-cont} is a little stronger than Theorem \ref{thm-Lyapunov-cont}.
Indeed, Theorem \ref{thm-Lyapunov-cont} guarantees (i) the existence of Lyapunov functions and (ii) local uniqueness of equilibria, while Theorem \ref{thm-hyp-cont} guarantees (i) the existence of Lyapunov functions, (ii) local uniqueness of equilibria and (iii) hyperbolcity of equilibria.
The difference of conditions comes from which the negative definiteness of $Df({\bf x})^T Y + YDf({\bf x})$ or $X^{-1}Df({\bf x})X$ is concerned.
\end{remark}

\subsection{Equivalence to cones: description of the stable and the unstable manifolds}
Asymptotic behavior of dynamical systems around hyperbolic equilibria deeply relates to descriptions of their stable and unstable manifolds. 
%Recall that the {\em stable manifold} of a hyperbolic equilibrium ${\bf x}^\ast$ is the set 
%\begin{equation*}
%W^s({\bf x}^\ast) = \{{\bf x}\in \mathbb{R}^n \mid \varphi(t,{\bf x})\to {\bf x}^\ast \text{ as }t\to +\infty\}.
%\end{equation*}
%Similarly, the {\em unstable manifold} of ${\bf x}$ is the set 
%\begin{equation*}
%W^u({\bf x}^\ast) = \{{\bf x}\in \mathbb{R}^n \mid \varphi(t,{\bf x})\to {\bf x}^\ast \text{ as }t\to -\infty\}.
%\end{equation*}
Since $L({\bf x}^\ast) = 0$ holds, by the definition of $W^s({\bf x})$ and $L$, it immediately holds that $W^s({\bf x}^\ast)\cap D_L\subset L^{-1}(0, +\infty)\cap D_L$.
Similarly, $W^u({\bf x}^\ast)\cap D_L\subset L^{-1}(-\infty, 0)\cap D_L$.

\bigskip
Now we refer to the concept of {\em cones} in, e.g., \cite{C, ZCov}.
In these references {\em a cone} is defined as a quadratic form $Q({\bf z}) = Q_1({\bf z}) - Q_2({\bf z})$ such that $Q_1$ and $Q_2$ are positive definite. 
A connection between $Q$ and stable and unstable manifolds of equilibria are well discussed in \cite{ZCov}; 
namely, under {\em the cone condition}:
\begin{equation}
\label{cone-cond-cont}
\frac{d}{dt}(Q(\varphi(t,{{\bf z}_1})) - Q(\varphi(t,{{\bf z}_2}))) > 0\text{ for all } {\bf z}_1, {\bf z}_2 \in N
\end{equation}
in a given contractible set with special property called {\em an $h$-set} $N$, the stable and unstable manifolds of an equilibrium in $N$ can be described by graphs of Lipschizian functions (see Lemma 27 in \cite{ZCov}). 
Although \cite{ZCov} only describes the essential idea in abstract settings, this property gives us descriptions of locally invariant manifolds in explicitly given domains.
Applications of cone conditions with computer assistance are discussed in \cite{KWZ2007, Mat2}.

In fact, the strict negative definiteness of $A({\bf z})$ in (\ref{matrix-neg-def}) is intrinsically the same as a sufficient condition of cone conditions stated in (\ref{cone-cond-cont}). 
This fact indicates that a condition with respect to $A({\bf z})$ enables us to study asymptotic behavior around (hyperbolic) equilibria from the viewpoint of both Lyapunov functions and cones.

\bigskip
Now consider the case that the negative definiteness of $A({\bf z})$ does not hold in a domain $D$. 
In such a case, there is a non-trivial problem if the existence of a Lyapunov function in $D$ implies that of a cone $Q$ in satisfying the cone condition in $D$, and vice versa.
There may be a case that a Lyapunov domain $D$ validated in Stage 2 (Section \ref{section-stage2-cont}) does not satisfy the cone condition in general.
In other words, there is a domain where we cannot consider Lyapunov functions and cones in the same terms.
If it is the case, we then have to develop and consider the theory of asymptotic behavior in $D$ in terms of cones and Lyapunov functions independently.

\begin{remark}\rm
Prof. P. Zgliczy\'{n}ski commented to the first author \cite{MZ} that Lyapunov functions $L({\bf x})$ focus on the behavior of {\em a point} along solution orbits, and that cone conditions focus on {\em the difference between two points} along solution orbits.
These differences can be seen in Definition \ref{dfn-Lyapunov-typical} and (\ref{cone-cond-cont}).
\end{remark}

\begin{remark}[Cones and ${\bf m}$-Lyapunov functions]\rm
The former idea of ${\bf m}$-Lyapunov functions (Section \ref{section-m-Lyapunov-cont}) is introduced in \cite{Mat2}, where it is said {\em $m$-cones}, so that the enclosure of the stable and the unstable manifolds of equilibria are sharpened more and more. The same idea is introduced by Capi\'{n}ski \cite{C} as well as Zgliczy\'{n}ski \cite{ZCov} before \cite{Mat2} in terms of cones with arbitrary choice.
In \cite{Mat2}, the concept of $m$-cones is introduced with identical ratios, i.e., $m_j$'s are $1$ (resp. $m$) for stable eigendirections and $m$ (resp. $1$) for unstable eigendirections, where $m > 1$ is a given positive number. 
Our ${\bf m}$-Lyapunov functions generalize the choice of $m_j$'s.

One of the powerful properties of our procedure of Lyapunov functions as well as ${\bf m}$-Lyapunov functions is that enclosures of locally invariant manifolds around invariant sets can be described by quadratic functions.
We comment that the extension of regions where ${\bf m}$-Lyapunov functions are validated is applied to validations of the stable and the unstable manifolds of equilibria in singular perturbation problems \cite{Mat2} in terms of $m$-cones.
\end{remark}

% New subsubsection
\subsection{Lyapunov tracing: transversality of flows and re-parametrization of trajectories}
By definition, $dL/dt < 0$ always holds off equilibria, which implies that all solution orbits intersect transversally with level sets $L^{-1}(s)\cap D_L$ for all $s\in \mathbb{R}$, if exist. 
If an equilibrium ${\bf x}^\ast$ is asymptotically stable, then the level set $L^{-1}(s)$ is either an ellipse or ellipsoid, in which case only the interior of ellipsoids in $D_L$ can be positively invariant for (\ref{autonomous}).
If an equilibrium ${\bf x}^\ast$ is a saddle, the level set $L^{-1}(s)$ is a hyperbolic surface except $L^{-1}(0)$; the surface of a cone.
%Such explicit description 
The above fact helps us with constructing crossing sections or isolating blocks in the Conley index theory of flows (e.g. \cite{Con}) explicitly.
Moreover, along a solution orbits with a given initial data, the time variable $t$ smoothly corresponds {\em one-to-one} to the value $s$ of Lyapunov functions as long as the orbit stays in $D_L$, which leads to a local {\em re-parametrization} of trajectories $\varphi(t,{\bf x}) \mapsto \tilde \varphi(L,{\bf x})$. 
More precisely, the following theorem holds.
\begin{theorem}[Retraction via Lyapunov tracing]
\label{thm-LT}
Let $D_L$ be a compact star-shaped Lyapunov domain containing an equilibrium ${\bf x}^\ast$ associated with a Lyapunov function $L$.
Assume that the boundary of the stable cone $B_s \equiv D_L\cap L^{-1}(0,\infty)$ consists of $b_s := L^{-1}(0)\cap D_L$ and a collection of hypersurfaces $\{S_i\}$ which are the immediate entrance in the sense that, for any ${\bf x}\in S_i$, there is a positive number $\epsilon_{\bf x} > 0$ such that
\begin{equation*}
\varphi((0,\epsilon_{\bf x}),{\bf x})\subset {\rm int} (D_L\cap L^{-1}(0,\infty)),\quad \varphi((-\epsilon_{\bf x},0),{\bf x})\cap {\rm int} (D_L\cap L^{-1}(0,\infty)) = \emptyset.
\end{equation*}
See Fig. \ref{fig-retract}.
\par
Define the map $R : B_s \to b_s$ by
\begin{equation*}
R({\bf x}) := \varphi(\tau({\bf x}), {\bf x}),\quad  \tau({\bf x}) := \sup\{t\geq 0\mid \varphi([0,t],{\bf x})\subset B_s\}.
\end{equation*}
Then $R$ is continuous. 
Moreover, $b_s$ is the strong deformation retract of $B_s$.
\end{theorem}
Note that we say a subset $A$ of a compact set $X$ is a {\em strong deformation retract} if there exists a continuous map $R : X\times [0,1]\to X$ such that $R({\bf x},t) = {\bf x}$ for all ${\bf x}\in A$, $R({\bf x},1)\in A$ for all ${\bf x}\in X$ and $R({\bf x},0) = {\bf x}$ for all ${\bf x}\in X$.

\begin{figure}[htbp]\em
\begin{minipage}{1\hsize}
\centering
\includegraphics[width=6.0cm]{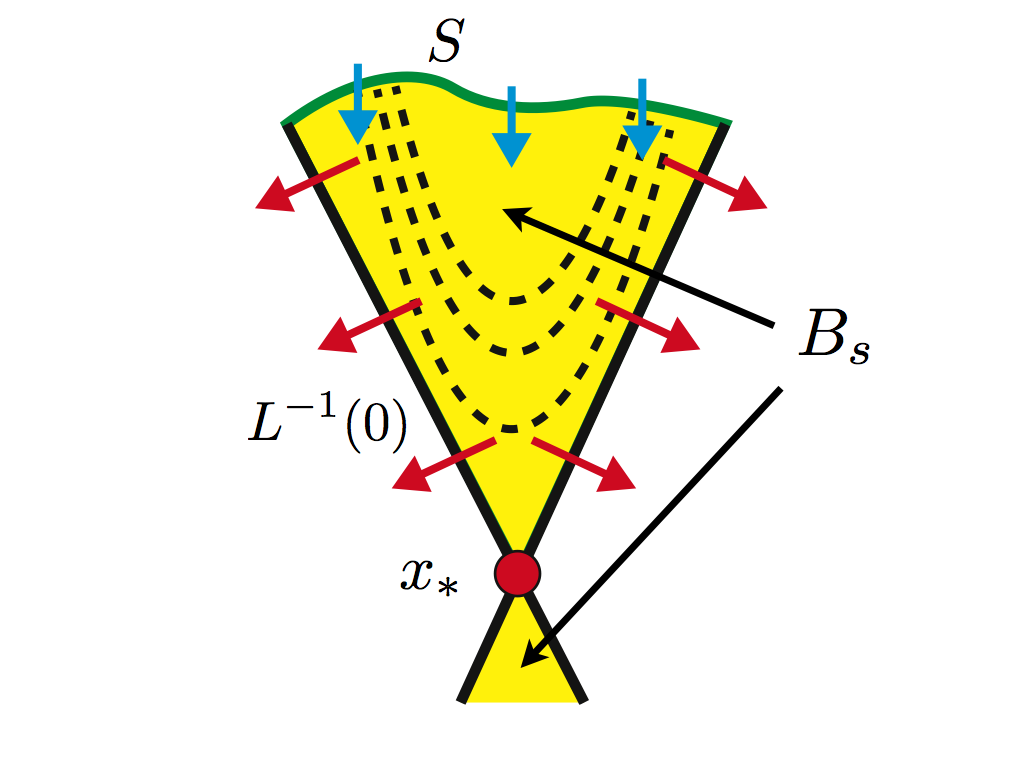}
\end{minipage}
\caption{The positive cone $B_s$ and dynamics inside it}
\label{fig-retract}
\begin{flushleft}
The yellow region shows the set $B_s \equiv D_L\cap L^{-1}(0,\infty)$ associated with a Lyapunov function $L$ in Theorem \ref{thm-LT}. 
Black lines denote the zero level set $L^{-1}(0)$ of the Lyapunov function $L$.
Dotted curves denote the family of positive level sets $\{L^{-1}(s)\cap D_L \mid s> 0\}$ of $L$.
The green curve $S$ shows a hypersurface being the immediate entrance of $B_S$.
In this setting, black lines which are also boundaries of the yellow region is $b_s$.
\end{flushleft}
\end{figure}

\begin{proof}
First note that $b_s\setminus \{{\bf x}^\ast\}$ is the immediate exit of $B_s$ in the sense that, for any ${\bf x}\in b_s\setminus \{{\bf x}^\ast\}$, there is a positive number $\epsilon_{\bf x} > 0$ such that
\begin{equation*}
\varphi((0,\epsilon_{\bf x}),{\bf x})\cap (D_L\cap L^{-1}(0,\infty)) = \emptyset,\quad \varphi((-\epsilon_{\bf x},0),{\bf x})\cap b_s\setminus \{{\bf x}^\ast\} = \emptyset.
\end{equation*}
Let ${\bf x}_0 \in B_s$ be an arbitrary point.
Then the solution orbit ${\bf x}(t)\equiv \varphi(t,{\bf x}_0)$ through ${\bf x}_0$ can be written by
\begin{equation*}
{\bf x}(t) = {\bf x}_0 + \int_0^t f({\bf x}(s))ds.
\end{equation*}
A Lyapunov function $L$ has the form $L({\bf x}) = ({\bf x}-{\bf x}^\ast)^T Y (x-{\bf x}^\ast)$ and it immediately follows that
\begin{equation*}
\frac{dL}{dt}({\bf x}(t))\mid_{t=0} = f({\bf x}_0)^T Y ({\bf x}_0 - {\bf x}^\ast) + ({\bf x}_0-{\bf x}^\ast)^T Y f({\bf x}_0).
\end{equation*}

Let $L_0 = L({\bf x}_0)$ and $L_t = L({\bf x}(t))$. 
Note that $t$ and $L_t = L({\bf x}(t))$ corresponds homeomorphically along the solution orbit $\{{\bf x}(t)\}$.
The solution ${\bf x}(t)\equiv {\bf x}(L)$ can be thus formally re-written by
\begin{equation}
\label{L-integral}
{\bf x}(L_t) = {\bf x}(L_0) + \int_{L_0}^{L_t} \frac{f({\bf x}(L))}{ f({\bf x}(L))^T Y ({\bf x}(L) - {\bf x}^\ast) + ({\bf x}(L)-{\bf x}^\ast)^T Y f({\bf x}(L))}dL.
\end{equation}
Since $B_s$ is a subset of the Lyapunov domain $D_L$, then $dL/dt \not = 0$ off ${\bf x}^\ast \in L^{-1}(0)$. 
The mapping $R({\bf x}_0)$ is thus given by
\begin{equation*}
R({\bf x}_0) = {\bf x}(0) = {\bf x}(L_0) + \int_{L_0}^{0} \frac{f({\bf x}(L))}{ f({\bf x}(L))^T Y ({\bf x}(L) - {\bf x}^\ast) + ({\bf x}(L)-{\bf x}^\ast)^T Y f({\bf x}(L))}dL.
\end{equation*}
We shall prove that this expression makes sense.
To this end, we divide the integral into two parts: $\int_{L_0}^\epsilon + \int_{\epsilon}^0$ for sufficiently small $\epsilon$.
\par
\bigskip
The first part $\int_{L_0}^\epsilon$ obviously makes sense and continuous with respect to ${\bf x}_0$, since $f$ is smooth and the denominator is bounded away from $0$ on the compact set $\overline{L^{-1}(\epsilon, \infty)}\cap D_L$ for any fixed $\epsilon > 0$. 
\par
Consider the second integral $\int_{\epsilon}^0$. 
Let $B({\bf x}^\ast, \epsilon)$ be the closed ball centered at ${\bf x}^\ast$ with the radius $\epsilon$.
If ${\bf x}_0\in B_s\setminus (B({\bf x}^\ast, \epsilon)\cup W^s({\bf x}^\ast))$, the integral $\int_{\epsilon}^0\cdots$ also makes sense since the denominator is bounded away from $0$.
Note that, thanks to the Mean Value Theorem, $f({\bf x})$ is close to $Df({\bf x}^\ast)({\bf x}-{\bf x}^\ast)$ and the functional $f({\bf x})^T Y ({\bf x} - {\bf x}^\ast) + ({\bf x} - {\bf x}^\ast)^T Y f({\bf x})$ is sufficiently close to
\begin{equation*}
({\bf x}-{\bf x}^\ast)^T (Df({\bf x}^\ast)^T Y + YDf({\bf x}^\ast))) ({\bf x}-{\bf x}^\ast)
\end{equation*}
for ${\bf x}\in B({\bf x}^\ast, \epsilon)$ if $\epsilon > 0$ is sufficiently small.
It is therefore sufficient to prove that the integral
\begin{equation}
\label{integral-r}
\int_{\epsilon}^{0} \frac{Df({\bf x}^\ast)({\bf x}(L)-{\bf x}^\ast)}{ ({\bf x}-{\bf x}^\ast)^T (Df({\bf x}^\ast)^T Y + YDf({\bf x}^\ast))) ({\bf x}-{\bf x}^\ast) }dL
\end{equation}
is bounded for proving boundedness of the integral $\int_\epsilon^0\cdots$.
\par
Since ${\bf x}\in B({\bf x}^\ast,\epsilon)$, the vector ${\bf x}-{\bf x}^\ast$ is written by ${\bf x}-{\bf x}^\ast = r v$, where $r = r_{\bf x} = \|{\bf x}-{\bf x}^\ast\|\leq \epsilon$ and $v = v_{\bf x} = ({\bf x}-{\bf x}^\ast) / \|{\bf x}-{\bf x}^\ast\|$.
The Lyapunov function $L({\bf x}) = ({\bf x}-{\bf x}^\ast)^T Y({\bf x}-{\bf x}^\ast)$ can be thus rewritten by $L(r, v) = r^2 v^T Y v$ and
\begin{equation*}
\frac{dL}{dr}({\bf x}(r,v)) = 2r v^T Y v.
\end{equation*}
The integral (\ref{integral-r}) thus admits the transformation of variables from $L$ to $r$ as follows:
\begin{align*}
&\int_{\epsilon}^{0} \frac{Df({\bf x}^\ast)({\bf x}(L)-{\bf x}^\ast)}{ ({\bf x}-{\bf x}^\ast)^T (Df({\bf x}^\ast)^T Y + YDf({\bf x}^\ast))) ({\bf x}-{\bf x}^\ast) }dL\\
	&\quad = \int_{\epsilon_{x}'}^0 \frac{r Df({\bf x}^\ast)v}{ r^2 v^T (Df({\bf x}^\ast)^T Y + YDf({\bf x}^\ast))) v }\cdot 2r v^T Yv dr\\
	&\quad = 2\int_{\epsilon_{x}'}^0 \frac{ v^T Yv \cdot Df({\bf x}^\ast)v}{ v^T (Df({\bf x}^\ast)^T Y + YDf({\bf x}^\ast))) v } dr,
\end{align*}
which is independent of $r$, where $\epsilon_{x}' = \|{\bf x}-{\bf x}^\ast\|$ with $L({\bf x}) = \epsilon$.
The integrand is obviously bounded and continuous with respect to $x$, and hence the integral (\ref{integral-r}) is bounded.
The integral 
\begin{equation*}
 \int_{\epsilon}^{0} \frac{f({\bf x}(L))}{ f({\bf x}(L))^T Y ({\bf x}(L) - {\bf x}^\ast) + ({\bf x}(L)-{\bf x}^\ast)^T Y f({\bf x}(L))}dL
\end{equation*}
is also bounded, which follows from the continuity of the integral.
Note that the integral (\ref{integral-r}) makes sense even for $r = 0$, i.e., ${\bf x} = {\bf x}^\ast$.
Consequently, the definition of $R({\bf x}_0)$ makes sense everywhere in $B_s$ and $R$ is continuous.
\par
\bigskip
The mapping $R({\bf x})$ determines a family of solution curves $c:[0,1]\times B_s \to B_s$ given by
\begin{equation*}
c(s;{\bf x}) = {\bf x}(L_0) + \int_{L_0}^{(1-s)L_0} \frac{f({\bf x}(L))}{ f({\bf x}(L))^T Y ({\bf x}(L) - {\bf x}^\ast) + ({\bf x}(L)-{\bf x}^\ast)^T Y f({\bf x}(L))}dL.
\end{equation*}
The preceding arguments imply that $c$ is continuous. 
It immediately holds that $c(0;{\bf x}) = {\bf x}\equiv {\bf x}(L_0)$ for all ${\bf x}\in B_s$.
Moreover, $c(1;{\bf x}) = R({\bf x})\in b_s$ for all ${\bf x}\in B_s$. 
Finally, $c(s;{\bf x}) = {\bf x}$ holds for all ${\bf x}\in b_s$ and $s\in [0,1]$, since ${\bf x} = R({\bf x})$ for ${\bf x}\in b_s$.
These arguments imply that $b_s$ is the strong deformation retract of $B_s$ and the proof is completed.
\end{proof}
This theorem has two aspects. 
Firstly, Lyapunov functions of the form (\ref{Lyapunov-flow}) give us an analytic proof which Lyapunov domains possess strong deformation retracts. 
The exit set being the strong deformation retract is a well-known fact for {\em isolating blocks} in the Conley index theory (e.g. \cite{Con, Smo}).
On the contrary, $B_s$ is not an isolating block since its boundary contains an equilibrium ${\bf x}^\ast$.
The property of $B_s$ and $b_s$ thus gives us several generalizations of \lq\lq exits" of compact sets from viewpoints of topology and dynamical systems.
\par
Secondly, the proof of this theorem indicates that solution orbits in $D_L$ can be re-parameterized by values of $L$ instead of $t$, even for (un)stable manifolds. 
In fact, representations of solutions in terms of the $L$-integral guarantees the boundedness of integrals {\em in bounded range}.
This fact is very useful, in particular, if we track trajectories near equilibria or along invariant manifolds by integrals.
We say this re-parameterization the {\em Lyapunov tracing}.
An application of the Lyapunov tracing can be seen in \cite{TMSTMO2016}; namely, validations of blow-up times of blow-up solutions.

% New section
\section{Lyapunov functions for discrete dynamical systems}
\label{section-discrete}
In this section, we consider discrete dynamical systems generated by continuously differentiable maps of the form
\begin{equation}
\label{discrete-system}
\displaystyle {\bf x}_{n+1} = \psi({\bf x}_n),\quad {\bf x}_0 \in \mathbb{R}^n,
\end{equation}
where $\psi : \mathbb{R}^n\to \mathbb{R}^n$ be a smooth map.
Assume that ${\bf x^\ast}$ is a hyperbolic fixed point of $\psi$, namely, all eigenvalues of the Jacobian matrix $D\psi({\bf x})$ of $\psi({\bf x})$ at ${\bf x}^\ast$ are away from the unit circle. 
We now introduce a Lyapunov function in the similar manner to continuous dynamical systems.

\begin{definition}\rm
\label{dfn-Lyapunov-map}
Let $U\subset \mathbb{R}^n$ be an open subset. A {\em Lyapunov function $L:\mathbb{R}^n\to \mathbb{R}$ on $U$} for the map $\psi$ is a continuous function satisfying the following conditions.
\begin{enumerate}
\item $L(\psi({\bf x}))\leq L({\bf x})$ holds for all ${\bf x}\in U$.
\item $L(\psi({\bf x}))= L({\bf x})$ implies $\psi({\bf x})= {\bf x}^\ast\in U$, where ${\bf x}^\ast$ is a fixed point of (\ref{discrete-system}).
\end{enumerate}
\end{definition}
Our aim here is to construct a Lyapunov function for $\psi$ defined in an explicitly given neighborhood of ${\bf x}^\ast$.

% New subsection
\subsection{Construction of quadratic functions}
Here we show a procedure for the construction of the Lyapunov function defined in a neighborhood of a fixed point ${\bf x}^\ast$.

\begin{itemize}
\item[1.] Let $A^\ast = D\psi({\bf x}^\ast)$ and diagonalize it; namely, 
\begin{equation*}
\Lambda = X^{-1}A^\ast X,
\end{equation*}
where $\Lambda = {\rm diag}(\lambda_1,\lambda_2,\cdots,\lambda_m)$ is the eigenvalue matrix of $A^\ast$ and $X$ is the nonsingular matrix given by corresponding eigenvectors. 
Note that these calculations are sufficient to be done with floating-point arithmetics.
Assume that, in the sense of floating-point arithmetics, $|\lambda_i|\not = 1$ for all $i=1,\cdots, n$.
\item[2.] 
	Let $I^\ast$ be the diagonal matrix $I^\ast = {\rm diag}(i_1,i_2,\cdots,i_m)$, where
\begin{equation}
\label{sign-change-map}
i_k = \left\{ \begin{array}{c} 
                1, \quad \mbox{if} \quad |\lambda_k| < 1, \\
                -1, \quad \mbox{if} \quad |\lambda_k| > 1. 
                \end{array}\right.
\end{equation}
Note that $|\lambda_k|\not =1$ holds by hyperbolicity of $x^\ast$.
\item[3.] Calculate the real symmetric matrix $Y$ as follows:
\begin{equation*}
\hat{Y} = X^{-H} I^\ast X^{-1}, \quad Y = \mbox{Re}( \hat{Y}),
\end{equation*}
where $X^{-H}$ denotes the inverse matrix of the Hermitian transpose $X^H$ of $X$, which is sufficient to be calculated by floating-point arithmetics.
\item[4.] Define a quadratic function $L({\bf x})$ by
\begin{equation}
\label{Lyapunov-map}
L({\bf x}) = ({\bf x}-{\bf x}^\ast)^T Y ({\bf x}-{\bf x}^\ast),
\end{equation}
which is a candidate of Lyapunov functions around ${\bf x}^\ast$. 
If we deal with $L$ with interval arithmetics, we replace $Y$ by $(Y+Y^T)/2$ or set $Y_{ji}=Y_{ij}$ so that the symmetry of $Y$ is guaranteed.
\end{itemize}

% New subsection
\subsection{Validity of $L({\bf x})$}
Here we find a sufficient condition such that the function $L({\bf x})$ introduced above is indeed a Lyapunov function in a given  neighborhood of ${\bf x^\ast}$. 
%We also prove that such a sufficient condition is actually satisfied in a sufficiently small neighborhood of ${\bf x^\ast}$.

Let $\{{\bf x}_n\}$ be a $\psi$-orbit: ${\bf x}_{n+1} = \psi({\bf x}_n)$. By using the Jacobian matrix  $D\psi({\bf x})$ of $\psi$ at ${\bf x}$, we obtain the following integral representation of the difference ${\bf x}_{n+1} - {\bf x}^\ast$:  
\begin{align*}
          {\bf x}_{n+1} - {\bf x}^\ast &= 
          \displaystyle \int_0^1 
           D\psi({\bf x}^\ast + s({\bf x_n}-{\bf x}^\ast)) ds ({\bf x}_n - {\bf x}^\ast)\\
           &\equiv A_I({\bf x}_n)({\bf x}_n - {\bf x}^\ast),
\end{align*}
where we used the following expression:
\begin{equation*}
A_I({\bf x}) = \displaystyle \int_0^1 D\psi({\bf x}^\ast + s({\bf x}-{\bf x}^\ast)) ds.
\end{equation*}

This expression implies that the condition $L({\bf x}_{n+1}) < L({\bf x}_n)$ is equivalent to the following inequality:
\begin{equation*}
{\bf x}_n^T (\, A_I({\bf x}_n)^T Y A_I({\bf x}_n) - Y \,) {\bf x}_n < 0.
\end{equation*}
This observation indicates that, if we can validate 
\begin{equation*}
A_I({\bf x})^T Y A_I({\bf x}) - Y
\end{equation*}
is strictly negative definite for all ${\bf x}\in D_L$, 
where $D_L$ is a star-shaped domain centered at ${\bf x}^\ast$, 
then we know that $L(\psi({\bf x}))-L({\bf x})$ is strictly negative for all ${\bf x} \in D_L\setminus \{{\bf x}^\ast\}$.
By assumption, ${\bf x}^\ast\in D_L$ is a hyperbolic fixed point of $\psi$ and $L(\psi({\bf x}^\ast)) = L({\bf x}^\ast) = 0$ holds.
Summarizing the above arguments, we obtain the following theorem:

\begin{theorem}
\label{thm-Lyapunov-map}
Let $D_L$ be a compact, star-shaped domain containing a fixed point ${\bf x}^\ast$ of $\psi$. 
Define a matrix $B({\bf x})$ by
\begin{equation}
\label{matrix-discrete}
B({\bf x}) := A_I({\bf x})^T Y A_I({\bf x}) - Y
\end{equation}
Assume that the matrix $B({\bf x})$ is strictly negative definite for all ${\bf x}\in D_L$. 
Then $L({\bf x})$ defined by (\ref{Lyapunov-map}) is a Lyapunov function on $D_L$ in the sense of Definition \ref{dfn-Lyapunov-map}. 
Moreover, ${\bf x}^\ast$ is the unique fixed point of $\psi$ in $D_L$.
\end{theorem}
\begin{definition}\rm
We shall call the domain $D_L$ in Theorem \ref{thm-Lyapunov-map} {\em a Lyapunov domain} (for $\psi$). 
\end{definition}

\begin{remark}[Lyapunov functions and cones for maps]\rm
As in the case of flows, the negative definiteness of $B({\bf x})$ in Theorem \ref{thm-Lyapunov-map} is equivalent to a sufficient condition of {\em cone conditions for maps} stated in \cite{ZCov}. 
Remark that {\em a cone} $Q$ (for $\psi$) is a quadratic form $Q({\bf z}) = Q_1({\bf z}) - Q_2({\bf z})$, where $Q_1$ and $Q_2$ are positive definite, and {\em the cone condition} for $\psi$ is 
\begin{equation}
\label{cone-cond-map}
Q(\psi({\bf z}_1) - \psi({\bf z}_2)) > Q({\bf z}_1- {\bf z}_2)\quad \text{ for all }{\bf z}_1, {\bf z}_2\in N \text{ with }\psi({\bf z}_1),\psi({\bf z}_2)\in N,
\end{equation}
where $N$ is an $h$-set (cf. Section \ref{section-aspect}). 
This fact can be compared with Section 3.3 in \cite{ZCov}.
Therefore, once we validate the negative definiteness of $B({\bf x})$ in $D_L$, we can study asymptotic behavior for $\psi$ in $D_L$ in terms of both Lyapunov functions $L({\bf x})$ and cones $Q({\bf z})$.
\end{remark}

% New subsection
\subsection{Verification of Lyapunov domains with interval arithmetics}

\subsubsection{Stage 1: Negative definiteness of $B({\bf x})$.}
\label{section-stage1-disc}
Theorem \ref{thm-Lyapunov-map} claims that $L({\bf x}) = ({\bf x}-{\bf x}^\ast)^T Y ({\bf x}-{\bf x}^\ast)$ is a Lyapunov function on a star-shaped domain $D_L$ centered at a fixed point ${\bf x}^\ast$, if the real symmetric matrix $B({\bf x})$ in (\ref{matrix-discrete}) is negative definite. 

Let $D_L$ be a star-shaped domain containing the fixed point $\{{\bf x}^\ast\}$.
Firstly we verify the strict negative definiteness of $B({\bf x})$ on $D_L$ directly.
In practical computations, we enclose $D_L$ by an interval vector $[D_L]$ and verify the negative definiteness of the interval matrix $B([D_L])$.
If it is succeeded, everything has done in this case.
\par
If not, we verify the negative definiteness of $B({\bf x})$ on $D_L$ after decomposing $D_L$ into subdomains: $D_L= \bigcup_{k=1}^{K} D_k$.
As verifications in the whole domain, we enclose each subdomain $D_k$ by an interval vector $[D_k]$ and try to verify the negative definiteness of $B([D_k])$ for all $k = 1,\cdots, K$.
The basic idea follows the way in the case of flows, but the detail is more complicated.
We revisit the point in Section \ref{section-example-discrete}.

\subsubsection{Stage 2: The case where $D_k$ does not contain ${\bf x}^\ast$.}
If there is a subdomain $D_k$ which does not contain ${\bf x}^\ast$, we directly verify if $L(\psi({\bf x})) < L({\bf x})$ holds on $D_k$ by interval arithmetics.
In particular, we verify
\begin{equation*}
L(\psi([{\bf x}_k])) - L([{\bf x}_k]) < 0
\end{equation*}
for interval vectors $[{\bf x}_k]$ containing $D_k$ with $[{\bf x}_k]\subset D_L$, which imply that the functional $L({\bf x}) = ({\bf x}-{\bf x}^\ast)^T Y ({\bf x}-{\bf x}^\ast)$ is a Lyapunov function on $D_k$.
In Stage 2, there is the case that overestimates of intervals with interval arithmetics cause the failure of validations, which is discussed in Section \ref{section-example-discrete}.

\subsection{${\bf m}$-Lyapunov functions}
As in the case of continuous systems, we can discuss the arbitrary choice of quadratic functions.
In the construction of $L({\bf x})$, we replace the diagonal matrix $I^\ast$ by $M^\ast = {\rm diag}(i_1,i_2,\cdots,i_n)$ in the definition of $Y$, where
\begin{equation*}
i_j = \left\{ \begin{array}{c} 
                m_j, \quad \mbox{if} \quad |\lambda_j| < 1, \\
                -m_j, \quad \mbox{if} \quad |\lambda_j| > 1
                \end{array}\right.
\end{equation*}
and ${\bf m} = \{m_j\}_{j=1}^n$ is a sequence of given positive numbers.
As in the case of continuous systems, define the quadratic function $L_{M^\ast}$ as
\begin{equation}
\label{Lyapunov-M-map}
L_{M^\ast}({\bf x}) := ({\bf x} - {\bf x}^\ast)^T Y_{M^\ast}({\bf x} - {\bf x}^\ast),\qquad Y_{M^\ast} = {\rm Re}(X^{-H}M^\ast X^{-1}), 
\end{equation}
which is nothing but $L$ replacing $I^\ast$ by $M^\ast$. 
We can then prove negative definiteness of the matrix $B^\ast$ corresponding to the differential $(dL_{M^\ast}/dt)({\bf x})$ at ${\bf x}={\bf x}^\ast$ by similar arguments as Proposition \ref{prop-hyp-map}.
Consequently, we obtain the following corollary.

\begin{corollary}
\label{cor-m-Lyap-map}
Let $D_L$ be a compact, star-shaped domain containing a fixed point ${\bf x}^\ast$ of $\psi$. 
Define a matrix $B_{M^\ast}({\bf x})$ by $B({\bf x})$ in (\ref{matrix-discrete}) replacing $I^\ast$ by $M^\ast$.
Assume that the matrix $B_{M^\ast}({\bf x})$ is strictly negative definite for all ${\bf z}\in D_L$. 
Then $L_{M^\ast}({\bf x})$ is a Lyapunov function on $D_L$ in the sense of Definition \ref{dfn-Lyapunov-map}. 
Moreover, ${\bf x}^\ast$ is the unique fixed point in $D_L$.
\end{corollary}

We shall call the function $L_{M^\ast}$ being a Lyapunov function {\em an {${\bf m}$}-Lyapunov function} (for $\psi$).
Note that the geometry of $D_L$ depends on the choice of positive numbers ${\bf m}$, which is discussed in Section \ref{section-example-discrete}.

\subsection{Lyapunov functions around hyperbolic fixed points}

In Theorem \ref{thm-Lyapunov-map}, we focused on the strict negative definiteness of $B({\bf x})$ in (\ref{matrix-discrete}) in a given domain containing ${\bf x}^\ast$.
The negative definiteness of $B({\bf x})$ reflects the hyprbolicity of ${\bf x}^\ast$. 
Indeed, for a {\em rigorous} hyperbolic fixed point ${\bf x}^\ast$, the matrix $B({\bf x}^\ast)$ is strictly negative definite.
This fact yields the following proposition; namely, {\em all hyperbolic fixed points locally admit Lyapunov functions.}

\begin{proposition}
\label{prop-hyp-map}
Assume that ${\bf x}^\ast$ is a hyperbolic fixed point of $\psi$. Then, for all points in a sufficiently small neighborhood of ${\bf x}^\ast$, the matrix $B({\bf x})$ in (\ref{matrix-discrete}) is strictly negative definite.
In particular, the functional $L({\bf x}) = ({\bf x}-{\bf x}^\ast)^T Y ({\bf x}-{\bf x}^\ast)$ is a Lyapunov function in such a neighborhood.
\end{proposition}

\begin{proof}
For proving the negative definiteness of $B({\bf x})$, it is sufficient to prove that $(A^\ast)^T Y A^\ast - Y$ is negative definite, where $A^\ast = A_I({\bf x}^\ast) = D\psi({\bf x}^\ast)$, since eigenvalues depend continuously on ${\bf z}$.
Note that, for quadratic forms given by real vectors, the negative definiteness of $(A^\ast)^T Y A^\ast - Y$ is equivalent to that of $(A^\ast)^H \hat{Y} A^\ast - \hat{Y}$. 
In what follows we prove the negative definiteness of $(A^\ast)^H \hat{Y} A^\ast - \hat{Y}$.

By the definition of the Hermitian matrix $\hat{Y}$, we get
\begin{align*}
(A^\ast)^H \hat{Y} A^\ast - \hat{Y} &= 
    X^{-H} (\Lambda^H I^\ast \Lambda - I^\ast) X^{-1} \\
    &= X^{-H} (\Lambda^H \Lambda - I) I^\ast X^{-1} \\
    &= X^{-H}( |\Lambda|^2 - I ) I^\ast X^{-1},
\end{align*}
where $|\Lambda|$ denotes the matrix whose entries are absolute values of the corresponding entries of $\Lambda$.
By the definition of the matrix $I^\ast$, one easily have that $(A^\ast)^H \hat{Y} A^\ast - \hat{Y}$ is a negative definite Hermitian matrix, which implies that $L({\bf x}) = ({\bf x}-{\bf x}^\ast)^T Y ({\bf x}-{\bf x}^\ast)$ is a Lyapunov function in a sufficiently small neighborhood of ${\bf x}^\ast$. 
\end{proof}

The same arguments as Proposition \ref{prop-Y-cont} with Stable Manifold Theorem for maps (e.g. \cite{Rob}) yield the eigenstructure of $Y$ in (\ref{Lyapunov-map}).

\begin{proposition}
Consider a functional $L({\bf x}) = {\bf x}^TY{\bf x}$ for some real symmetric matrix $Y$.
Assume that the origin ${\bf 0}$ is a hyperbolic fixed point of $\psi$ such that $D\psi({\bf 0})$ has $u$ eigenvalues with moduli larger than $1$ and $s = n-u$ eigenvalues with moduli less than $1$.
Assume that, for a compact star-shaped domain $D_L$ with ${\bf 0}\in {\rm int}D_L$, the following inequality holds:
\begin{equation*}
L(\psi({\bf x})) < L({\bf x}),\quad  \forall {\bf x}\in D_L\setminus \{{\bf 0}\}.
\end{equation*}
Then $Y$ is non-singular and has $u$ negative eigenvalues and $s = n-u$ positive eigenvalues.
\end{proposition}

% New section
\section{Towards applications to periodic orbits}
\label{section-Lyapunov-periodic}
For applications of Lyapunov functions for discrete dynamical systems, a typical example is to periodic orbits as fixed points of Poincar\'{e} maps. 
As preliminaries of applications in this direction, we review the verification method of periodic orbits in continuous dynamical systems in this section. 
During verifications, implementations of Poincar\'{e} maps and their differentials arise, which are necessary to validate Lyapunov functions.

% New subsection
\subsection{Poincar\'{e} maps and their differentials}
Here we review the construction of the Poincar\'{e} map and its differential.

%Assume that the dynamical system generated by (\ref{autonomous}) possesses the periodic orbit and that it is specified as a fixed point of the Poincar\'{e} map defined on a hyperplane $\Gamma$ with the unit normal vector ${\bf n}_\Gamma$.

Let $\varphi$ be the flow generated by (\ref{autonomous}).
The Poincar\'{e} map $P({\bf x})$ for the point ${\bf x}$ on a Poincar\'{e} section is then constructed as follows.
Firstly, we put $\tilde {\bf x}_0$ as a point on an approximate periodic orbit.
Secondly, let ${\bf n}$ be a unit vector which is approximately parallel to $f(\tilde {\bf x}_0)$. 
Let $\Gamma$ be the hyperplane such that ${\bf n}$ is the unit normal vector to $\Gamma$, which is a candidate of our Poincar\'{e} section.
Thirdly, 
\begin{enumerate}
\item Compute the solution orbit $\varphi(t,{\bf x}_0)$ with the initial point ${\bf x}_0\in \Gamma$;
\item Compute times $\underline{t} < \overline{t}$ satisfying $\{{\bf n}^T\cdot (\varphi(\underline{t}, {\bf x}_0)-\tilde {\bf x}_0)\} \{{\bf n}^T\cdot (\varphi(\overline{t}, {\bf x}_0)-\tilde {\bf x}_0)\} < 0$;
\item Letting $[t] := [\underline{t}, \overline{t}]$, compute an enclosure $[{\bf x}] := \varphi([t], {\bf x}_0)$; and
\item Verify $0\not \in {\bf n}^T\cdot f([{\bf x}])$.
\end{enumerate}
The interval enclosure $[{\bf x}]$ contains the image of the Poincar\'{e} map $P$ of ${\bf x}_0$.
Conditions 2 and 4 imply the unique existence of the time $T_r({\bf x}_0)>0$ within $[\underline{t}, \overline{t}]$ such that $\varphi(T_r({\bf x}_0),{\bf x}_0) \in \Gamma$ holds. 
Since the vector field $f$ is continuous, $P$ is locally diffeomorphic, which implies that $[{\bf x}]$ is indeed the rigorous enclosure of the Poincar\'{e} map $P({\bf x}_0)$.

The Jacobian matrix of the Poincar\'{e} map $P$ at ${\bf x}_0$ is derived as follows.
For a trajectory $\{\varphi(t,{\bf x})\}_{t\geq 0}$, let
\begin{equation*}
V(t;{\bf x}) = \displaystyle \frac{\partial}{\partial {\bf x}}\varphi(t,{\bf x}),
\end{equation*}
which is the $n$-squared matrix satisfying the following variational equation around $\{\varphi(t,{\bf x})\}_{t\geq 0}$ for fixed ${\bf x}$:
\begin{align*}
\displaystyle \frac{dV}{dt} &= Df(\varphi(t,{\bf x}))V,\\
V(0; {\bf x}) &= I_n,
\end{align*}
where $I_n$ is the $n$-dimensional identity matrix.
%By the similar argument to the above, the enclosure of the solution matrix is given by $[V] := V([t],{\bf x})$, where $[t] = [\underline{t}, \overline{t}]$.

The Jacobian matrix of the Poincar\'{e} map $P({\bf x}_0)$ at ${\bf x}_0$ is then calculated by
\begin{equation*}
DP({\bf x}_0) = \displaystyle \frac{d}{d{\bf x}'}\varphi(T_r({\bf x}_0),{\bf x}_0) = \left( I_n - \displaystyle \frac{ f(\varphi(T_r({\bf x}_0),{\bf x}_0)){\bf n}^T} {{\bf n}^T \cdot f( \varphi(T_r({\bf x}_0),{\bf x}_0) )} \right) V(T_r({\bf x}_0); \varphi(T_r({\bf x}_0),{\bf x}_0)),
\end{equation*}
where ${\bf x}'$ denotes a coordinate on $\Gamma$. 
Note that $0\not \in {\bf n}^T\cdot f([{\bf x}])$ holds by Condition 4, which guarantees the well-definedness of the differential $DP$ at ${\bf x}_0$.

%\bigskip
%In what follows we assume that there is a periodic orbits $\gamma$ as well as associated Poincar\'{e} map, and that the existence domain is specified as an interval domain on a hyperplane which corresponds to the Poincar\'{e} section.
%In this case, Lyapunov functions for discrete dynamical systems defined by Poincar\'{e} maps are considered.

% New subsection
\subsection{Remark on verification of the existence of periodic orbits}
\label{section-valid-per}

%\bigskip
%Consider the iteration
%\begin{equation*}
%{\bf x}_{n+1} = P({\bf x}_n), \quad P({\bf x}) = \varphi(T_r({\bf x}),{\bf x}), 
%\end{equation*}
%which gives us the Poincar\'{e} map $P : \Gamma\to \Gamma$.

We have prepared the (rigorous enclosures of) Poincar\'{e} map $P$ and its differential $DP$ in the previous subsection.
Before moving to validations of Lyapunov domains of $P$, we review a method for computing the enclosure of fixed points of $P$; namely, periodic orbits.
\par
There are mainly two approaches to verify the existence of periodic orbits with interval arithmetics:
\begin{itemize}
\item[1.] Construct Poincar\'{e} maps on a section and verify their fixed point, e.g., \cite{ZLoh};
\item[2.] Regard periodic orbits as solutions of the two-point boundary value problem
\begin{equation}
\label{BVP}
\begin{cases}
\displaystyle \frac{d}{d\tilde t}{\bf x}(\tilde t) = Tf({\bf x}(\tilde t)) & \\
{\bf x}(1) = {\bf x}(0) &
\end{cases}
\end{equation}
and verify solutions with bordering conditions.
\end{itemize}
Here we briefly review the second approach proposed by the second and the third authors (e.g. \cite{HY}).
Consider the autonomous differential equation (\ref{autonomous}).
Let $\Gamma$ be a hyperplane and ${\bf n}_\Gamma$ be the unit normal vector to $\Gamma$.
Our aim here is to validate a periodic orbit through a point ${\bf w}^\ast \in \Gamma$ as well as its period $T^\ast$.

Our strategy is the reduction to the boundary value problem (\ref{BVP}) with {\em the bordering condition}:
\begin{equation}
\label{bvp-border}
\begin{cases}
{\bf n}_\Gamma \cdot ({\bf w}-\tilde {\bf w})= 0, & \\
%\varphi(T,{\bf w}) = {\bf w}, & 
\varphi(T,{\bf w}) = {\bf w}, & 
\end{cases}\quad T>0, {\bf w}\in \Gamma,
\end{equation}
where $\tilde {\bf w} \in \Gamma$ be a point on, say, an approximate periodic orbit.
The first equation in (\ref{bvp-border}) is called the bordering condition and remove the ambiguity of detection of points on periodic orbits caused by translation invariance.

Define a map $K : \mathbb{R}^{1+n}\to \mathbb{R}^{1+n}$ by
\begin{equation*}
K({\bf z}) := 
\begin{pmatrix}
{\bf n}_\Gamma \cdot ({\bf w}-\tilde {\bf w}) \\
\varphi(T,{\bf w}) - {\bf w}
%P({\bf w}) - {\bf w}
\end{pmatrix},\quad {\bf z} = (T, {\bf w})^T.
\end{equation*}
The pair $(T^\ast, {\bf w}^\ast)$ of a periodic point on $\Gamma$ and its period corresponds to the zero of $K$.
The Jacobi matrix $DK$ of $K$ at ${\bf z} = (T, {\bf w})^T$ is
\begin{equation*}
DK({\bf z}) = 
\begin{pmatrix}
{\bf 0} & {\bf n}_\Gamma^T \\
f(\varphi(T,{\bf w})) & V(T;{\bf w}) - I_n
\end{pmatrix},
\end{equation*}
where $V$ is the variation matrix associated with $\varphi$.

In the next step, we construct a Newton-like operator $N$ using $K$ and $DK$, and apply the quasi-Newton method to validate a zero of $K$.
Let $DK_a$ be a nonsingular matrix which is an approximation to $DK(\tilde {\bf z})$ at the pair $\tilde {\bf z} = (\tilde T, \tilde {\bf w})$ of an approximate period $\tilde T$ and an approximate periodic point $\tilde {\bf w}$ on $\Gamma$.
Then define the map $N :\mathbb{R}^{1+n}\to \mathbb{R}^{1+n}$ as
\begin{equation*}
N({\bf z}) := {\bf z} - DK_a^{-1}\cdot K({\bf z}).
\end{equation*}

Finally, setting initially a small interval set $[Z_0] = ([T], [W])^T$ containing $\tilde {\bf z}$ with an interval vector $[W]\subset \Gamma$, %containing a pair of approximate periodic point and period, 
apply the following algorithm:
\begin{algorithm}
\label{alg-fixpt}
Initially set $k=0$ and $\epsilon > 0$ small in advance.
\begin{enumerate}
\item Check if $N([Z_k])\subset [Z_k] = ([T_k], [W_k])$. If this operation passes, return \lq\lq succeeded" and stop the algorithm.
\item If Step 1 fails, reset $[Z_{k+1}] = ([T_{k+1}], [W_{k+1}]) := (1+\epsilon)N([Z_k]) - \epsilon N([Z_k])$ and go back to Step 1 replacing $k$ by $k+1$.
\end{enumerate}
\end{algorithm}
If this algorithm returns \lq\lq succeeded" at $k=k_0$, then there is a fixed point $(T^\ast, {\bf w}^\ast)$ of $N$ in $[Z_{k_0}]$, which is actually  a zero of $K$. 
It implies that ${\bf w}^\ast$ is the intersection point of $\Gamma$ and a periodic orbit with the period $T^\ast$.

\begin{remark}\rm
For computations with $N$, we can use the following Krawczyk-type operator $\tilde N$ instead of $N$:
\begin{equation*}
\tilde N([Z]) := DK_a^{-1}\{ DK_a \hat {\bf z} - K(\hat {\bf z}) + (DK_a - DK([Z])) ([Z] - \hat {\bf z})\},
\end{equation*}
where $\hat {\bf z}$ is the center point of $[Z]$,
which is obtained by considering the mean value form of $N$.

Also note that, if Algorithm \ref{alg-fixpt} returns \lq\lq succeeded" at $k=k_0$, we know that the Poincar\'{e} map can be defined on its Poincar\'{e} section $\Gamma\cap [W_{k_0}]$.
\end{remark}

% New section
\section{Numerical examples for flows}
\label{section-example-cont}

As demonstrating validations of Lyapunov functions for flows, we consider the three dimensional FitzHugh-Nagumo system:
\begin{align}
\notag
\frac{du}{dt} &= v,\\
\label{FN}
\frac{dv}{dt} &= \frac{1}{\delta}(cv-f(u)+w),\\
\notag
\frac{dw}{dt} &= \frac{\epsilon}{c}(u-\gamma w),
\end{align}
where $f(u) = u(u-a)(1-u)$, and $a, c, \delta, \epsilon$ and $\gamma$ are (positive) parameters.
The system (\ref{FN}) is regarded as the traveling wave equation of the following partial differential equation:
\begin{align*}
\frac{\partial u}{\partial \tau} &= \delta \frac{\partial^2 u}{\partial x^2} - f(u) + w,\\
\frac{\partial w}{\partial \tau} &= \epsilon (u-\gamma w)
\end{align*}
with $(u(x,\tau),w(x,\tau)) = (u(t),w(t))$, $t = x-c\tau$.

Here we fix parameters as follows so that the system possesses three equilibria:
\begin{equation*}
a=0.2,\quad c=5,\quad \delta = 5,\quad \epsilon = 0.15,\quad \gamma = 20. 
\end{equation*}

\begin{remark}\rm
Throughout our computations, we have used the following computation environments.
\begin{itemize}
\item OS: Windows $7$ Professional $64$-bit (6.1, Build 7601) Service Pack 1 (7601. win7sp1\_gdr.150928-1507).
\item Memory: 16384MB RAM.
\item Software: MATLAB version R2012a and INTLAB version: v6 \cite{INTLAB}.
\end{itemize}
\end{remark}

Equilibria with these parameters can be (approximately) computed below:
\begin{align*}
{\bf x}_1^\ast &= (u_1^\ast ,v_1^\ast ,w_1^\ast ) = \left(0,0,0\right),\\
{\bf x}_2^\ast &= (u_2^\ast ,v_2^\ast ,w_2^\ast ) \approx \left(0.268337520964460,0,0.013416876048223\right),\\
{\bf x}_3^\ast &= (u_3^\ast ,v_3^\ast ,w_3^\ast ) \approx \left(0.931662479035540,0,0.046583123951777\right).
\end{align*}

It numerically turns out that 
\begin{align*}
&\dim W^s({\bf x}_1^\ast) = 2,\quad \dim W^s({\bf x}_2^\ast) = 1,\quad \dim W^s({\bf x}_3^\ast) = 2,\\
&\dim W^u({\bf x}_1^\ast) = 1,\quad \dim W^u({\bf x}_2^\ast) = 2,\quad \dim W^u({\bf x}_3^\ast) = 1.
\end{align*}

We compute symmetric matrices $Y_j$, $j=1,2,3$, in (\ref{Y-cont}) to obtain
	\begin{align*}
	Y_1&=\left(\begin{array}{ccc} 
	1.9045048614 &-1.9684846596 &-0.7930467270 \\
	-1.9684846596 &-1.8022725548 &0.2703701350 \\ 
	-0.7930467270 &0.2703701350 &2.3772099623
	\end{array}\right),
	\end{align*}
	\begin{align*}
	Y_2&=\left(\begin{array}{ccc} 
	-2.2485667721 & 2.5290401659 & 0.6528894939 \\ 
	2.5290401659 &-6.9945543958 &-1.2823752332 \\ 
	0.6528894939 &-1.2823752332 & 1.8533105010
	\end{array}\right),
	\end{align*}
	\begin{align*}
	Y_3&=\left(\begin{array}{ccc} 
	1.7202944579 &-1.8934526570 &-0.8110063777 \\ 
	-1.8934526570 &-1.6051655895 & 0.3062332874 \\ 
	-0.8110063777 & 0.3062332874 & 2.4375167185
	\end{array}\right).
	\end{align*}
Note that matrices $Y_j$, $j=1,2,3$, are indeed symmetric.
We now verify Lyapunov domains of $x_j^\ast$.
Firstly, set sample domains $D_j$, $j=1,2,3$, as
\begin{align*}
D_1 &:= [-0.5, 0.5]\times [-0.5, 0.5]\times [-0.5, 0.5],\\
D_2 &:= [0, 1]\times [-0.5, 0.5]\times [-0.5, 0.5],\\
D_3 &:= [0.5, 1.5]\times [-0.5, 0.5]\times [-0.5, 0.5], 
\end{align*}
each of which contains ${\bf x}_j^\ast$. 
Let 
\begin{equation*}
L_j({\bf x}) := ({\bf x}-{\bf x}_j^\ast)^T Y_j ({\bf x}-{\bf x}_j^\ast),\quad j=1,2,3,
\end{equation*}
be the candidates of Lyapunov functions.
Next we divide these domains into $50\times 50\times 50$ small uniform cubes. 
We then verify the strict negative definiteness of the matrix $A({\bf z})$ in (\ref{matrix-neg-def}) on each small cubes.
Note that, if the matrix $A({\bf z})$ associated with $L_j$ is strictly negative definite on a cube, then $L_j$ is a Lyapunov function on it.

\begin{description}
\item[Description of Figs. \ref{fig-FNfp} - \ref{fig-LorenzEvec}] 
\end{description}
In these figures we distinguish verification results of Lyapunov domains by different colors.
\begin{itemize}
\item {\bf Blue}: Both Stage 1 and 2 are succeeded.
\item {\bf Light Blue}: Only Stage 1 is succeeded.
\item {\bf Yellow}: Only Stage 2 is succeeded.
\item {\bf Red}: Both Stage 1 and 2 are failed.
\end{itemize}
White disks denote locations of equilibria.

\bigskip
Fig. \ref{fig-FNfp} describes domains $D_j\cap \{w=0\}$ and contours of Lyapunov fucntions $L_j$.

\begin{figure}[htbp]\em
\begin{minipage}{0.32\hsize}
\centering
\includegraphics[width=5.0cm]{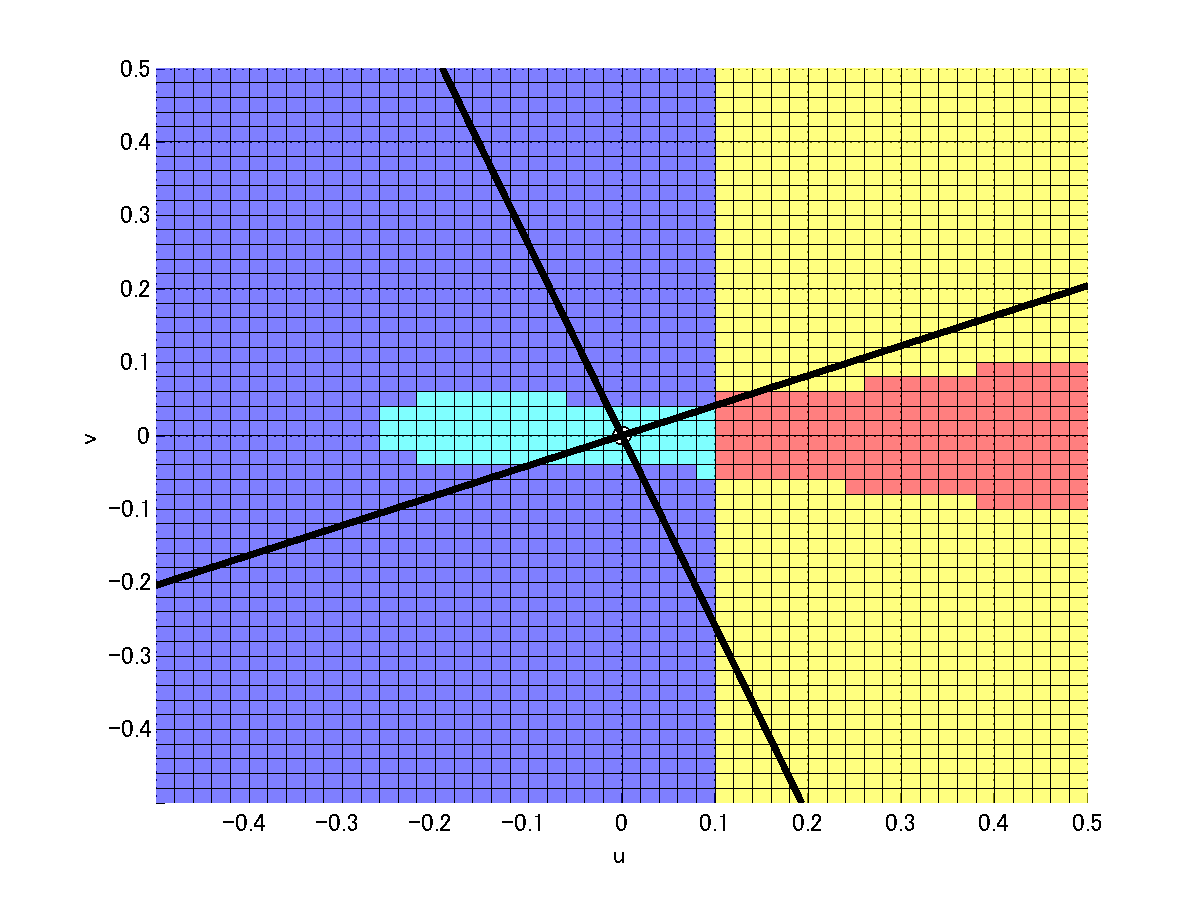}
(a)
\end{minipage}
\begin{minipage}{0.32\hsize}
\centering
\includegraphics[width=5.0cm]{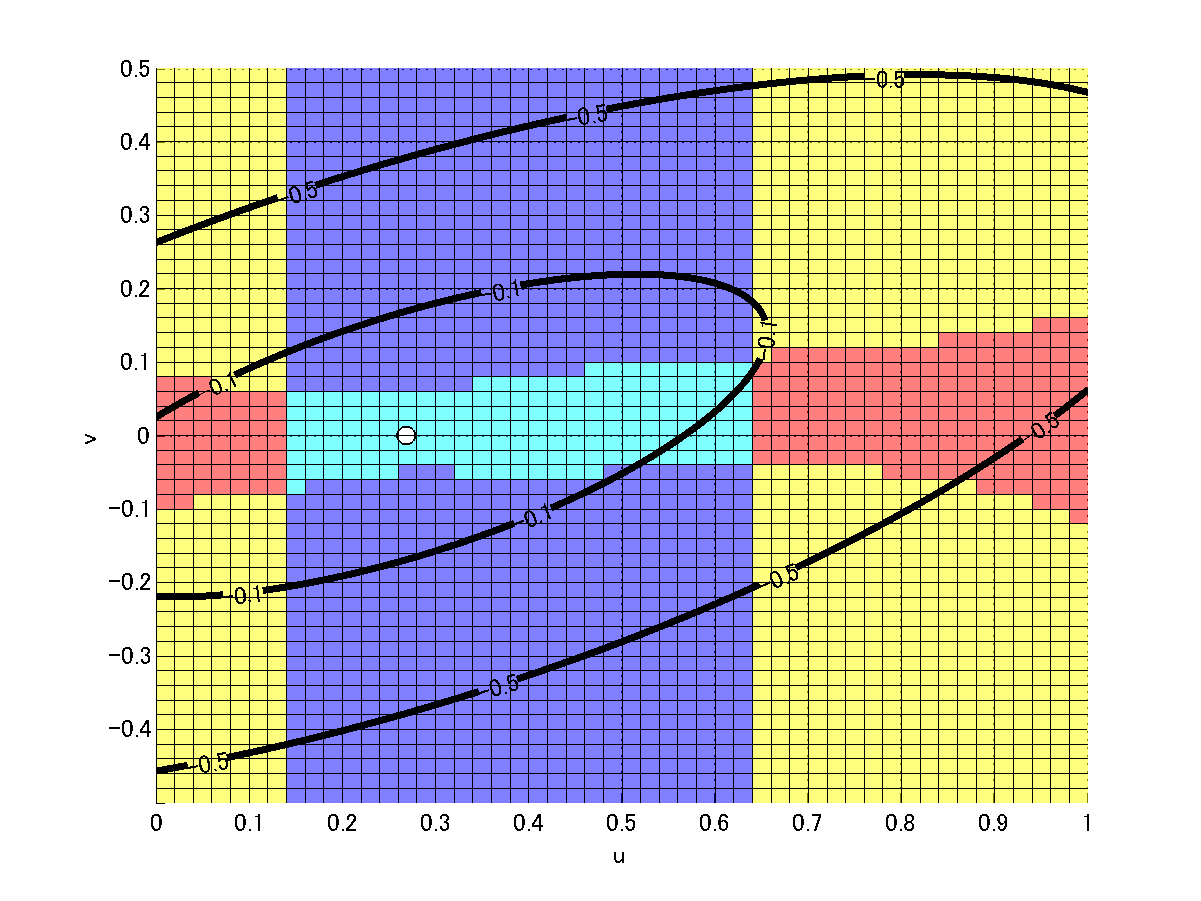}
(b)
\end{minipage}
\begin{minipage}{0.32\hsize}
\centering
\includegraphics[width=5.0cm]{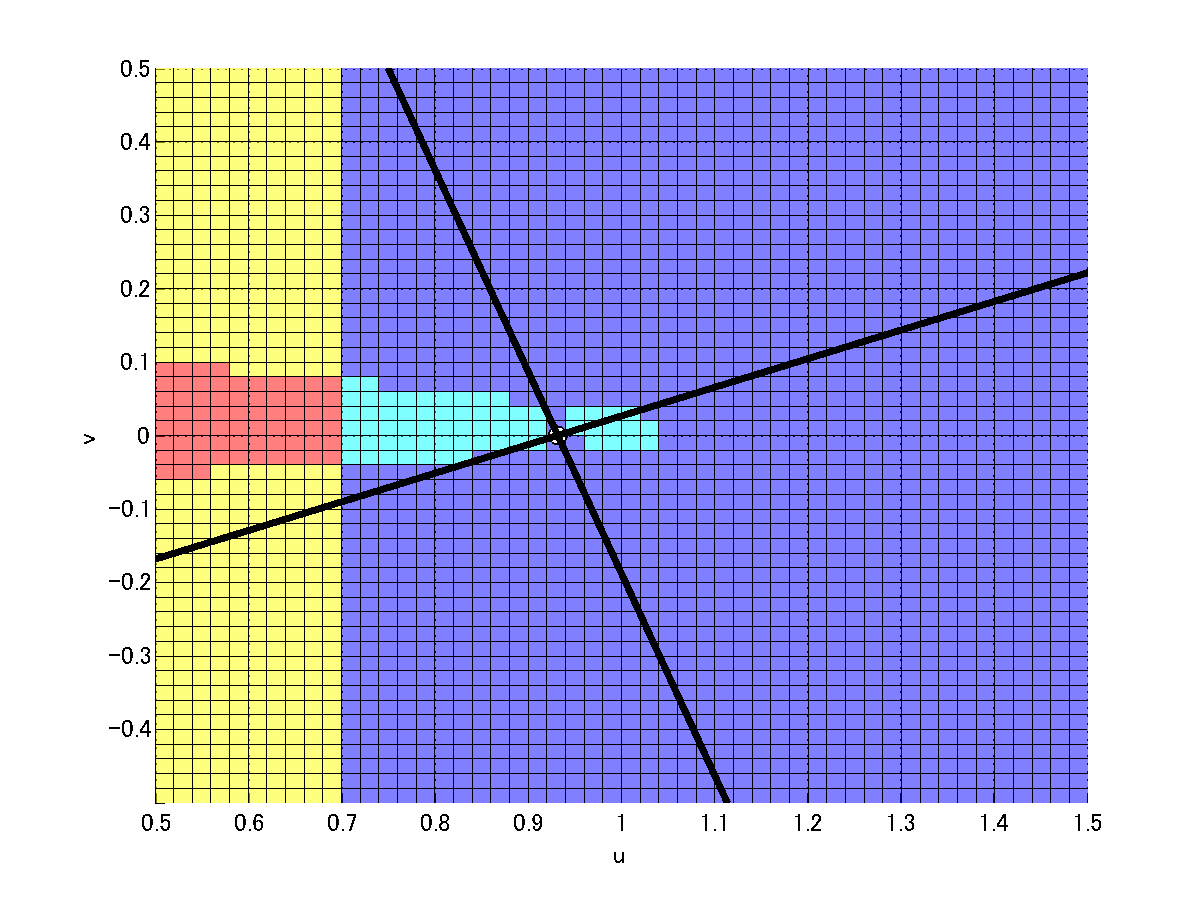}
(c)
\end{minipage}
\caption{Domains $D_i\cap \{w=0\}$, $i=1,2,3$, and contours of Lyapunov functions.}
\label{fig-FNfp}
\begin{flushleft}
(a) shows $D_1\cap \{w=0\}$. Black lines represent contours $\{L_1=0\}$.
\par
(b) shows $D_2^\ast\cap \{w=0\}$. Black curves represent contours $\{L_2=-0.5, -0.1\}$.
\par
(c) shows $D_3\cap \{w=0\}$. Black lines represent contours $\{L_3=0\}$.
\end{flushleft}
\end{figure}

\subsection{Validation of Lyapunov functions}

Figs. \ref{fig-FNfp-V1} - \ref{fig-FNfp-V5} describe sections of domains $D_j$ with $\{v = \text{const.}\}$ with contours of Lyapunov functions $L_j$.

\begin{figure}[htbp]\em
\begin{minipage}{0.32\hsize}
\centering
\includegraphics[width=5.0cm]{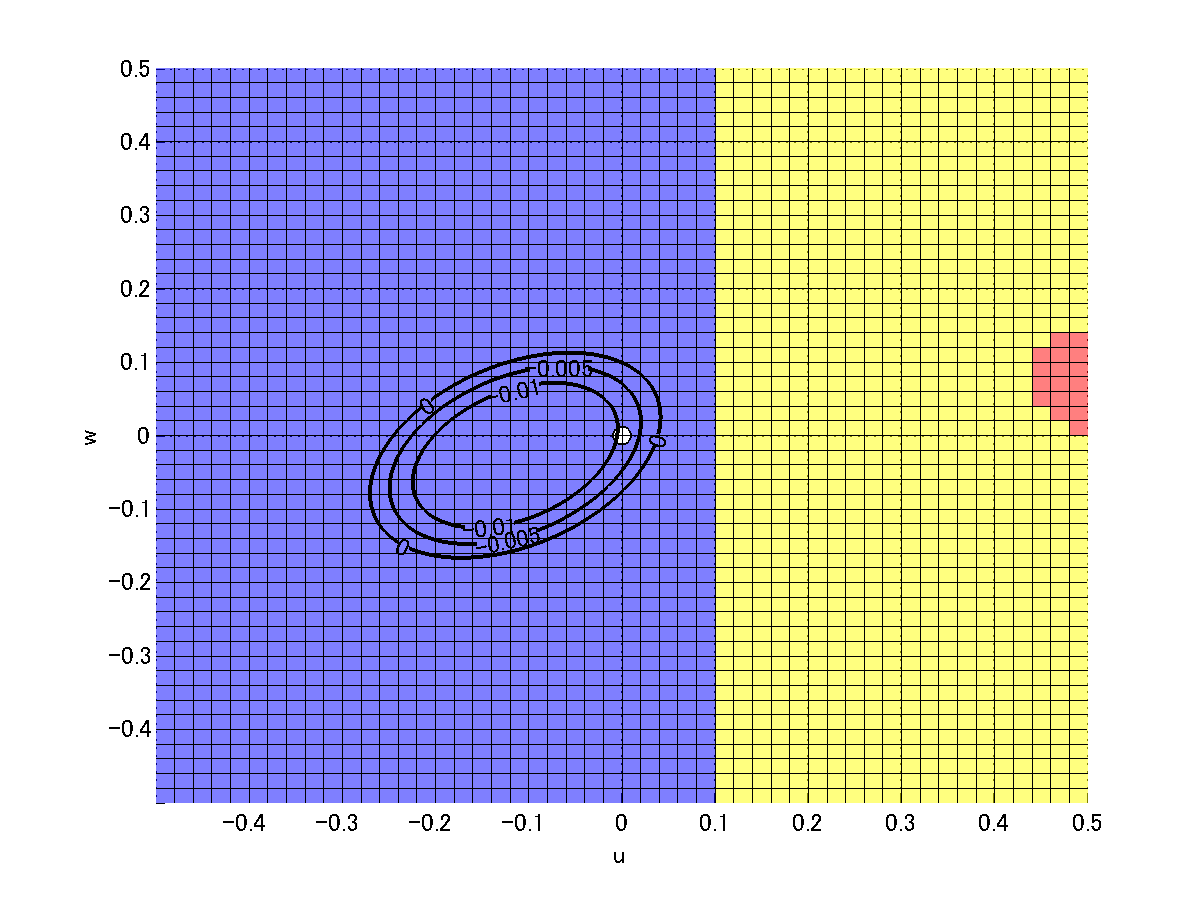}
(a)
\end{minipage}
\begin{minipage}{0.32\hsize}
\centering
\includegraphics[width=5.0cm]{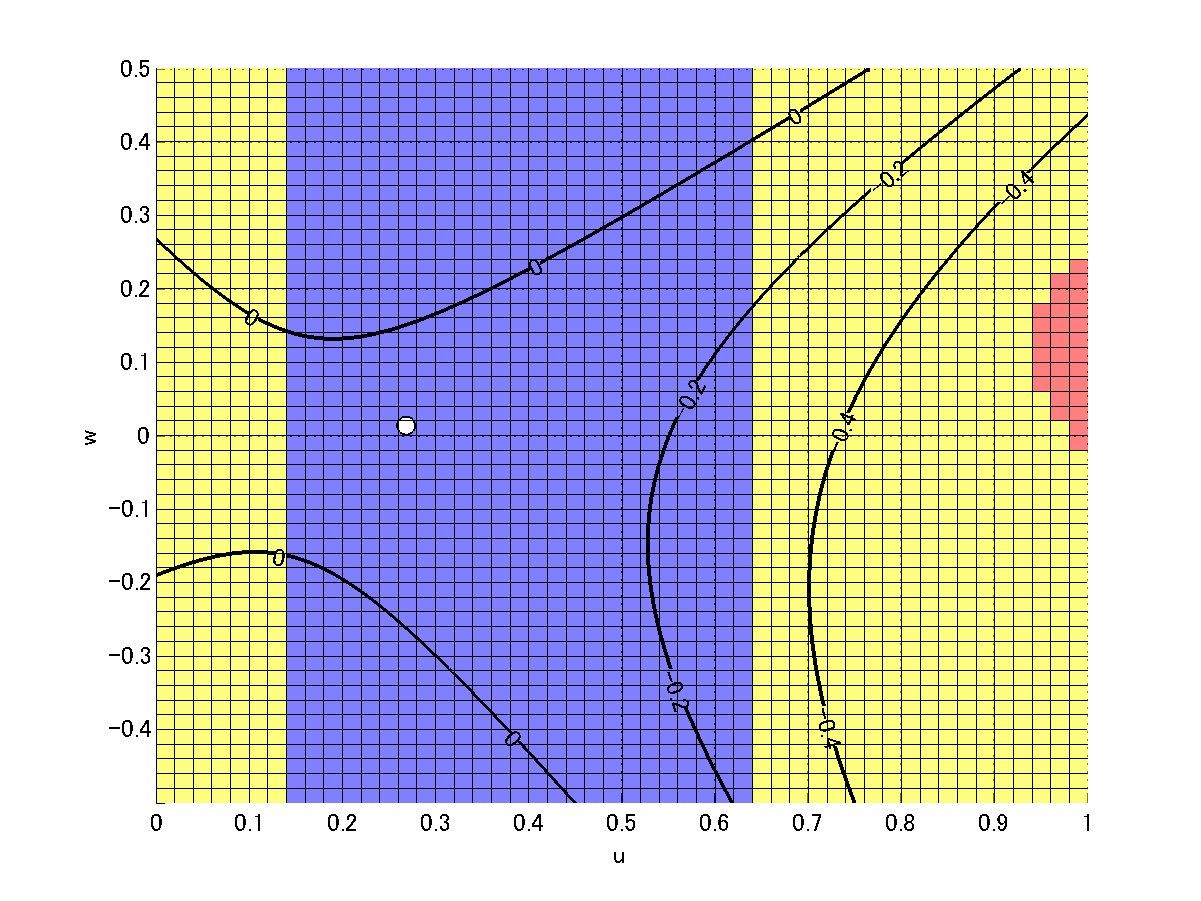}
(b)
\end{minipage}
\begin{minipage}{0.32\hsize}
\centering
\includegraphics[width=5.0cm]{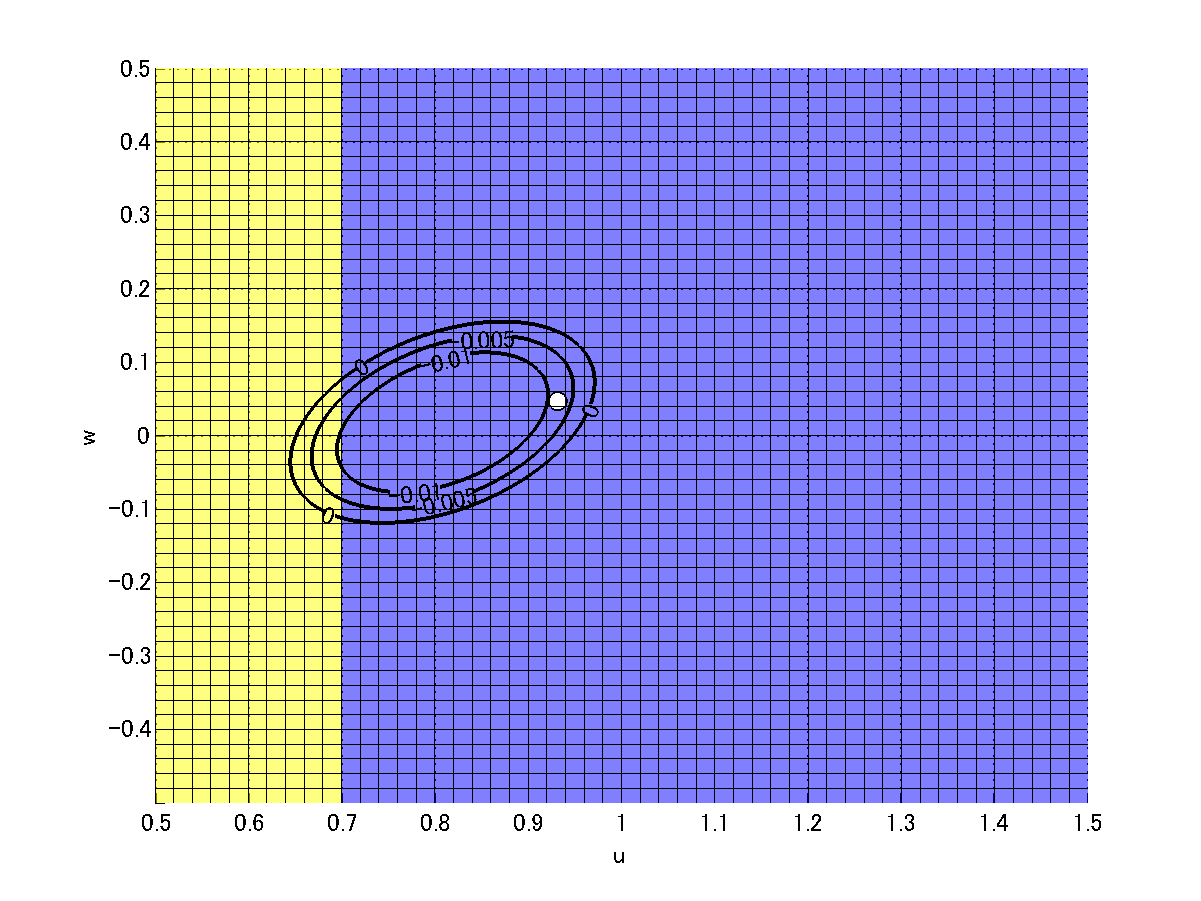}
(c)
\end{minipage}
\caption{Domains $D_i\cap \{v=-0.1\}$, $i=1,2,3$, and contours of Lyapunov functions.}
\label{fig-FNfp-V1}
\begin{flushleft}
(a) shows $D_1\cap \{v=-0.1\}$. Black curves represent contours $\{L_1=-0.01, -0.005, 0\}$.
\par
(b) shows $D_2^\ast\cap \{v=-0.1\}$. Black curves represent contours $\{L_2=-0.4, -0.2, 0\}$.
\par
(c) shows $D_3\cap \{v=-0.1\}$. Black curves represent contours $\{L_3=-0.01, -0.005, 0\}$.
\end{flushleft}
\end{figure}

\begin{figure}[htbp]\em
\begin{minipage}{0.32\hsize}
\centering
\includegraphics[width=5.0cm]{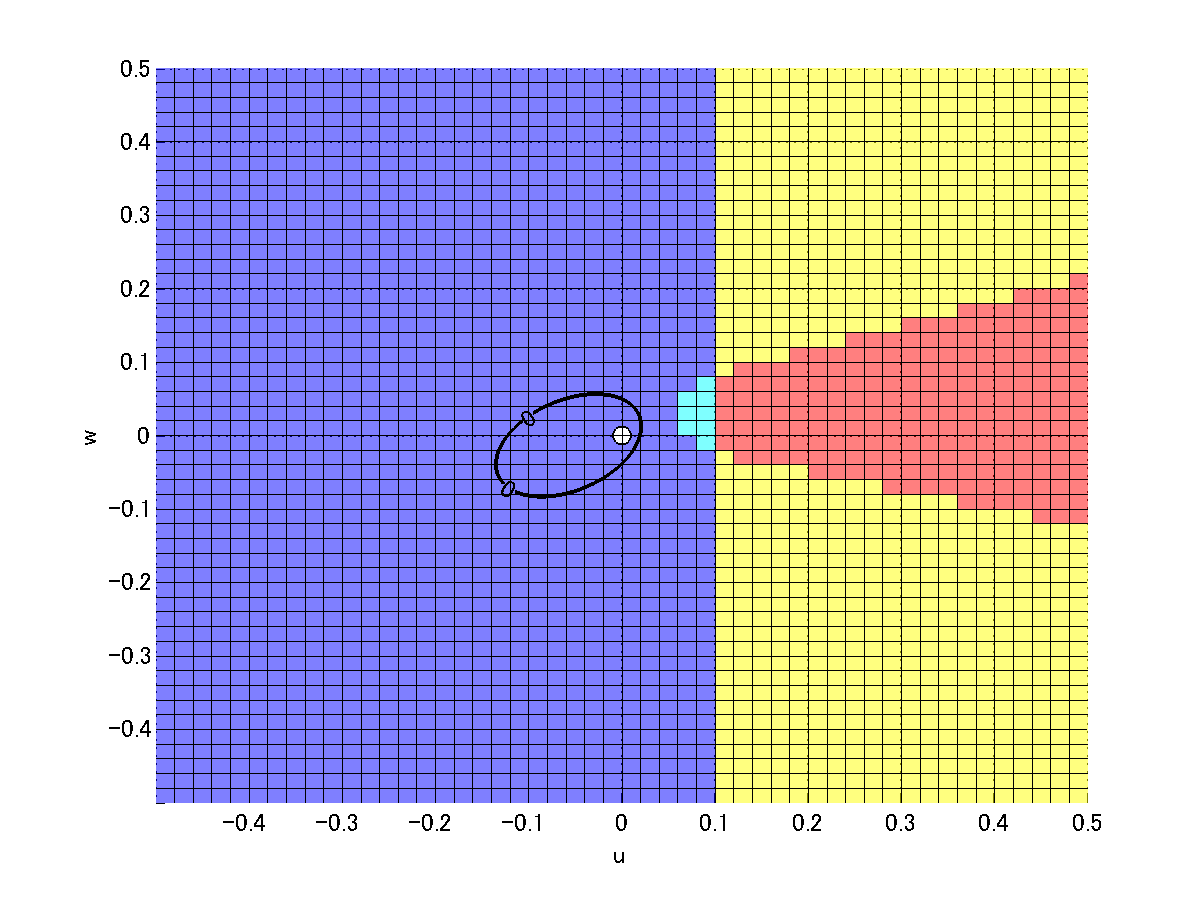}
(a)
\end{minipage}
\begin{minipage}{0.32\hsize}
\centering
\includegraphics[width=5.0cm]{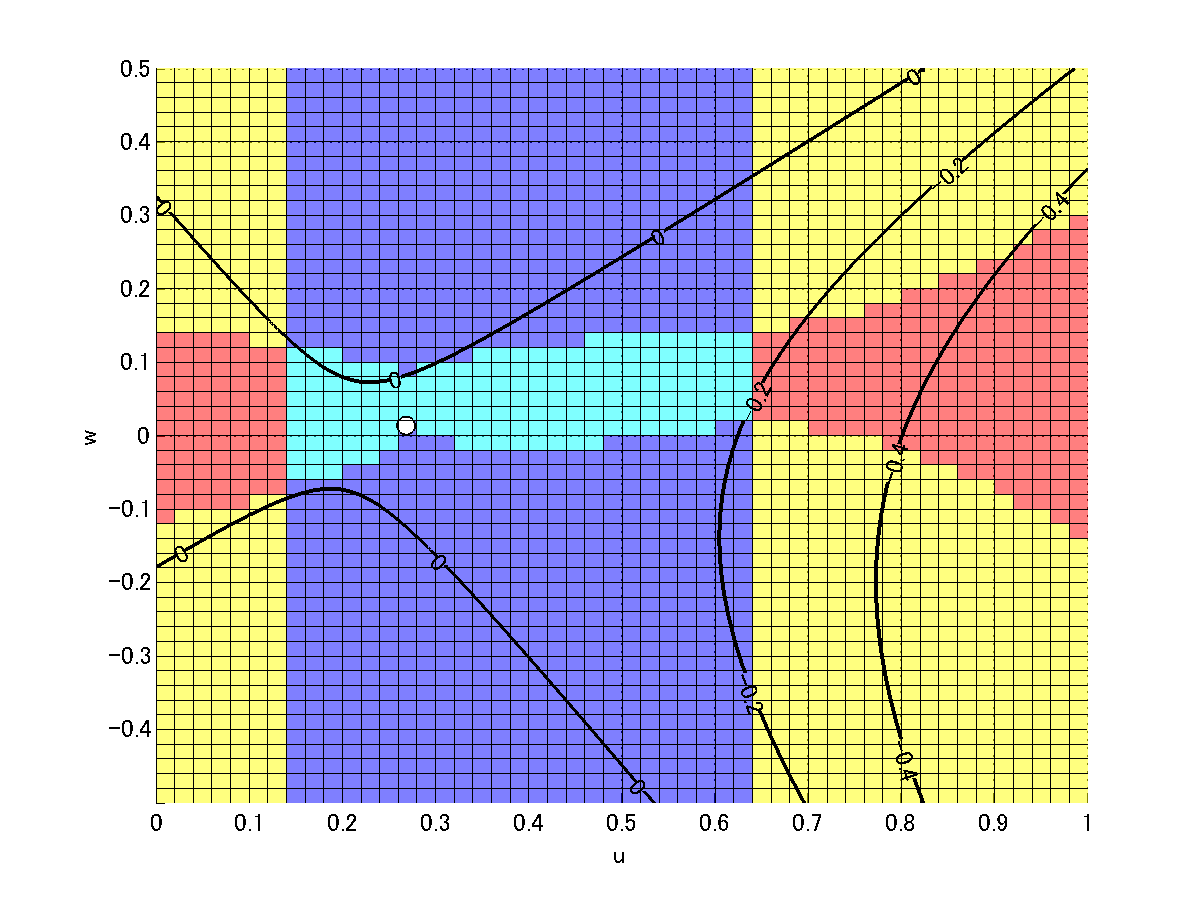}
(b)
\end{minipage}
\begin{minipage}{0.32\hsize}
\centering
\includegraphics[width=5.0cm]{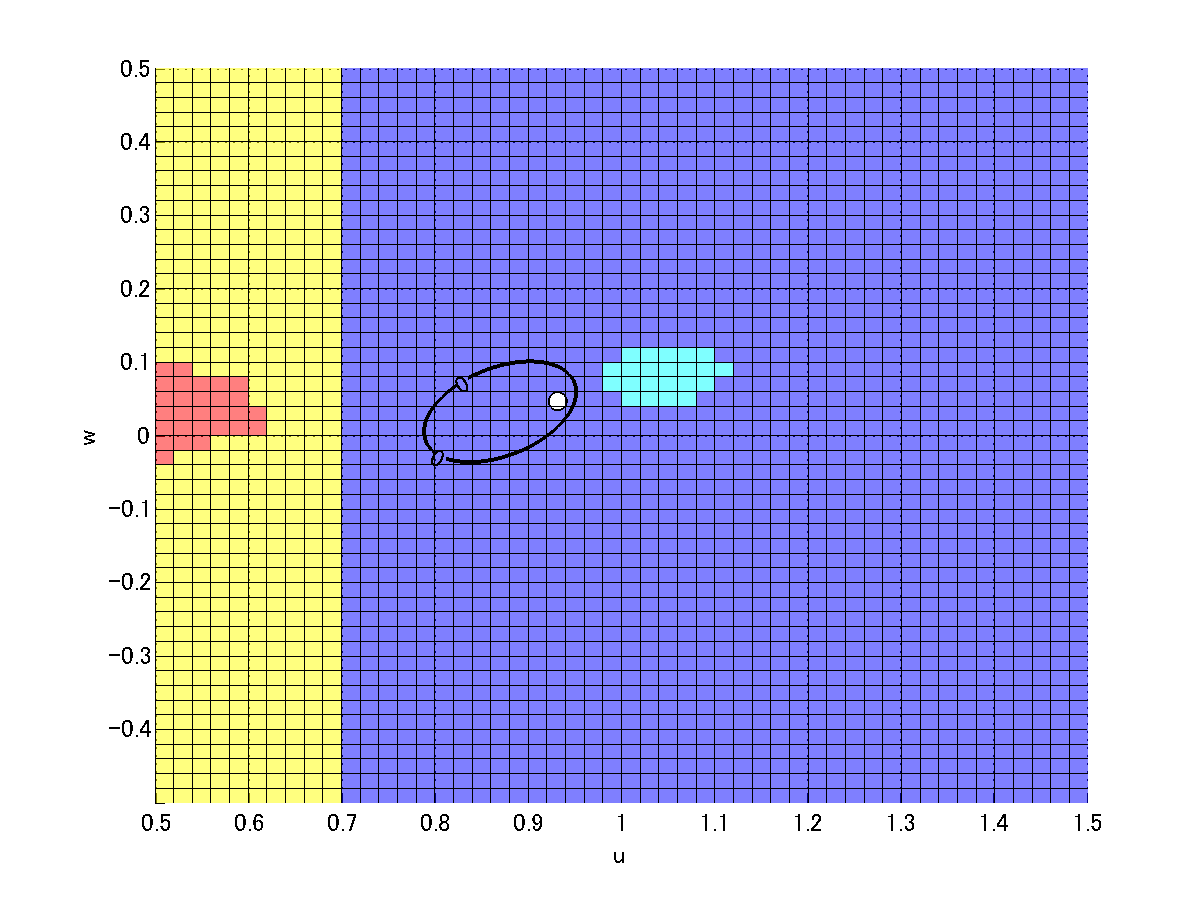}
(c)
\end{minipage}
\caption{Domains $D_i\cap \{v=-0.05\}$, $i=1,2,3$, and contours of Lyapunov functions.}
\label{fig-FNfp-V2}
\begin{flushleft}
(a) shows $D_1\cap \{v=-0.05\}$. Black curve represents the contour $\{L_1= 0\}$.
\par
(b) shows $D_2^\ast\cap \{v=-0.05\}$. Black curves represent $\{L_2=-0.4, -0.2, 0\}$.
\par
(c) shows $D_3\cap \{v=-0.05\}$. Black curve represents the contour $\{L_3= 0\}$.
\end{flushleft}
\end{figure}

\begin{figure}[htbp]\em
\begin{minipage}{0.32\hsize}
\centering
\includegraphics[width=5.0cm]{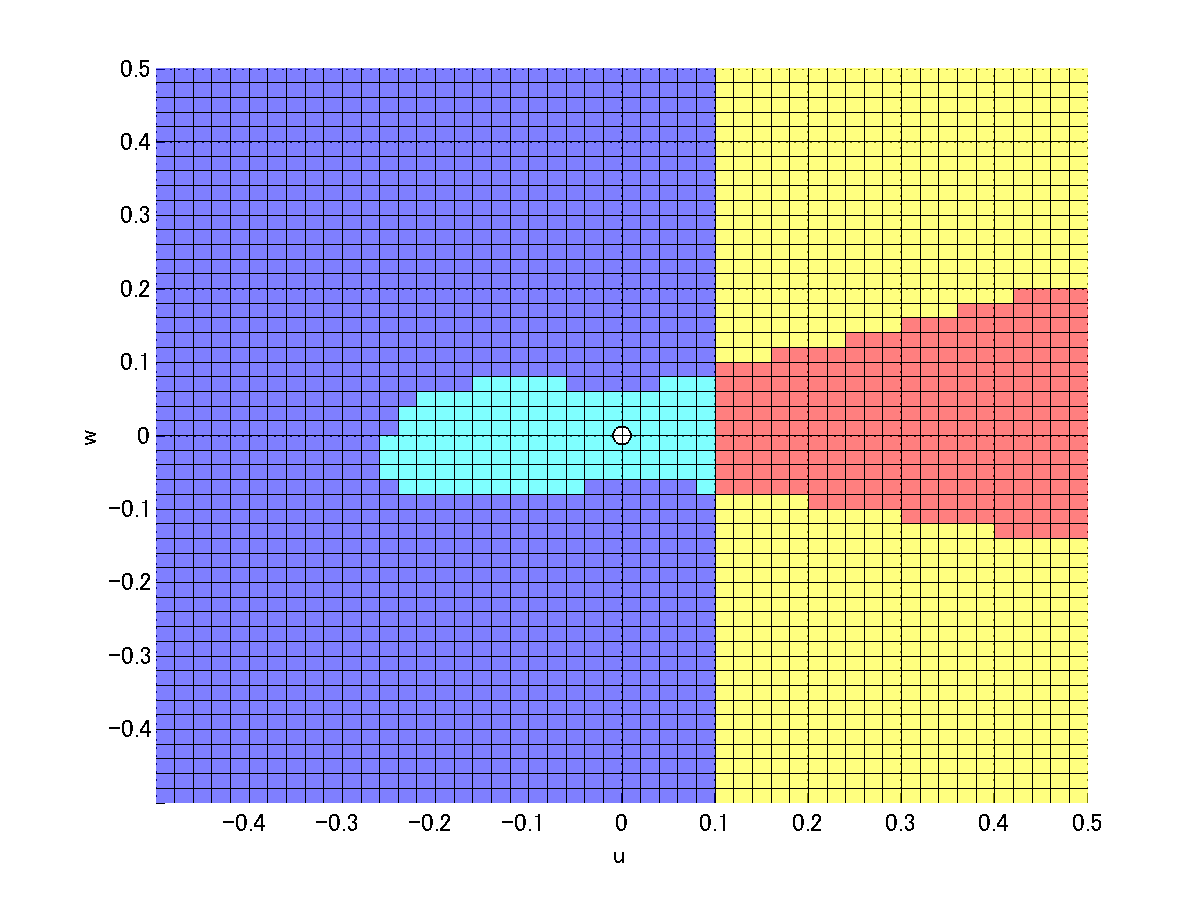}
(a)
\end{minipage}
\begin{minipage}{0.32\hsize}
\centering
\includegraphics[width=5.0cm]{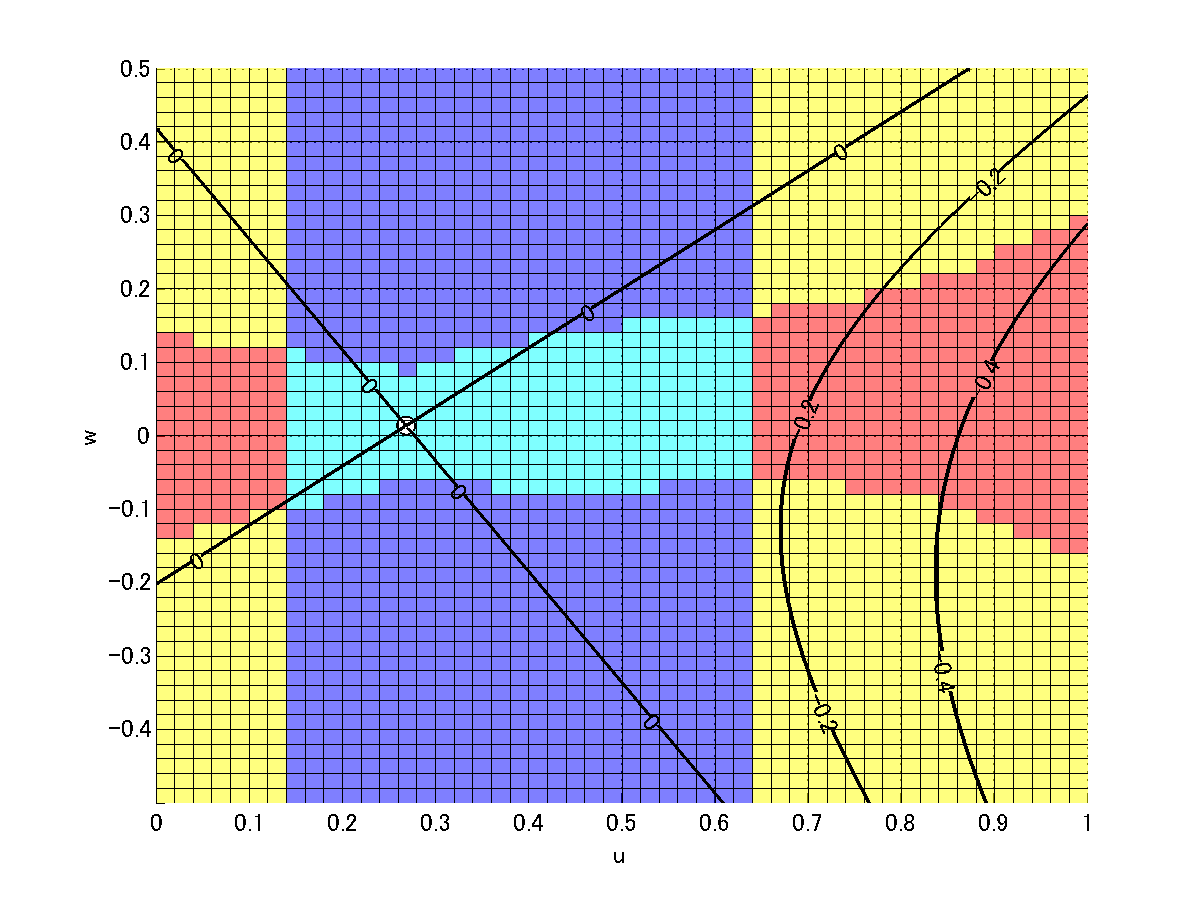}
(b)
\end{minipage}
\begin{minipage}{0.32\hsize}
\centering
\includegraphics[width=5.0cm]{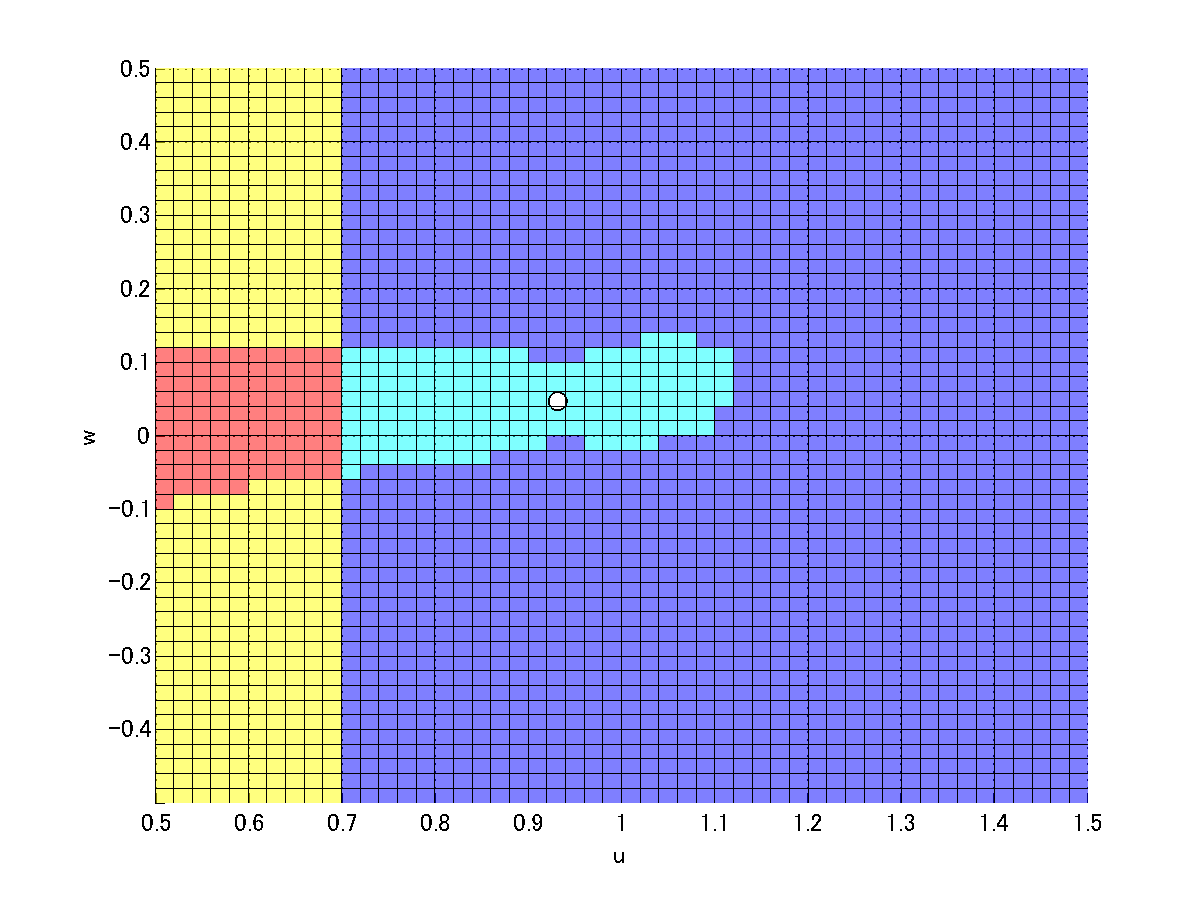}
(c)
\end{minipage}
\caption{Domains $D_i\cap \{v=0\}$, $i=1,2,3$, and contours of Lyapunov functions.}
\label{fig-FNfp-V3}
\begin{flushleft}
(a) shows $D_1\cap \{v=0\}$. 
\par
(b) shows $D_2^\ast\cap \{v=0\}$. Black curves represent contours $\{L_2=-0.4, -0.2, 0\}$.
\par
(c) shows $D_3\cap \{v= 0\}$. 
\end{flushleft}
\end{figure}

\begin{figure}[htbp]\em
\begin{minipage}{0.32\hsize}
\centering
\includegraphics[width=5.0cm]{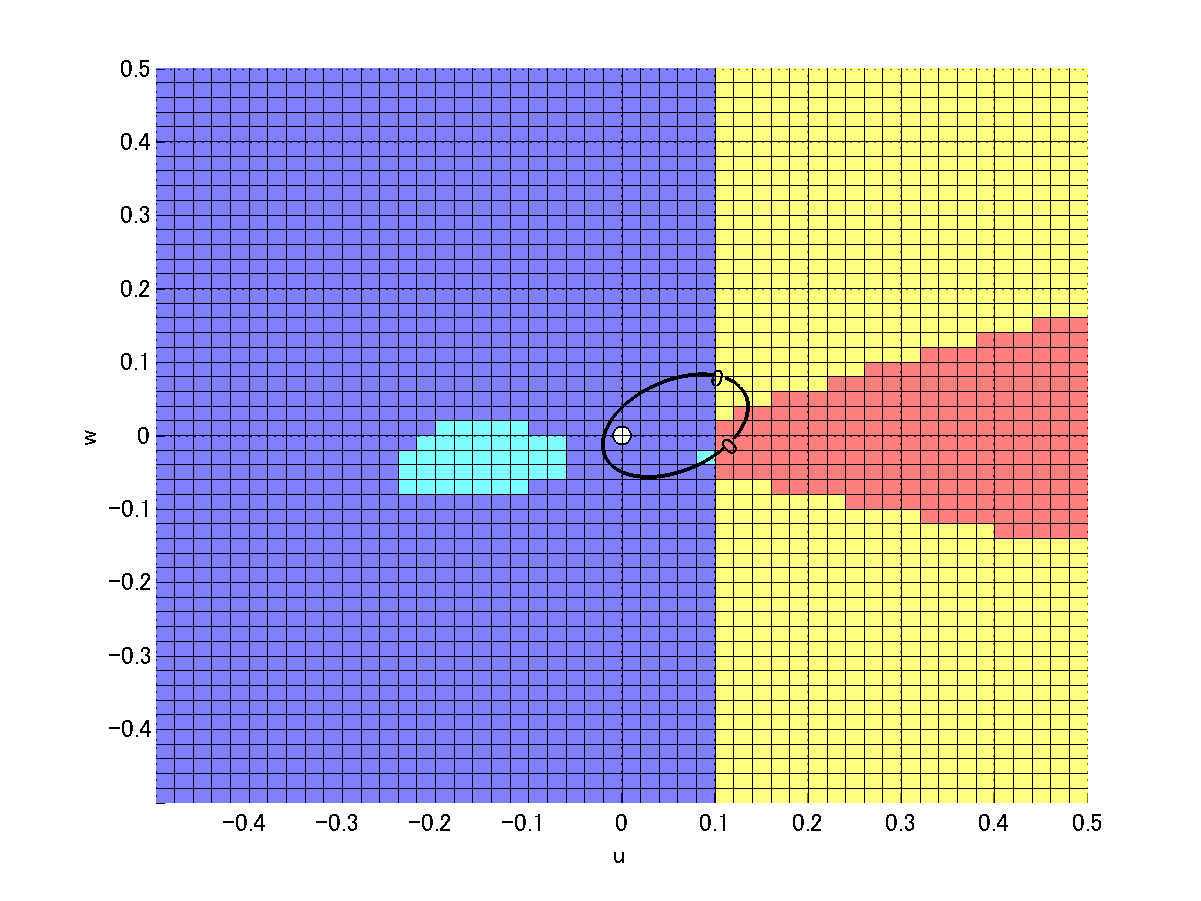}
(a)
\end{minipage}
\begin{minipage}{0.32\hsize}
\centering
\includegraphics[width=5.0cm]{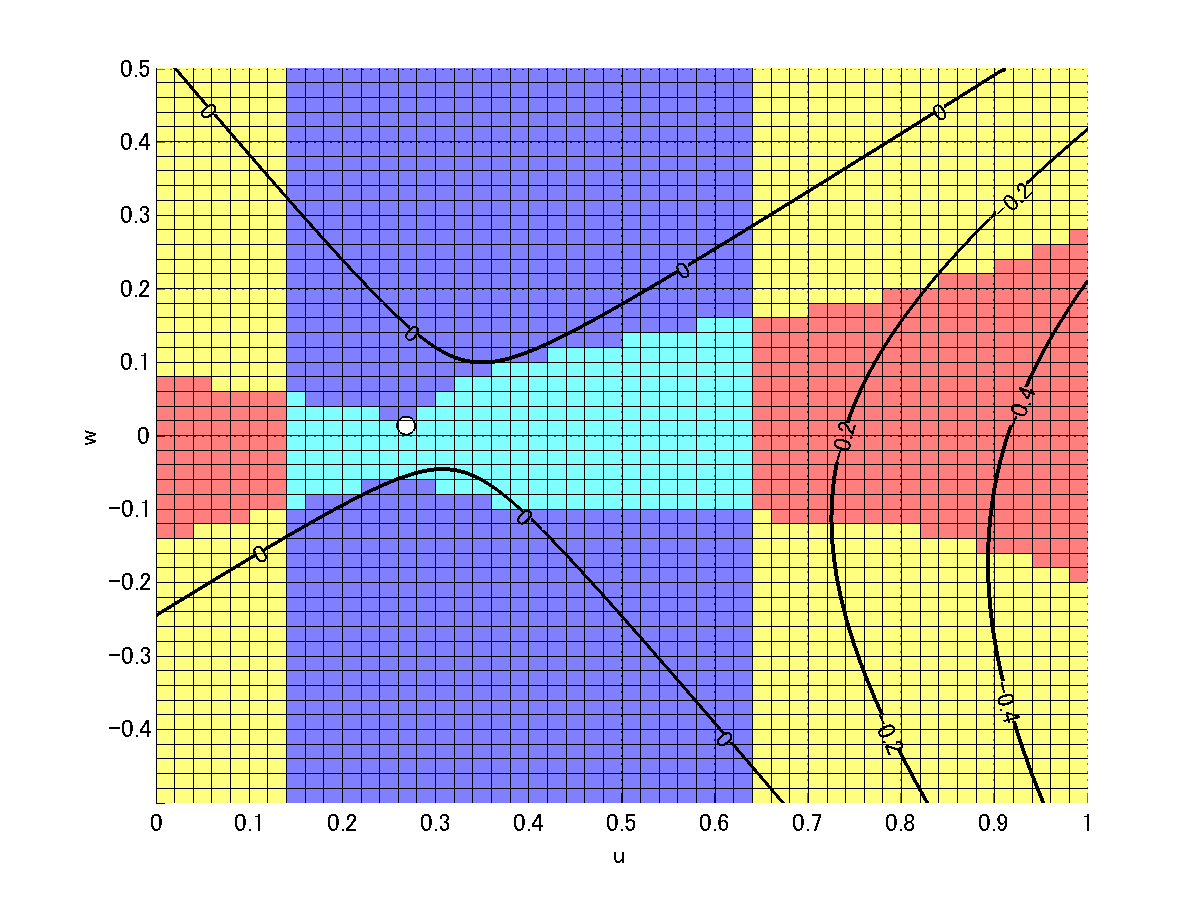}
(b)
\end{minipage}
\begin{minipage}{0.32\hsize}
\centering
\includegraphics[width=5.0cm]{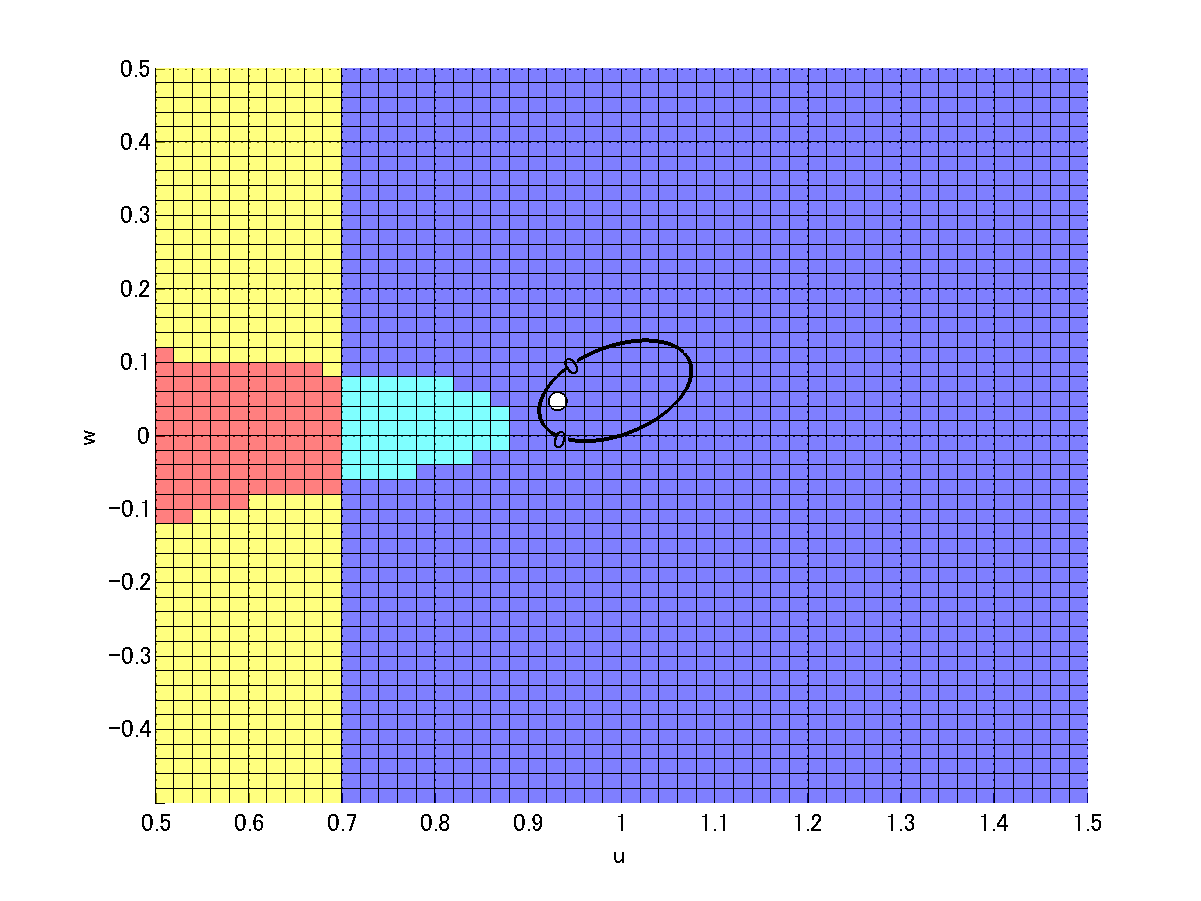}
(c)
\end{minipage}
\caption{Domains $D_i\cap \{v=0.05\}$, $i=1,2,3$, and contours of Lyapunov functions.}
\label{fig-FNfp-V4}
\begin{flushleft}
(a) shows $D_1\cap \{v=0.05\}$. Black curve represents the contour $\{L_1= 0\}$.
\par
(b) shows $D_2^\ast\cap \{v=0.05\}$. Black curves represent contours $\{L_2=-0.4, -0.2, 0\}$.
\par
(c) shows $D_3\cap \{v=0.05\}$. Black curve represents the contour $\{L_3= 0\}$.
\end{flushleft}
\end{figure}

\begin{figure}[htbp]\em
\begin{minipage}{0.32\hsize}
\centering
\includegraphics[width=5.0cm]{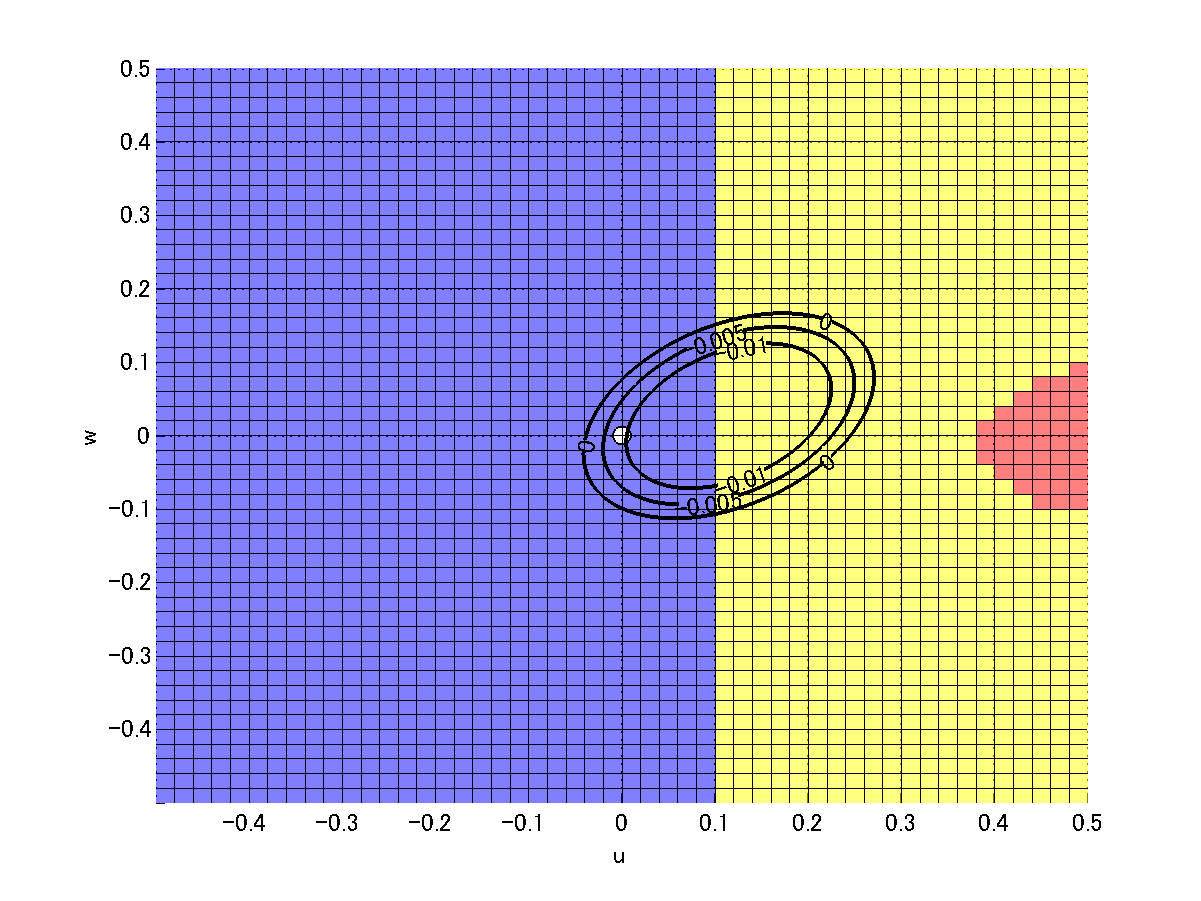}
(a)
\end{minipage}
\begin{minipage}{0.32\hsize}
\centering
\includegraphics[width=5.0cm]{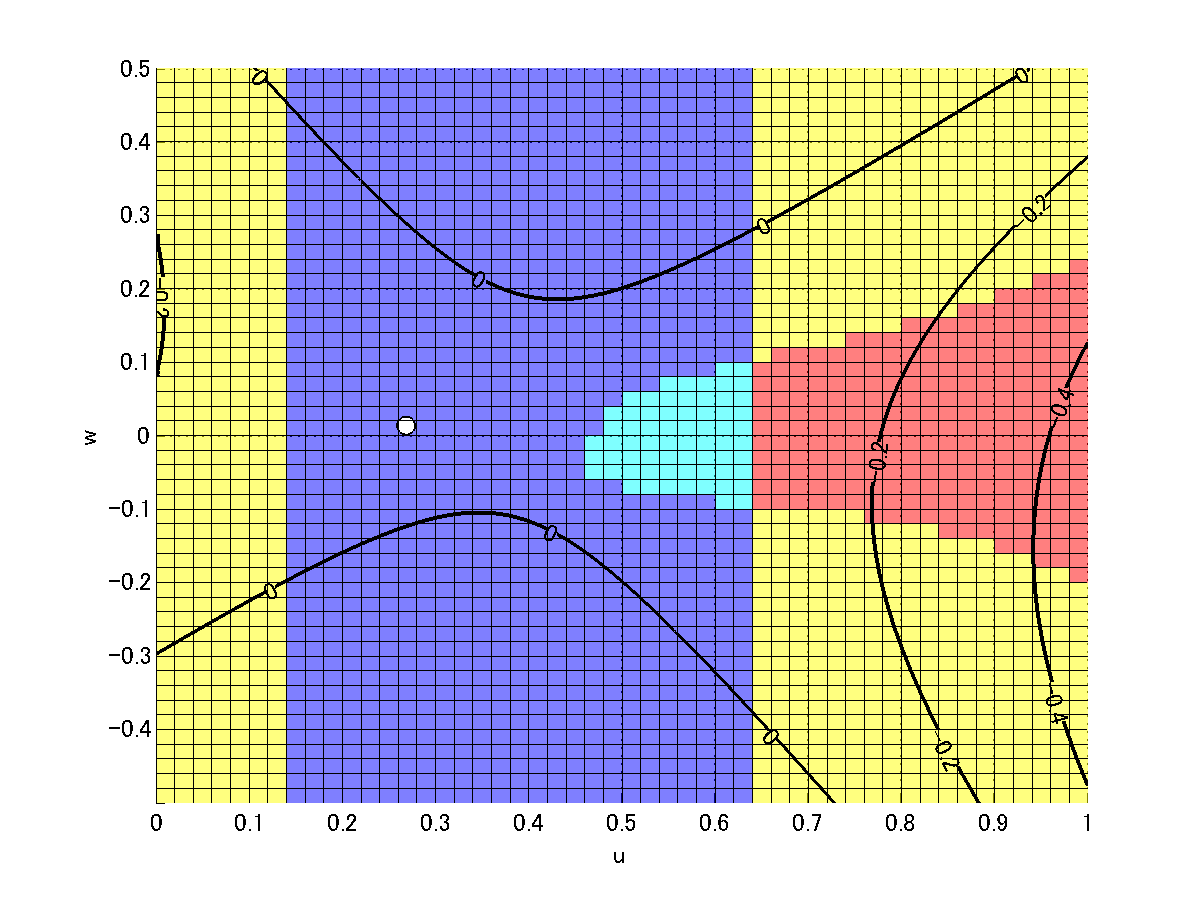}
(b)
\end{minipage}
\begin{minipage}{0.32\hsize}
\centering
\includegraphics[width=5.0cm]{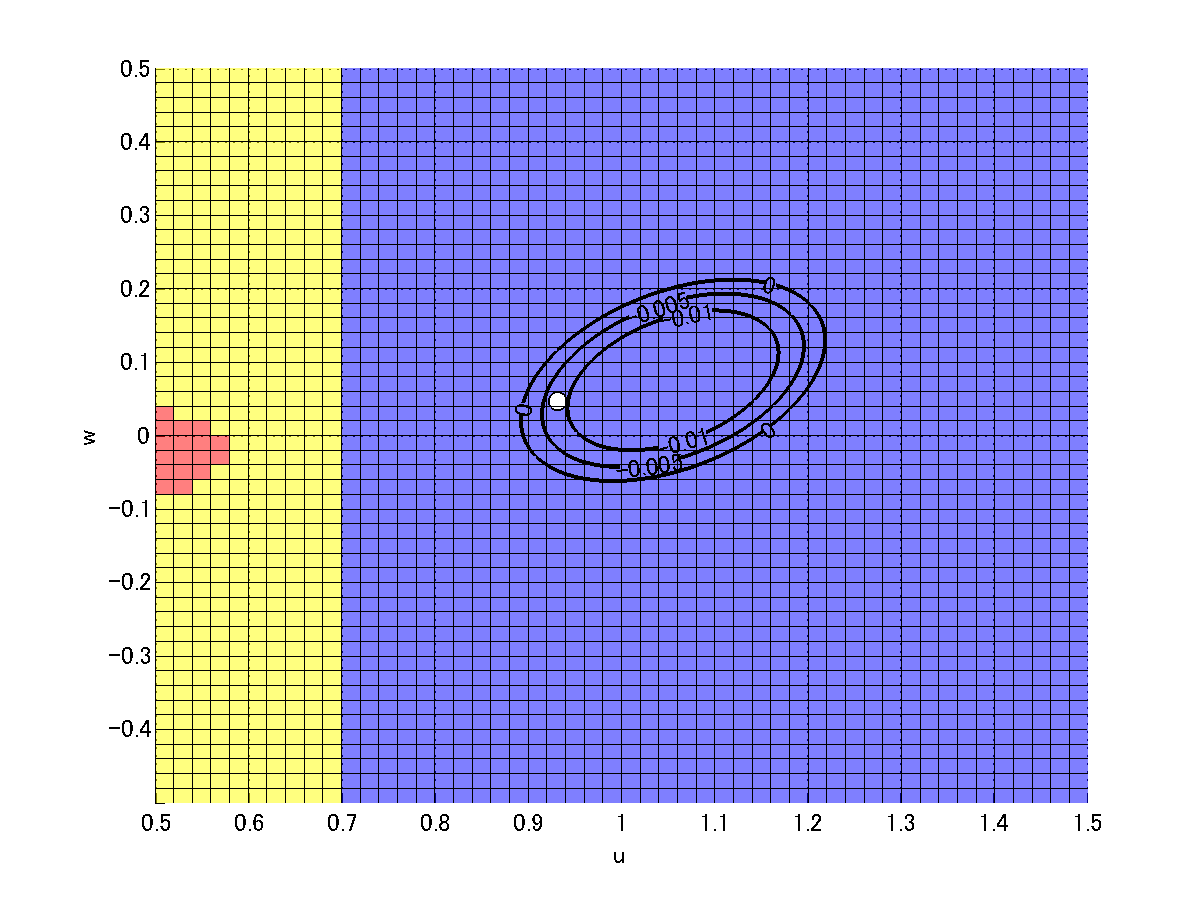}
(c)
\end{minipage}
\caption{Domains $D_i\cap \{v=0.1\}$, $i=1,2,3$, and contours of Lyapunov functions.}
\label{fig-FNfp-V5}
\begin{flushleft}
(a) shows $D_1\cap \{v=0.1\}$. Black curves represent contours $\{L_1=-0.01, -0.005, 0\}$.
\par
(b) shows $D_2^\ast\cap \{v=0.1\}$. Black curves represent contours $\{L_2=-0.4, -0.2, 0\}$.
\par
(c) shows $D_3\cap \{v=0.1\}$. Black curves represent contours $\{L_3=-0.01, -0.005, 0\}$.
\end{flushleft}
\end{figure}

\bigskip
Looking at Fig. \ref{fig-FNfp-V4}-(a), it turns out that red cubes are contained in the interior of $L_1^{-1}(0)$; namely, $L_1^{-1}(0,\infty)$. 
This can be also seen in Fig. \ref{fig-FNfp}-(a).
We then divide the domain $[0,0.4]\times [-0.2,0.2]\times [-0.2,0.2]$ into $200\times 200\times 200$ small uniform cubes and verify the strict negative definiteness of $A({\bf z})$ in (\ref{matrix-neg-def}) on these smaller cubes again.
Fig. \ref{fig-FNReVal} shows validation results in this setting, which shows that the Lyapunov domain in $D_1$ is extended.
More precisely, yellow area corresponding to succeeded domains in Stage 2 becomes larger.

\begin{figure}[htbp]\em
\begin{minipage}{0.32\hsize}
\centering
\includegraphics[width=5.0cm]{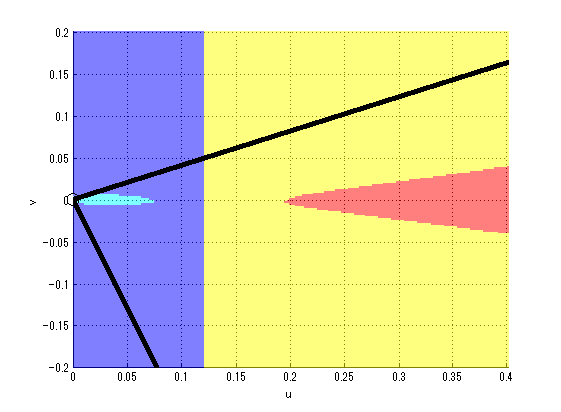}
(a)
\end{minipage}
\begin{minipage}{0.32\hsize}
\centering
\includegraphics[width=5.0cm]{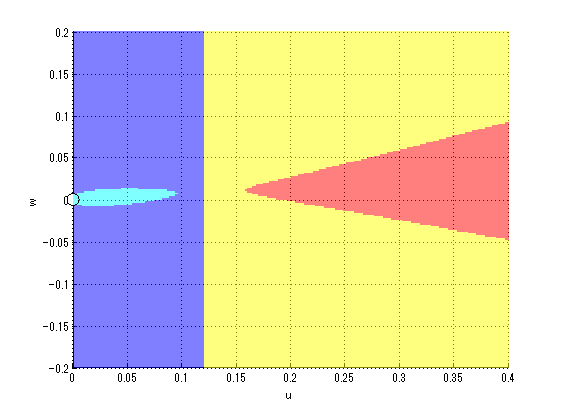}
(b)
\end{minipage}
\begin{minipage}{0.32\hsize}
\centering
\includegraphics[width=5.0cm]{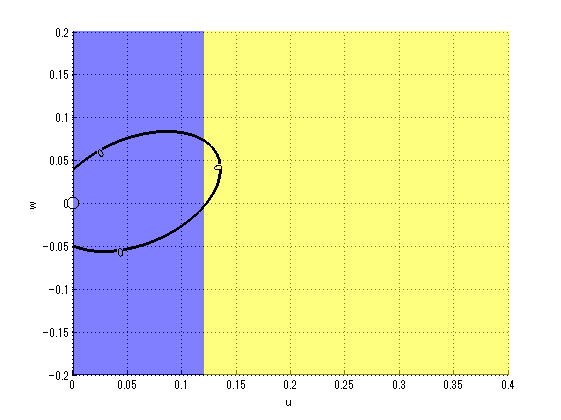}
(c)
\end{minipage}
\caption{Refinement of the region $[0,0.4]\times [-0.2,0.2]^2$ in Fig. \ref{fig-FNfp-V4}-(a). }
\label{fig-FNReVal}
\begin{flushleft}
(a) shows $D_1\cap \{w=0\}$. Black lines represent contours $\{L_1=0\}$.
\par
(b) shows $D_1^\ast\cap \{v=0\}$. %Black curves represent $\{L_2=0\}$.
\par
(c) shows $D_1\cap \{v=0.05\}$. Black curve represents the contour $\{L_3=0\}$.
\end{flushleft}
\end{figure}

\bigskip
Finally, note that there are several regions where only validations in Stage 2 are succeeded.
Lyapunov domains containing such domains include geometric cones (e.g. yellow domains in Figs. \ref{fig-FNfp}-(c) and \ref{fig-FNfp-V1}-(c)), but we cannot apply the theory of dynamical cone $Q$ to asymptotic behavior in those domains in general.

\subsection{Validation of ${\bf m}$-Lyapunov functions}

Next we construct ${\bf m}$-Lyapunov functions around ${\bf x}_1^\ast$ and compare Lyapunov domains with those of $L_1$.
Let $M^\ast = {\rm diag}\{i_1,\cdots, i_n\}$ be given by
\begin{align*}
i_j=\left\{\begin{array}{ccc}
m_j, & \text{if} & Re\left(\lambda_j\right)<0 \\ -m_j, & \text{if} & Re\left(\lambda_j\right)>0.
\end{array}\right.
\end{align*}
In our current validations, we set (i) $m_1 = 10, m_2 = 1$, and (ii) $m_1 = 1, m_2 = 10$.

Figs. \ref{fig-FNM1} - \ref{fig-FNM3} show our validation results for Lyapunov domains of $L_1$, (i) and (ii), respectively.
Note that the matrix $I^\ast$ with (i) makes the stable cone sharper.
Similarly, the matrix $I^\ast$ with (ii) makes the unstable cone sharper.

Figs. \ref{fig-FNM1}-(b) and \ref{fig-FNM2}-(b) imply that the stable cone contains more red cubes if it becomes sharper.
Similarly, Figs. \ref{fig-FNM1}-(c) and \ref{fig-FNM2}-(c) imply that there are many red cubes on the right side of equilibria.
${\bf m}$-Lyapunov functions actually sharpen enclosures of stable and unstable manifolds, but corresponding Lyapunov domains tend to be smaller than those for ordinary Lyapunov functions, even close to equilibria.

\begin{figure}[htbp]\em
\begin{minipage}{0.32\hsize}
\centering
\includegraphics[width=5.0cm]{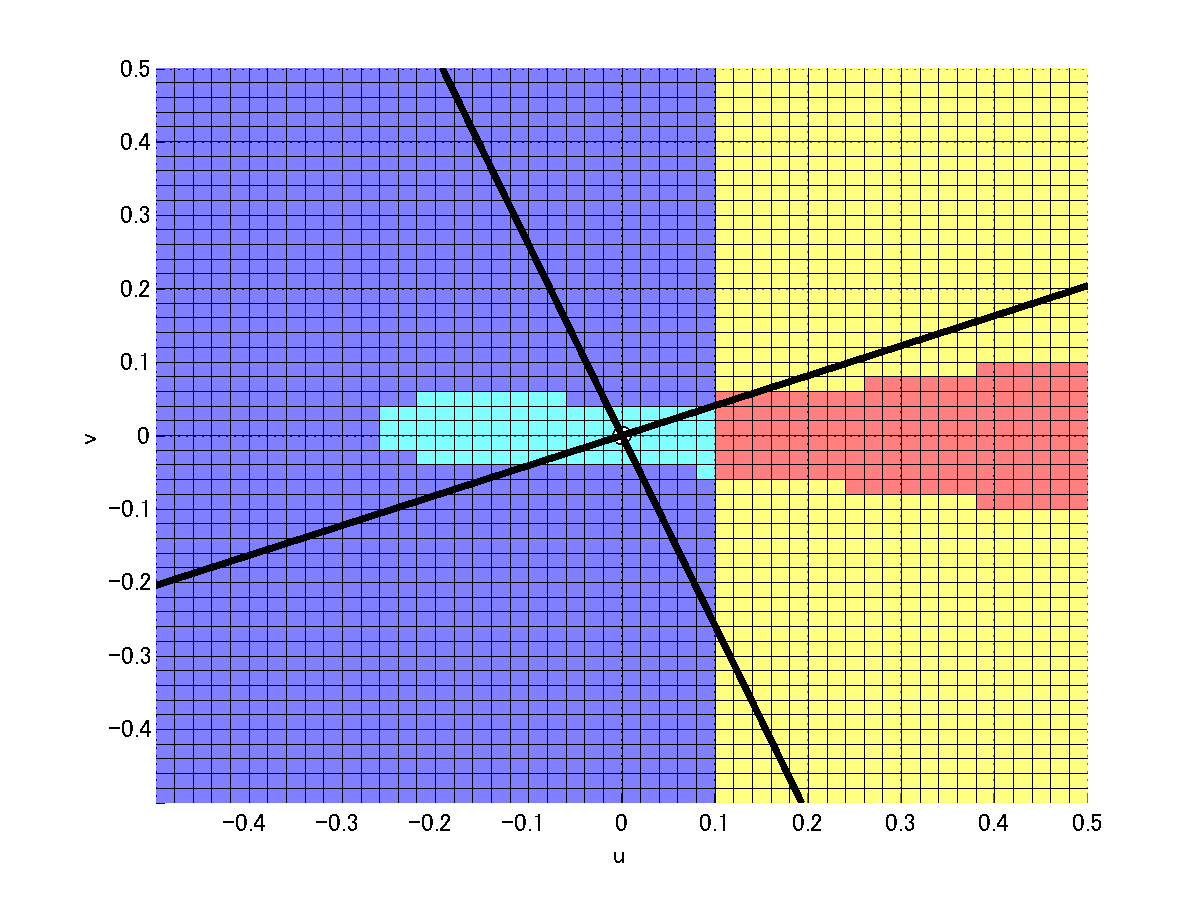}
(a)
\end{minipage}
\begin{minipage}{0.32\hsize}
\centering
\includegraphics[width=5.0cm]{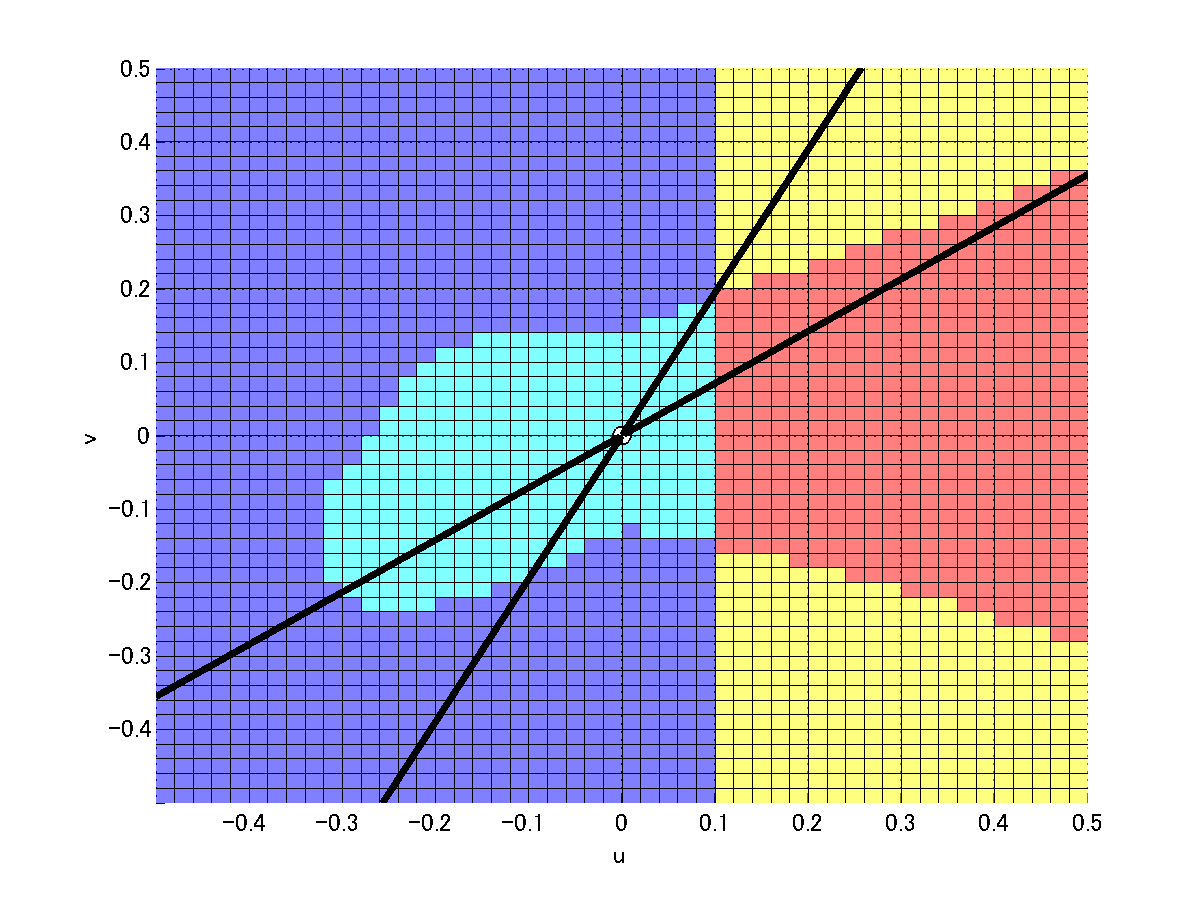}
(b)
\end{minipage}
\begin{minipage}{0.32\hsize}
\centering
\includegraphics[width=5.0cm]{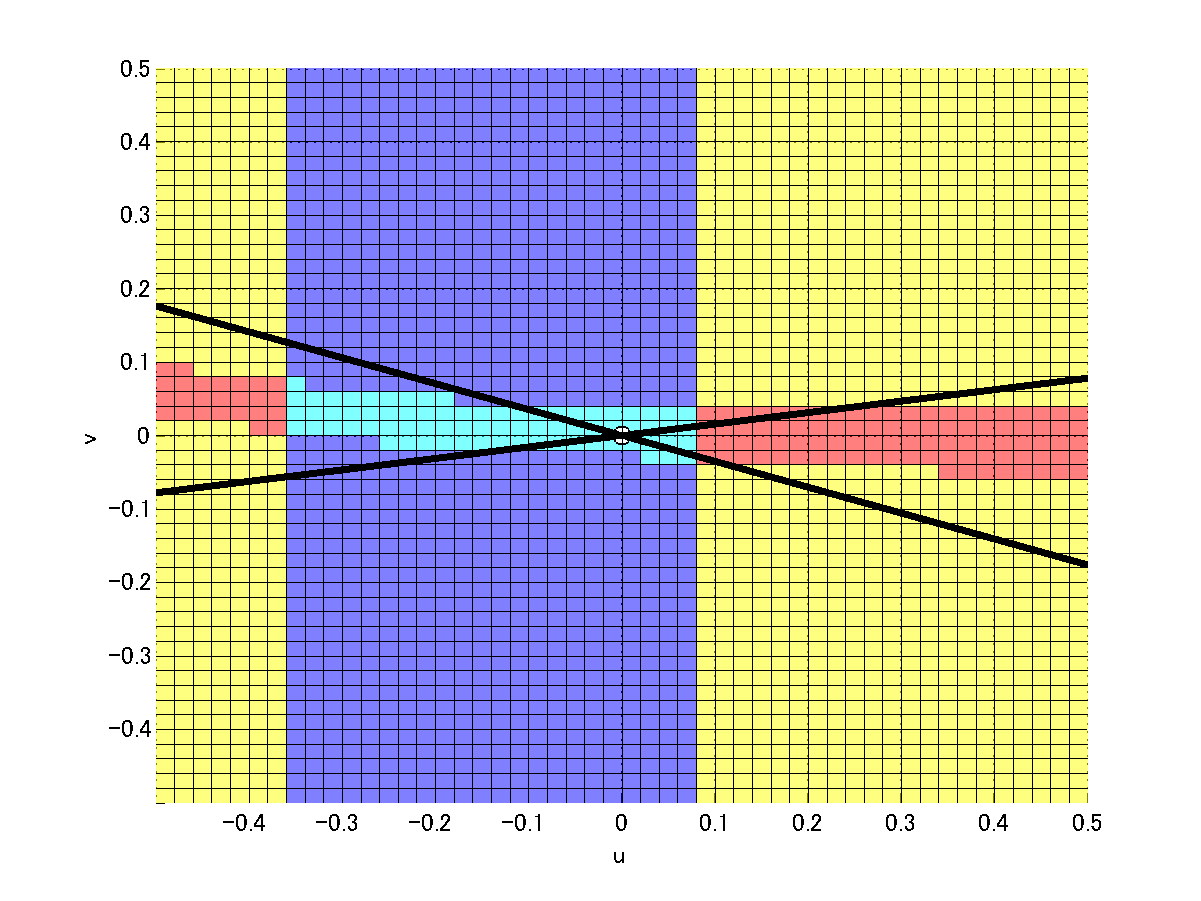}
(c)
\end{minipage}
\caption{Validation of ${\bf m}$-Lyapunov functions: 1.}
\label{fig-FNM1}
\begin{flushleft}
(a) shows $D_1\cap \{w=0\}$. Black lines represent contours $\{L_1=0\}$.
\par
(b) shows $D_1\cap \{w=0\}$. Black lines represent contours $\{L_{M^\ast}=0\}$, where $m_1 = 10, m_2 = 1$, which implies that the stable cone becomes sharper.
\par
(c) shows $D_1\cap \{w=0\}$. Black lines represent contours $\{L_{M^\ast}=0\}$, where $m_1 = 1, m_2 = 10$, which implies that the unstable cone becomes sharper.
\end{flushleft}
\end{figure}

\begin{figure}[htbp]\em
\begin{minipage}{0.32\hsize}
\centering
\includegraphics[width=5.0cm]{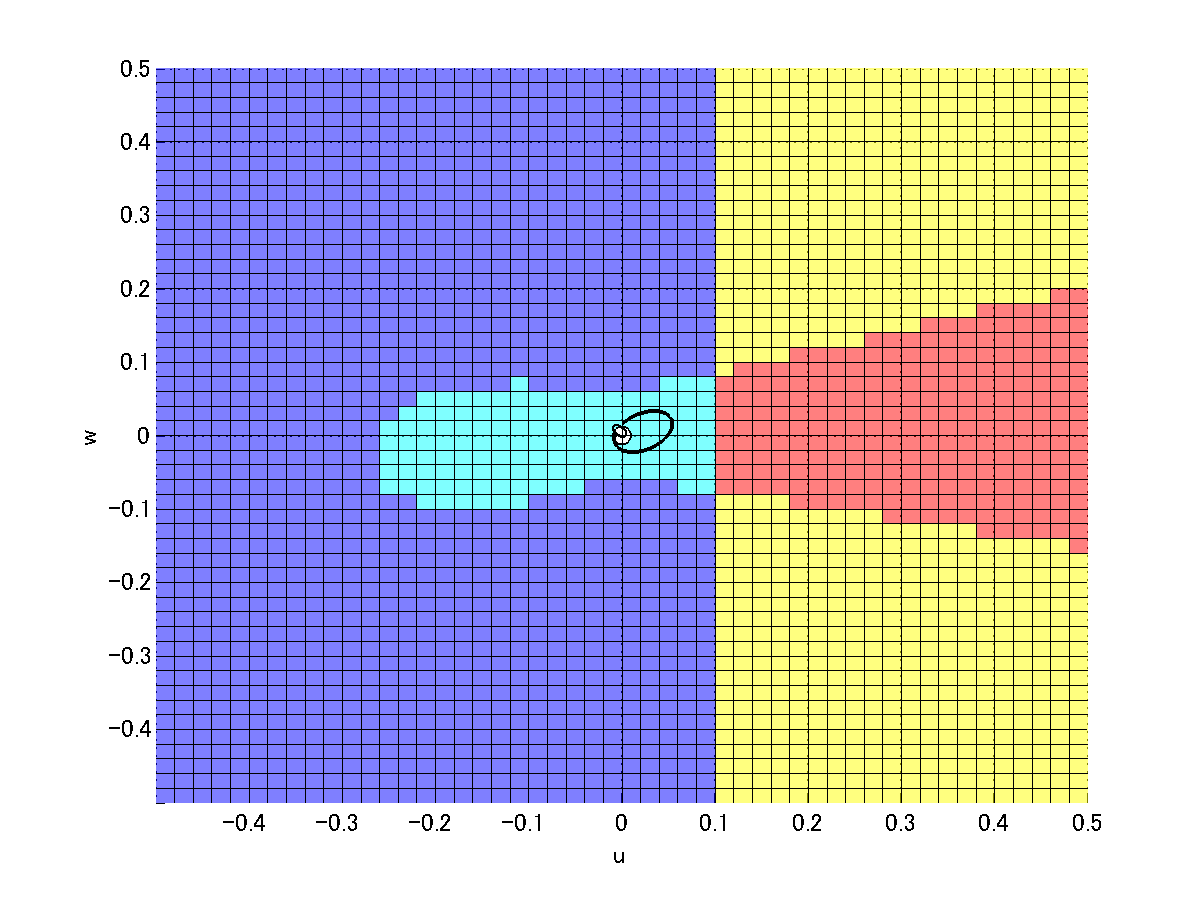}
(a)
\end{minipage}
\begin{minipage}{0.32\hsize}
\centering
\includegraphics[width=5.0cm]{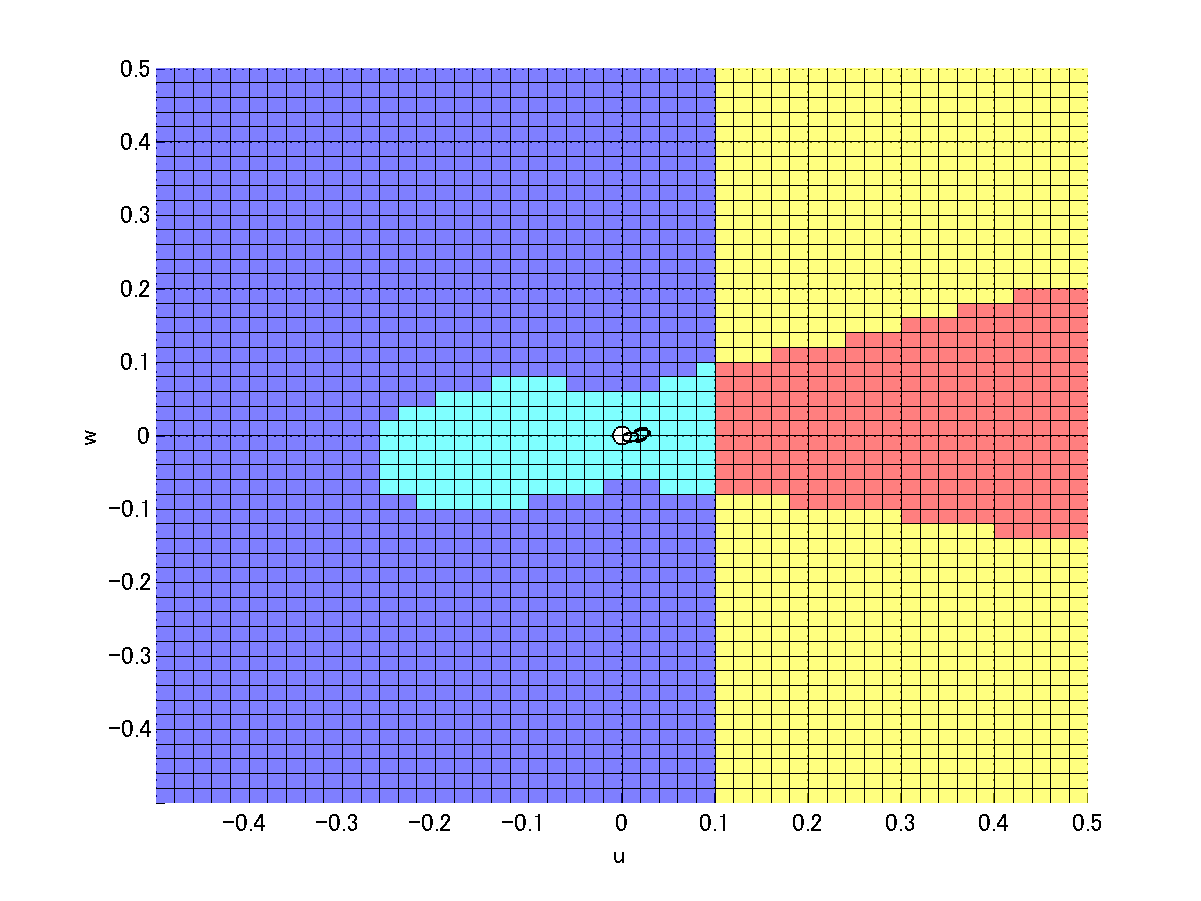}
(b)
\end{minipage}
\begin{minipage}{0.32\hsize}
\centering
\includegraphics[width=5.0cm]{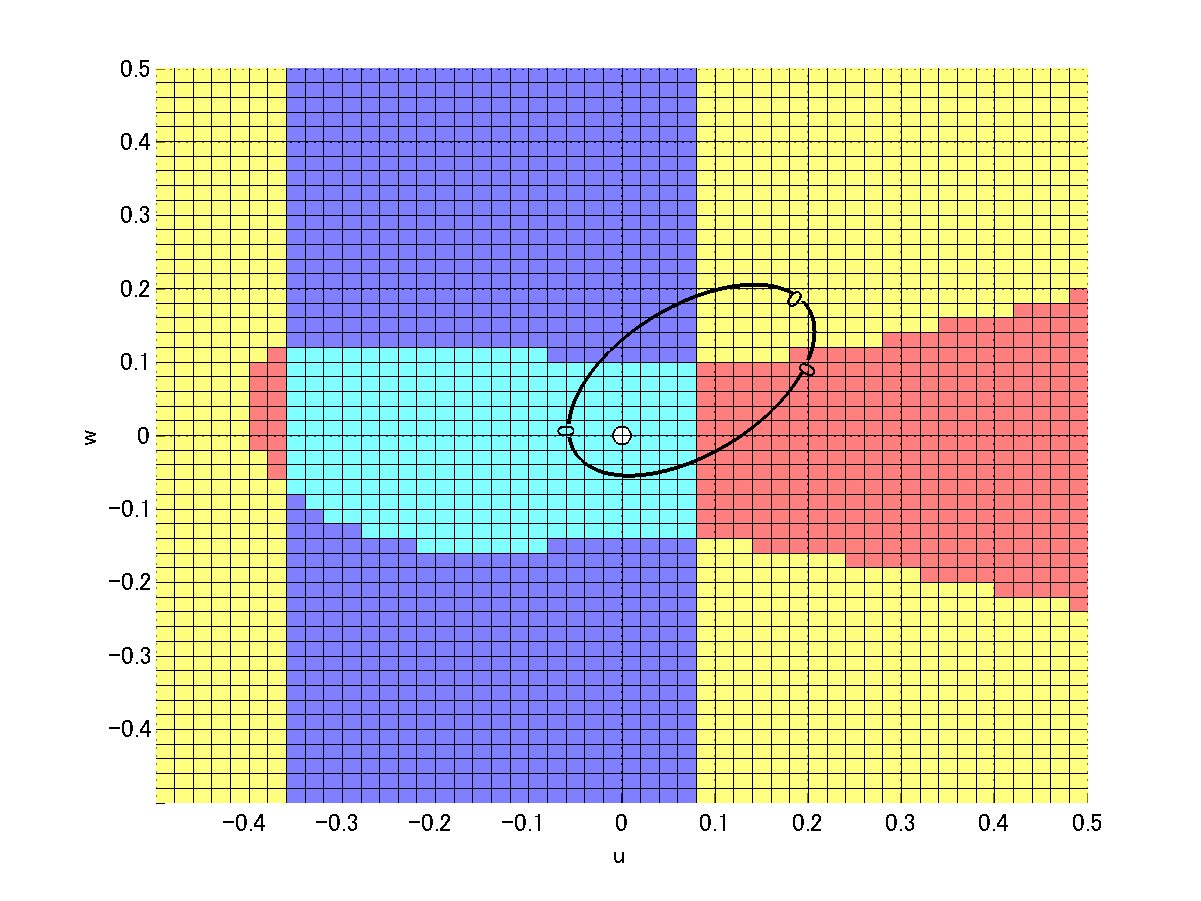}
(c)
\end{minipage}
\caption{Validation of ${\bf m}$-Lyapunov functions: 2.}
\label{fig-FNM2}
\begin{flushleft}
(a) shows $D_1\cap \{v=0.02\}$. Black curve represents the contour $\{L_1=0\}$.
\par
(b) shows $D_1\cap \{v=0.02\}$. Black curve near the origin represents the contour $\{L_{M^\ast}=0\}$, where $m_1 = 10, m_2 = 1$.
\par
(c) shows $D_1\cap \{v=0.02\}$. Black curve represents the contour $\{L_{M^\ast}=0\}$, where $m_1 = 1, m_2 = 10$.
\end{flushleft}
\end{figure}

\begin{figure}[htbp]\em
\begin{minipage}{0.32\hsize}
\centering
\includegraphics[width=5.0cm]{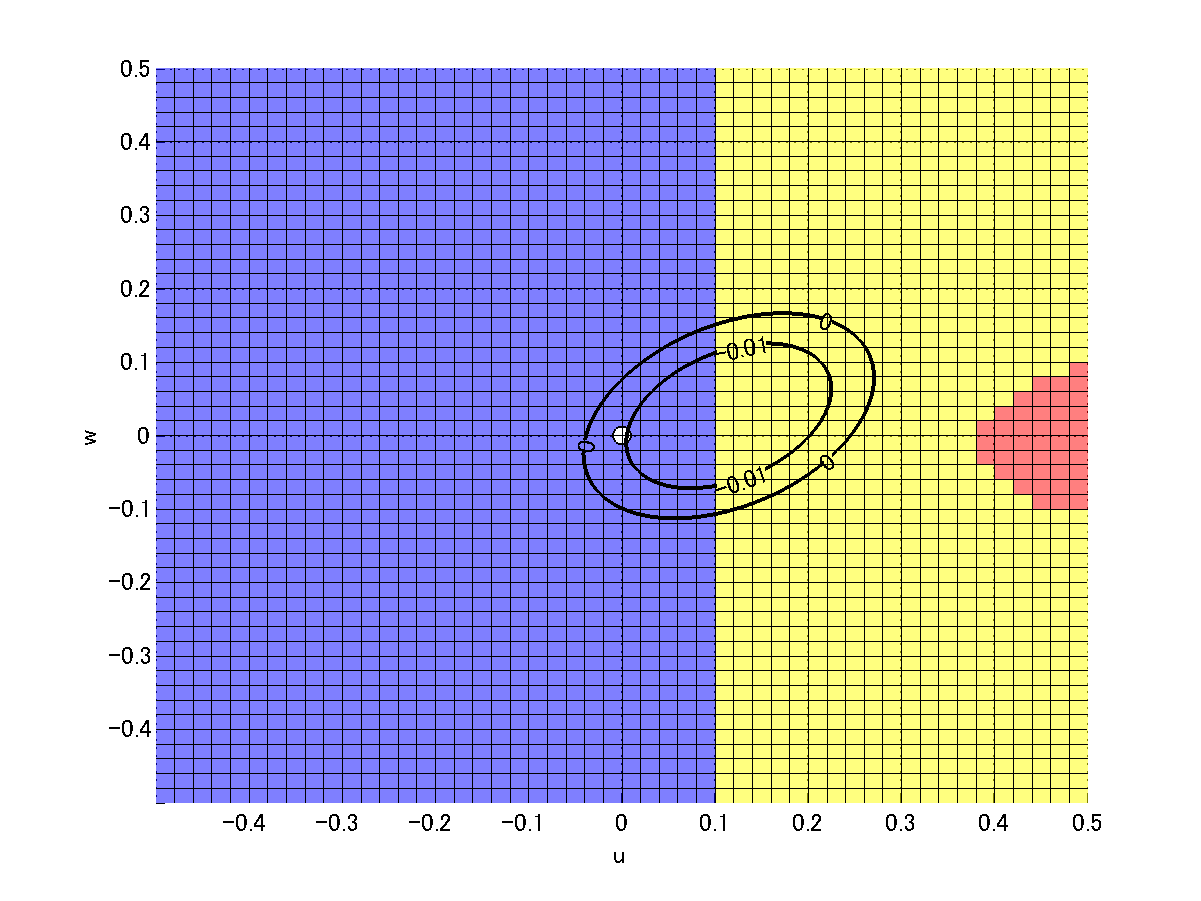}
(a)
\end{minipage}
\begin{minipage}{0.32\hsize}
\centering
\includegraphics[width=5.0cm]{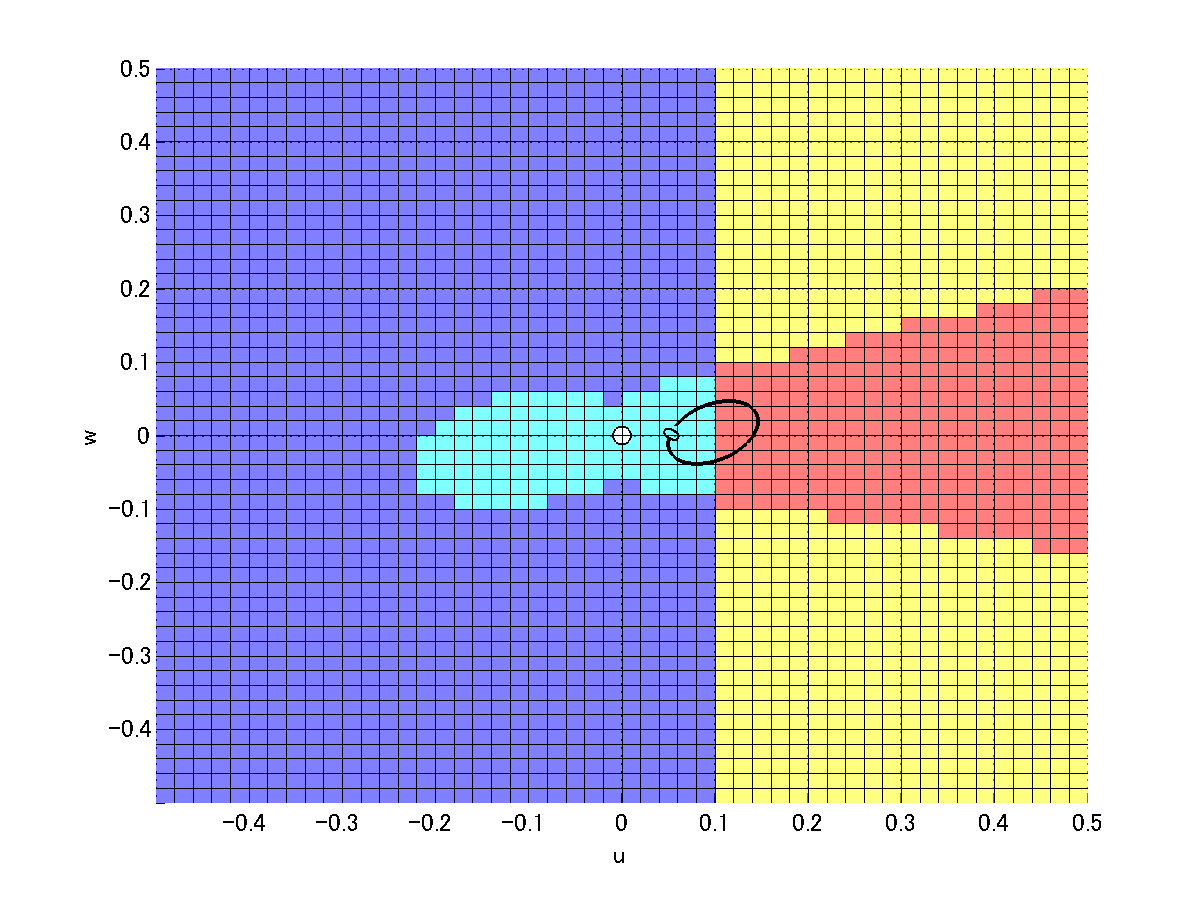}
(b)
\end{minipage}
\begin{minipage}{0.32\hsize}
\centering
\includegraphics[width=5.0cm]{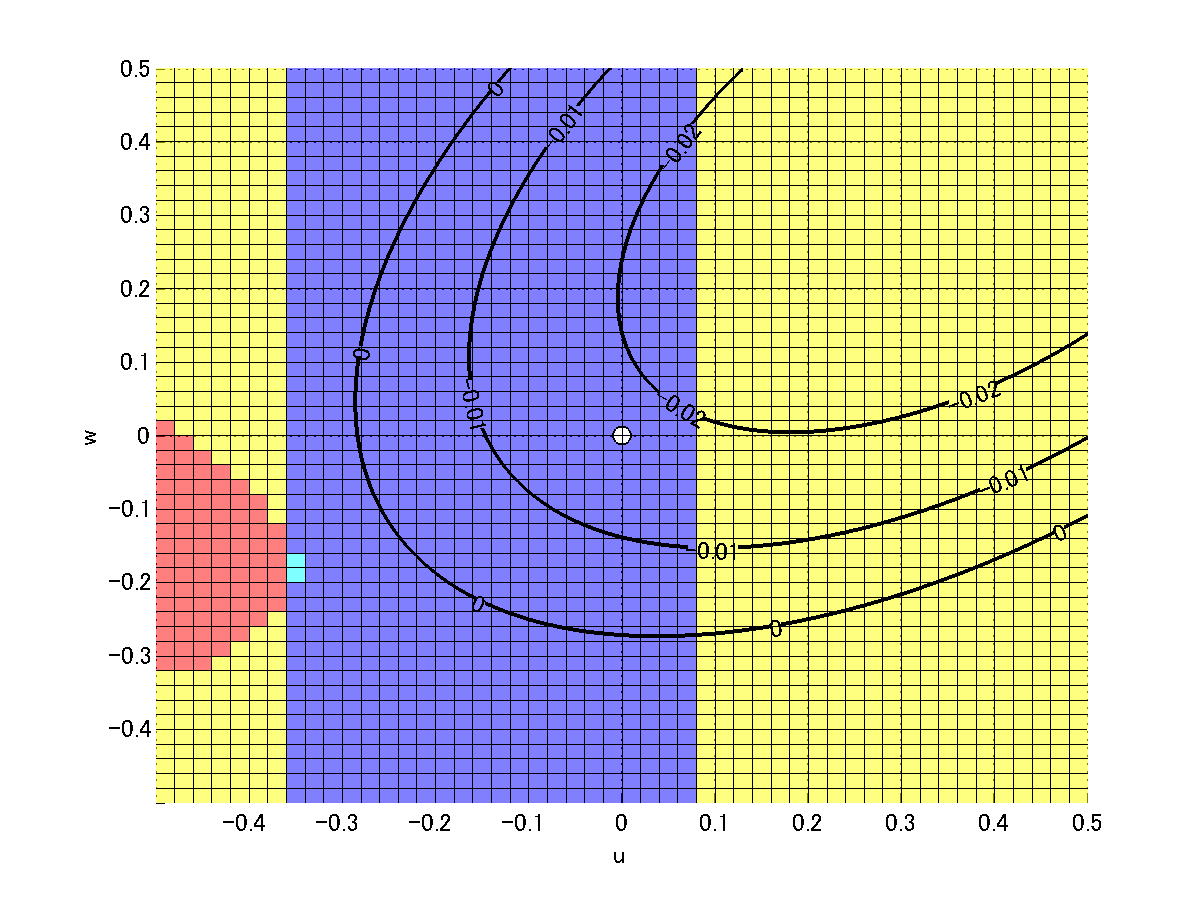}
(c)
\end{minipage}
\caption{Validation of ${\bf m}$-Lyapunov functions: 3.}
\label{fig-FNM3}
\begin{flushleft}
(a) shows $D_1\cap \{v=0.1\}$. Black curves represent contours $\{L_1=-0.01,0\}$.
\par
(b) shows $D_1\cap \{v=0.1\}$. Black curve represents the contour $\{L_{M^\ast}=0\}$, where $m_1 = 10, m_2 = 1$.
\par
(c) shows $D_1\cap \{v=0.1\}$. Black curves represent contours $\{L_{M^\ast}=-0.02, -0.01, 0\}$, where $m_1 = 1, m_2 = 10$.
\end{flushleft}
\end{figure}

\bigskip
Computation times for our verifications are about $11$ minutes for $50\times 50\times 50$ small uniform cubes, and about $14.5$ hours for $200\times 200\times 200$ smaller uniform cubes shown in Fig. \ref{fig-FNReVal}.

% New section	
\section{Numerical examples for Poincar\'{e} maps}
\label{section-example-discrete}

We move to validations of Lyapunov domains for discrete dynamical systems.
In particular, we deal with Poincar\'{e} maps for flows on Poincar\'{e} sections as our test problems.

% New subsection
\subsection{Remarks on verification of the negative definiteness of $B({\bf x})$}

In practical verification of the negative definiteness of $B({\bf x})$ in (\ref{matrix-discrete}), 
we should pay attention to differences from the case of flows.

Let $D_L$ be a star-shaped domain containing a fixed point ${\bf x}^\ast$ of a map $\psi$ with a decomposition $D_L = \bigcup_{k=1}^K D_k$.
For each point ${\bf x}\in D_k$, the path $\{{\bf x}^\ast + s({\bf x}-{\bf x}^\ast)\mid 0\leq s\leq 1\}$ is contained in $D_L= \bigcup_{k=1}^{K} D_k$.
By using the integral form of $A_I$, the matrix $B({\bf x})$ can be written by
\begin{align*}
B({\bf x}) &= A_I({\bf x})^T Y A_I({\bf x}) - Y\\
	&= \int_0^1 \{ D\psi ({\bf x}^\ast + s({\bf x}-{\bf x}^\ast))^T Y A_I({\bf x}) - Y\}ds\\
	&= \int_0^1ds \int_0^1ds' \{ D\psi ({\bf x}^\ast + s({\bf x}-{\bf x}^\ast))^T Y D\psi ({\bf x}^\ast + s'({\bf x}-{\bf x}^\ast)) - Y\}\\
	&= \int_0^1ds \int_0^1ds'  \tilde B({\bf x};s,s').
\end{align*}
These equalities indicate that the negative definiteness of $B({\bf x})$ is reduced to that of the matrix $\tilde B({\bf x};s,s')$ for all $s,s'\in [0,1]$, as in the case of flows.
Since both segments $\{{\bf x}^\ast + s({\bf x}-{\bf x}^\ast)\}$ and $\{{\bf x}^\ast + s'({\bf x}-{\bf x}^\ast)\}$ are contained in $D_L$, then we know that the matrix $\tilde B({\bf x};s,s')$ is contained in the following set of matrices for any ${\bf x}\in D_k$ for all $k=1,\cdots, K$:
\begin{equation*}
\{\tilde B({\bf z}, {\bf z}') = D\psi ({\bf z})^T Y D\psi ({\bf z}') - Y \mid {\bf z}, {\bf z}'\in D_L\}.
\end{equation*}
Obviously any points ${\bf z}, {\bf z}'$ are contained in $D_k$ and $D_{k'}$ for some $k,k'\in \{1,\cdots, K\}$, respectively, while they are different in general.
We obtain the following sufficient condition of the negative definiteness of $B({\bf x})$ for all $x\in D_L$ making use of the decomposition $D_L= \bigcup_{k=1}^{K} D_k$.

\begin{lemma}
\label{lem-neg-def-map1}
Let $D_L$ be a star-shaped domain containing a fixed point ${\bf x}^\ast$ of a map $\psi$ with a decomposition $D_L = \bigcup_{k=1}^K D_k$.
Let $\tilde B({\bf z}, {\bf z}') = D\psi ({\bf z})^T Y D\psi ({\bf z}') - Y$ for ${\bf z}, {\bf z}'\in D_L$.
Enclose each subdomain $D_k$ by an interval vector $[D_k]$.
Assume that all matrices of the matrix set 
\begin{equation*}
\tilde B([D_k], [D_{k'}]) = \{\tilde B({\bf z}, {\bf z}') \mid {\bf z}\in [D_k], {\bf z}'\in [D_{k'}]\}
\end{equation*}
are negative definite for all $k,k'\in \{1,\cdots, K\}$. 
Then $B({\bf x})$ is negative definite for all ${\bf x}\in D_L$.
\end{lemma}

Using the identity
\begin{equation*}
({\bf x}-{\bf x}^\ast)^T C({\bf x}-{\bf x}^\ast) = ({\bf x}-{\bf x}^\ast)^T C^T ({\bf x}-{\bf x}^\ast) = ({\bf x}-{\bf x}^\ast)^T \frac{(C^T + C)}{2}({\bf x}-{\bf x}^\ast)
\end{equation*}
for any $n$-squared matrix $C$, we know that the verification of Lemma \ref{lem-neg-def-map1} is replaced by the negative definiteness of the symmetrized matrix set
\begin{equation}
\label{sym-matrix-B}
\check B([D_k], [D_{k'}]) = \left\{\frac{\tilde B({\bf z}, {\bf z}')^T + \tilde B({\bf z}, {\bf z}')}{2} \mid {\bf z}\in [D_k], {\bf z}'\in [D_{k'}]\right\},\quad k, k' \in \{1,\cdots, K\}.
\end{equation}
\bigskip
Taking these observations into account, we propose the following algorithm with the decomposition $D_L = \bigcup_{k=1}^K D_k$.
\begin{algorithm}
\label{alg-disc}
Let $D_L$ be a star-shaped domain containing a fixed point ${\bf x}^\ast$ of $\psi$ with a decomposition $D_L = \bigcup_{k=1}^K D_k$.
Let $[D_k]$ be the interval hull of $D_k$ (namely, the smallest interval set containing $D_k$).
\begin{enumerate}
\item For each $k, k'\in \{1,\cdots, K\}$, compute the interval matrix $\check B(k,k') := \check B([D_k], [D_{k'}])$, according to (\ref{sym-matrix-B}).
\item Compute the negative definiteness of the interval matrix $\check B(k,k')$ for all possible choices of $k, k'\in \{1,\cdots, K\}$ with, say, the Gershgorin Circle Theorem in Algorithm \ref{alg-cont}.
\end{enumerate}
If the above verifications pass, return \lq\lq succeeded". If not, return \lq\lq failed".
\end{algorithm}
Negative definiteness of all $\check B(k,k')$ yields the negative definiteness of $B({\bf x})$ for each ${\bf x}\in D_L$.
Consequently, we obtain the following corollary.
\begin{corollary}
\label{cor-subdivision-disc}
Let $D_L$ be a star-shaped domain containing a fixed point ${\bf x}^\ast$ of $\psi$.
Assume that $D_L$ admits a decomposition $D_L = \bigcup_{k=1}^{K} D_k$ into subdomains.
We also assume that there is a symmetric matrix $Y$ such that Algorithm \ref{alg-disc} returns \lq\lq succeeded".
Then the functional $L({\bf x}) = ({\bf x}-{\bf x}^\ast)^T Y ({\bf x}-{\bf x}^\ast)$ is a Lyapunov function on $D_L$.
\end{corollary}
Note that we have to verify the negative definiteness of $\check B(k,k')$ for all choices of the pair $(k,k')$ in $\{1,\cdots, K\}$, unlike the case of flows (Corollary \ref{cor-subdivision-cont}).
This is because two $A_I({\bf x})$'s are contained in a single term of $B({\bf x})$ and we have to consider integral forms with two individual parameters.

\bigskip
Now we discuss another trap for validating the negative definiteness of $B({\bf x})$. 
Since $A_I({\bf x})$ has the integral form $A_I({\bf x}) = \int_0^1 D\psi({\bf x}^\ast + s({\bf x} - {\bf x}^\ast))ds$, one can think of the mean value form of $A_I$ with interval arithmetics. 
In such a case, we obtain
\begin{equation*}
A_I({\bf x}) \in \int_0^1 D\psi \left( \bigcup_{k=1}^K [D_k]\right)ds,\quad {\bf x}\in D_L.
\end{equation*}
By using the mean value form, we obtain
\begin{equation}
\label{correct-1}
A_I({\bf x}) \in D\psi \left( \bigcup_{k=1}^K [D_k]\right).
\end{equation}
Notice that the inclusion (\ref{correct-1}) does just mean
\begin{equation}
\label{correct-2}
(A_I({\bf x}))_{i,j} \in \bigcup_{k=1}^K \{ (D\psi ({\bf z}))_{i,j}\mid {\bf z}\in [D_k]\},\quad i,j\in \{1,\cdots, n\}
\end{equation}
and {\em does not mean}
\begin{equation}
\label{wrong}
A_I({\bf x}) \in \bigcup_{k=1}^K \{ D\psi ({\bf z})\mid {\bf z}\in [D_k]\}.
\end{equation}
The trap is explained as follows.
Note that $A_I({\bf x})$ is a matrix valued function and $\int_0^1 D\psi \left( \bigcup_{k=1}^K [D_k]\right)ds$ is an interval matrix whose entries have integral forms.
In this case, the mean value form (\ref{correct-1}) is applied {\em to each entry, not to the whole matrix in general}.
Therefore, the verification making use of the wrong inclusion (\ref{wrong}) may return wrong results.
If we apply the mean value form of $A_I$ with the decomposition $D_L = \bigcup_{k=1}^{K} D_k$ and interval arithmetics, we can just apply the inclusion (\ref{correct-2}) for negative definiteness of $B({\bf x})$, which may cause the combinatorial explosion for computations.
Although Algorithm \ref{alg-disc} also makes use of the integral form of $A_I$, we should be very careful of estimation criteria for obtaining rigorous enclosures of matrices.
%% footnote
\footnote{If (\ref{wrong}) was correct, it would follow from (\ref{wrong}) that
\begin{equation*}
A_I({\bf x}) = D\psi ({\bf z}')\quad \text{ for some } {\bf z}\in [D_k]\text{ with some }k,
\end{equation*}
which would indicate that the negative definiteness of $B({\bf x})$ would hold if the matrix $D\psi ({\bf z}')^T Y D\psi ({\bf z}')-Y$ was negative definite for all ${\bf z}'\in D_L$.
This would provide an effective procedure for verifying the negative definiteness {\em if (\ref{correct-1}) implied (\ref{wrong})}, which is not the case.
}
%% footnote end

\par
\bigskip
A procedure such as Algorithm \ref{alg-disc} works to prove the existence of Lyapunov functions as long as the above algorithm passes successfully, even if domains does not contain fixed points.
An immediate benefit of Algorithm \ref{alg-disc} as well as Corollary \ref{cor-subdivision-disc} is that we can use better enclosure of matrices associated with $B$ in (\ref{matrix-discrete}).
In general, however, Algorithm \ref{alg-disc} is not an effective method in practical computations if the problem concerns with higher dimensional dynamical systems or the number of decompositions $K$ is large, which cause the combinatorial explosion of computations.
Nevertheless, if both the dimension of the problem and $K$ are small, Algorithm \ref{alg-disc} is available in reasonable computation processes.

\subsection{R\"{o}ssler system}
The first example is the R\"{o}ssler system: 
\begin{align}
\notag
\frac{du}{dt} &= -v-w,\\
\label{rossler}
\frac{dv}{dt} &= -u-av,\\
\notag
\frac{dw}{dt} &= b-w(c-u),
\end{align}

We consider the following parameters:
\begin{equation*}
a = 0.2,\quad b = 0.2,\quad c = 2.2,
\end{equation*}
in which case (\ref{rossler}) possesses an asymptotically stable periodic orbit.

Let the Poincar\'{e} section $\Gamma$ and its unit normal vector $n_\Gamma$
%% footnote
\footnote{
The setting of $n_\Gamma$ provide us with a simple coordinate on Poincar\'{e} sections.
In general, Poincar\'{e} sections for $n$-dimensional dynamical systems yield an $(n-1)$-dimensional coordinate, which is called an {\em section coordinate} (e.g. \cite{ZLoh}).
Since $n_\Gamma$ below is along an coordinate axis in the original coordinate on $\mathbb{R}^3$, the section coordinate is just the choice of two entries. 
In particular, our verifications are reduced to two dimensional dynamical systems.
For general $n_\Gamma$, we can choose an appropriate section coordinate after an affine transformation.
}
%% footnote
 be
\begin{align*}
\Gamma &= \{(u,v,w)\in \mathbb{R}^3 \mid v = -0.039538545829724\},\\
n_\Gamma &= (0,-1,0)^T.
\end{align*}
Then intervals containing the intersection point of a periodic orbit of (\ref{rossler}) with $\Gamma$ as well as the rigorous period can be validated as follows (\cite{HY}): 
\begin{align*}
\left[{\bf x}^\ast \right] &= \left(\begin{array}{c}
\left[-3.33960829479577,  -3.33960817367770\right] \\
\left[-0.03955770987867,  -0.03955770858575\right]\\
\left[ 0.03932476640565,   0.03932477007120\right]
\end{array}\right)\\
\left[T^\ast\right] &= \left[5.72694905401293,   5.72694917018763\right].
\end{align*}
Eigenvalue enclosures of the linearized matrix $DP([{\bf x}^\ast])$ of the Poincar\'{e} map $P:\Gamma\to \Gamma$ at the intersection are
\begin{align*}
[\lambda_1] &= \left[-0.55389294656450,  -0.53463411180113\right], \\
[\lambda_2] &= \left[-0.00045527197787,   0.00037278836683\right],
\end{align*}
which yield that the corresponding periodic orbit is asymptotically stable.
Note that, for computations of the Poincar\'{e} maps, we applied the $C^1$-Lohner algorithm \cite{ZLoh} with various time steps $N_t$.
The order of Taylor expansion computing the residual term is set $p=5$.
In our case, we compute the first arrival map for the Poincar\'{e} section $\Gamma$ and hence the adjustment of $N_t$ corresponds to that of the size of one time step $\Delta t$.

The symmetric matrix $Y$ computed at the center of $[{\bf x}^\ast]$ is computed as follows:
\begin{align*}
Y&=\left(\begin{array}{cc} 
  1.007925375406352  &  -0.208072133943174 \\
  -0.208072133943174 &  1.077435243689143 
\end{array}\right)
\end{align*}

Firstly, we verify the strict negative definiteness of the matrix $\check B({\bf z},{\bf z}')$ in (\ref{matrix-discrete}) around a fixed point of $P$ following Algorithm \ref{alg-disc}.
For the computation of $P$, we set $N_t = 2000$ and let $[X]$ be the interval vector on $\Gamma$ given by
\begin{align*}
\left[X\right] = \left(\begin{array}{c}
\left[-3.43960823423674,  -3.23960823423673\right] \\
\left[-0.06067523176158,   0.13932476823843\right]
\end{array}\right),
\end{align*}
which actually contains $[{\bf x}^\ast]$.
We then divide $[X]$ into $10\times 10$ small uniform squares.
Through computations following Algorithm \ref{alg-disc}, we have confirmed the strict negative definiteness, which is shown as light blue regions in Fig. \ref{fig-Rossler2000}.

\bigskip
Secondly, we validate Lyapunov domains away from the fixed point.
According to Stage 2, we verify the sign of $L(P([{\bf x}]))-L([{\bf x}])$, in which case the operation returns \lq\lq succeeded" if $L(P([{\bf x}]))-L([{\bf x}])$ is negative.
If we cannot confirm that $L(P({\bf x}))-L({\bf x}) < 0$ via interval arithmetics, then the operation returns \lq\lq failed".
Note that the result \lq\lq failed" does not mean that $L(P([{\bf x}]))-L([{\bf x}]) \geq 0$ but actually means that the resulting interval $L(P([{\bf x}]))-L([{\bf x}])$ contains $0$.

Fig. \ref{fig-Rossler2000} shows the succeeded regions colored by blue, light-blue and the failed regions colored by red with $N_t=2000$.

\begin{figure}[htbp]\em
\begin{minipage}{1\hsize}
\centering
\includegraphics[width=6.0cm]{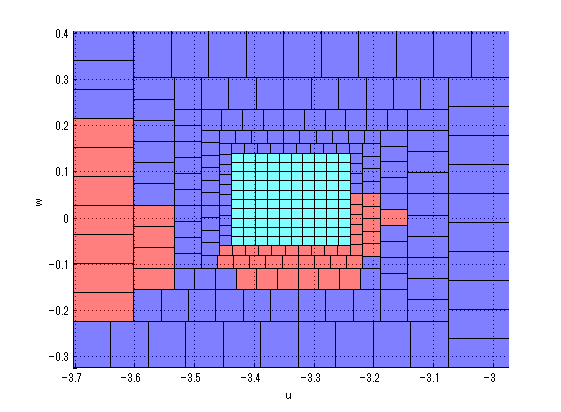}
\end{minipage}
\caption{Lyapunov domain for (\ref{rossler}) validated with $2000$ time steps.}
\label{fig-Rossler2000}
\end{figure}

\bigskip
Thirdly, we validate Lyapunov domains away from the fixed point with $N_t = 4000$.
Validation results is shown in Fig. \ref{fig-Rossler4000}.
Compared with Fig. \ref{fig-Rossler4000}, the blue region is enlarged, which implies that the accuracy of ODE computations deeply relates to constructions of Lyapunov regions.

\begin{figure}[htbp]\em
\begin{minipage}{0.4\hsize}
\centering
\includegraphics[width=6.0cm]{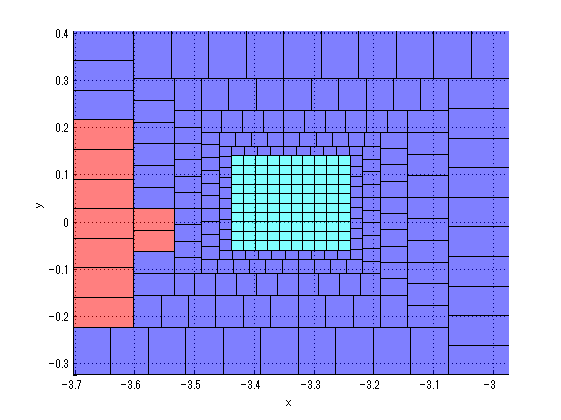}
\end{minipage}
\begin{minipage}{0.4\hsize}
\centering
\includegraphics[width=6.0cm]{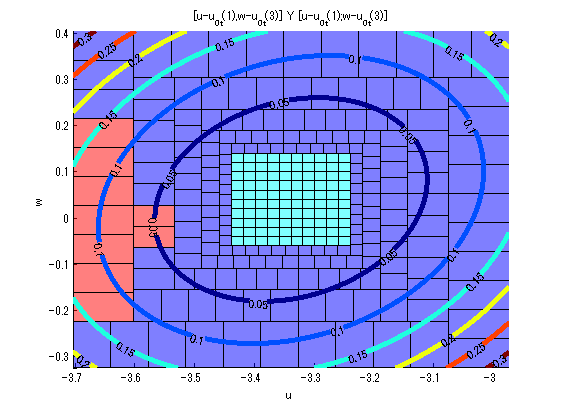}
\end{minipage}
\caption{Lyapunov domain for (\ref{rossler}) validated with $4000$ time steps and contours of the Lyapunov function.}
\label{fig-Rossler4000}
\end{figure}

\bigskip	
Nevertheless, there are several regions which our validations return failed.
To consider the origin of the failure of our validations, we numerically computed eigenvectors of $DP({\bf x}^\ast)$ at the center ${\bf x}^\ast \in [{\bf x}^\ast]$.
The result is shown in Fig. \ref{fig-RosslerEvec}.
$DP({\bf x}^\ast)$ has two eigendirections, one of which is the vector $r_1$ associated with $\lambda_1 = -0.000040978780281$ (green), and the other is  the vector $r_2$ associated with $\lambda_2 = -0.544259673645596$ (yellow).
Fig. \ref{fig-RosslerEvec} implies that the \lq\lq failed" region is clustered on the direction of $r_2$. 
Note that the modulus of $\lambda_2$ is closer to $1$ than that of $\lambda_1$, which implies that the strength of hyperbolicity 
relates to the difficulty for validating Lyapunov domains. 

\begin{figure}[htbp]\em
\begin{minipage}{1\hsize}
\centering
\includegraphics[width=6.0cm]{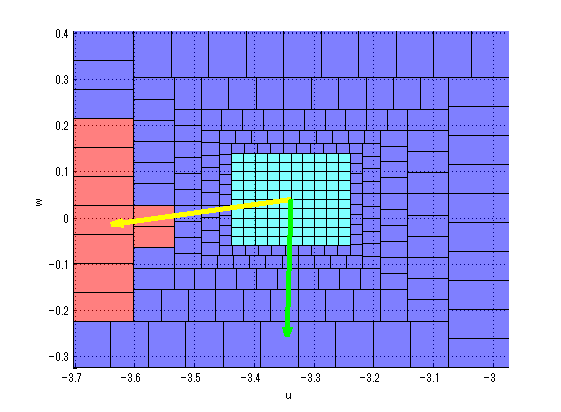}
\end{minipage}
\caption{Eigendirections of the Poincar\'{e} map for (\ref{rossler}).}
\label{fig-RosslerEvec}
\end{figure}

\bigskip
Computation times for our verifications are about $25$ hours in Stage $1$ with $C^1$-Lohner method, and about $4.8$ hours in Stage $2$ with the ordinary (i.e. $C^0$-)Lohner method.
Such a big difference of computation times is due to the difference of optimizations of programs.

%%%%%%%%%%%%%%%%%%%%%%%%%%%%%%%%%%%%%

\subsection{Lorenz system}
The second example is the Lorenz system: 
\begin{align}
\notag
\frac{du}{dt} &= -p(u+v),\\
\label{Lorenz}
\frac{dv}{dt} &= u(r-z)-v,\\
\notag
\frac{dw}{dt} &= uv-bw.
\end{align}

We set
\begin{equation*}
r = 28,\quad p = 6,\quad b = \frac{8}{3}.
\end{equation*}
The system (\ref{Lorenz}) with the above parameter values admits a periodic orbit (e.g. \cite{SV1980}), which can be validated with computer assistance that it is of saddle-type \cite{HY}.

The transformation by Sinai-Vul \cite{SV1980, SV1981} (cf. \cite{YYN}) yields the following transformed system and parameter values:
\begin{align}
\notag
\frac{du}{dt} &= a_1 u + b_1(u+v)w,\\
\label{Lorenz-trans}
\frac{dv}{dt} &= a_2 v - b_1(u+v)w,\\
\notag
\frac{dw}{dt} &= -a_3w + (u+v)(b_2 u + b_3 v),
\end{align}
\begin{equation}
\label{Lorenz-param-trans}
\begin{cases}
a_1=  9.700378782,\quad a_2=-16.700378782,\quad a_3=  2.666666667, &\\
b_1=-0.227266206,\quad b_2= 2.616729797,\quad b_3=-1.783396463. &
\end{cases}
\end{equation}
We regard (\ref{Lorenz-trans})-(\ref{Lorenz-param-trans}) as the center of our considerations. 

A Poincar\'{e} section $\Gamma$ as well as its unit normal vector $n_\Gamma$ in our verification are set as
\begin{equation*}
\Gamma = \{(u,v,w)\in \mathbb{R}^3 \mid w = 27\},\quad n_\Gamma = (0,0,1)^T.
\end{equation*}

Numerical validation of the Poincar\'{e} map $P$ based on discussions in Section \ref{section-valid-per} yields enclosures of the intersection point between the periodic orbit $\gamma$ and $\Gamma$ as well as the period of $\gamma$ as follows:

\begin{align*}
	\left[{\bf x}^\ast\right]&=\left(
		\begin{array}{c}
			\left[   3.50078722830696,   3.50078739732356\right] \\
			\left[   3.33033175959478,   3.33033178479988\right] \\
			\left[  27.00000000000000, 27.00000000000000\right]
		\end{array}\right)\\
	\left[T^\ast\right] &= \left[0.68991868010675,   0.68991868537750\right]
\end{align*}
Eigenvalues of the Jacobian matrix $DP([{\bf x}^\ast])$ of $P$ in $[{\bf x}^\ast]$ are enclosed by the following intervals:
\begin{align*}
\lambda_1 &= \left[ 1.03671803803776,   1.06505368292434\right]\\
\lambda_2 &= \left[-0.01460002975740,   0.01701745287359\right],
\end{align*}
which implies that the validated periodic orbit $\gamma$ is of saddle-type.

The symmetric matrix $Y$ computed at the center of $[{\bf x}^\ast]$ is computed as follows:
\begin{align*}
		Y&=\left(\begin{array}{cc} 
		-0.915485816680675 &  0.474862621686875\\
		 0.474862621686875  & 0.915485816680675
	\end{array}\right)
\end{align*}

We are then ready to validate the Lyapunov domain around $[{\bf x}^\ast]\cap \Gamma$.
Firstly, set a sample region containing $[{\bf x}^\ast]$ in $\Gamma$ as
\begin{align*}
\left[X\right] = \left(\begin{array}{c}
\left[  3.48078731281545,   3.52078731281546\right] \\
\left[  3.31033177219729,   3.35033177219730\right]
\end{array}\right)
\end{align*}
and divide $[X]$ into $4\times 4$ small uniform rectangular domains.
We then validate the strict negative definiteness of the matrix $B({\bf z})$ in (\ref{matrix-discrete}) for $P$ with $N_t = 4000$ in the whole region $[X]$ following Algorithm \ref{alg-disc}, which is shown in light blue regions in Fig. \ref{fig-Lorenz}.

\bigskip
Secondly, we extend the Lyapunov region outside $[X]$ by verifying if $L(P({\bf x})) - L({\bf x})$ becomes negative in each interval region $[{\bf x}]$ with $N_t = 2000$.
Fig. \ref{fig-Lorenz} around light blue region represents regions where our validation returns \lq\lq succeeded" (blue) and those our validation returns \lq\lq failed" (red).

In Fig. \ref{fig-LorenzEvec}, we can see correspondence between red regions and the arrow colored yellow denoting the eigenvector of $DP([{\bf x}^\ast])$ associated with the eigenvalue $\lambda_1 = -1.050919958691306$.
Note that the green arrow in Fig. \ref{fig-LorenzEvec} denotes the eigenvector of $DP([{\bf x}^\ast])$ associated with the eigenvalue $\lambda_2 = -0.001207890459913$.

We observe that, as in the case of the R\"{o}ssler system, validations of Lyapunov domains become hard in eigendirections associated with eigenvalues whose moduli are close to $1$.
In other words, the strength of hyperbolicity corresponds to the difficulty of validations of Lyapunov domains.

\begin{figure}[htbp]\em
\begin{minipage}{0.4\hsize}
\centering
\includegraphics[width=6.0cm]{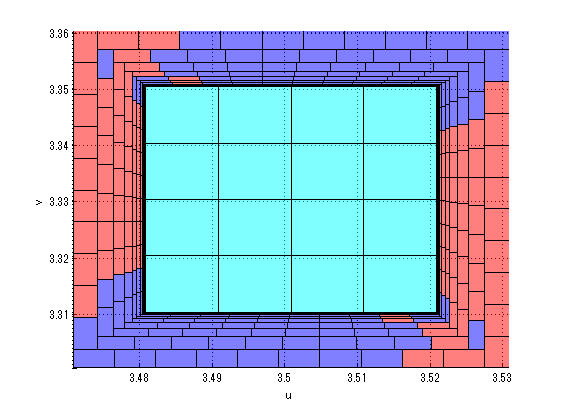}
\end{minipage}
\begin{minipage}{0.4\hsize}
\centering
\includegraphics[width=6.0cm]{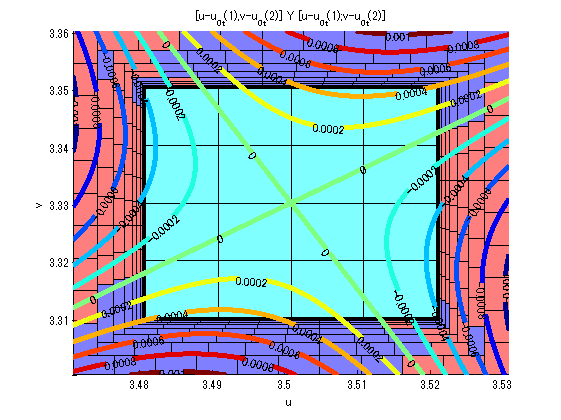}
\end{minipage}
\caption{Validated Lyapunov domain for (\ref{Lorenz-trans}) around $[{\bf x}^\ast]$ in $\Gamma$ and contours of $L$.}
\label{fig-Lorenz}
\end{figure}

\begin{figure}[htbp]\em
\begin{minipage}{1\hsize}
\centering
\includegraphics[width=6.0cm]{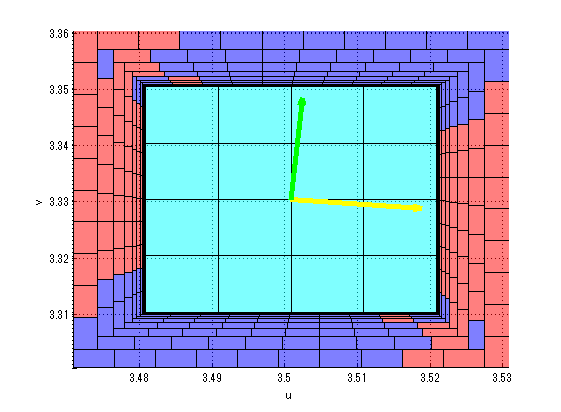}
\end{minipage}
\caption{Eigendirections of the Poincar\'{e} map for (\ref{Lorenz-trans}) and contours of the Lyapunov function.}
\label{fig-LorenzEvec}
\end{figure}

Next, we consider a ${\bf m}$-Lyapunov function around the saddle fixed point.
Let ${\bf m} = (m_1,m_2) = \{1, 1/4\}$, which yields
\begin{equation*}
M^\ast = \begin{pmatrix}
-1 & 0\\
0 & \frac{1}{4}
\end{pmatrix},
\end{equation*}
and contours of the ${\bf m}$-Lyapunov function $L_{M^\ast}$ is shown in Fig \ref{fig-LorenzContour} as well as the Lyapunov domain.

\begin{figure}[htbp]\em
\begin{minipage}{1\hsize}
\centering
\includegraphics[width=6.0cm]{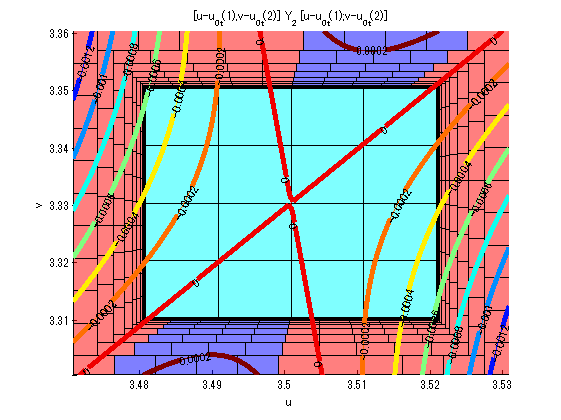}
\end{minipage}
\caption{Contours of $L_{M^\ast}$ for (\ref{Lorenz-trans}).}
\label{fig-LorenzContour}
\end{figure}

We observe that we cannot validate the strict negative definiteness of $B_{M^\ast}({\bf w})$ anywhere around the fixed point if $m_1 < 1$.
Similarly, if $m_2 < 1/4$, the validated domain where $B_{M^\ast}({\bf w})$ is strictly negative definite becomes smaller.
The Lyapunov domain of $L_{M^\ast}$ is smaller than that of $L$ and it is clustered on the direction of stable manifolds. 
Indeed, the blue region in Fig. \ref{fig-LorenzContour} is concentrated on $L^{-1}(0,\infty)$.
This observation is due to the small modulus of $\lambda_1$ (eigenvalue associating the stable direction).

\bigskip
Computation times for our verifications are about $43$ hours in Stage $1$ with $C^1$-Lohner method, and about $19$ hours in Stage $2$ with the ordinary (i.e. $C^0$-)Lohner method. Finally, verification of $L_{M^\ast}$ takes about $2$ seconds.

\section*{Conclusion}
We have discussed a systematic procedure of Lyapunov functions with explicit domains of definition, called Lyapunov domains, with computer assistance.

Computer assisted analysis such as interval arithmetics enables us to construct Lyapunov functions around fixed points with explicit ranges.
Our Lyapunov functions are quadratic and hence we can easily describe local dynamics in Lyapunov domains, such as level sets of Lyapunov functions and enclosures of the stable and the unstable manifolds of fixed points.

\bigskip
In our procedure, there are two stages for constructing Lyapunov domains for flows for a given quadratic structure:
\bigskip
\begin{itemize}
\item {\bf Stage 1: Strict negative definiteness of the associated matrix $A({\bf z})$}. 
\end{itemize}
This verification around equilibria yields not only validation of Lyapunov domains but also local uniqueness of equilibria.
We have proved that all hyperbolic equilibria locally admit Lyapunov functions of the form (\ref{Lyapunov-flow}). 
In particular, the negative definiteness of $A$ reflects hyperbolicity of the Jacobian matrix at a hyperbolic equilibrium ${\bf x}^\ast$.
We have also observe that the negative definiteness is intrinsically the same as cone conditions discussed in e.g. \cite{ZCov}.
Additionally, we have discussed another validation procedure of Lyapunov functions as well as hyperbolicity of equilibria for flows in the preceding work \cite{Mat}. 
It turns out that the condition discussed in \cite{Mat} is stronger than the negative definiteness of $A({\bf z})$.

\begin{itemize}
\item {\bf Stage 2: Direct calculations of $(dL/dt) < 0$ along solution orbits}.
\end{itemize}
This verification stage is effective away from equilibria.
Even if Stage 1 is failed, this stage gives us further extension of Lyapunov domains of a given quadratic function.
Combining validations in these two stages, Lyapunov domains can be extended away from equilibria.
\par
\bigskip
We also have derived a validation procedure of Lyapunov functions for discrete dynamical systems.
A sufficient condition for validating Lyapunov functions, strict negative definiteness of the associated matrix $B({\bf z})$, guarantees the existence of Lyapunov functions and local uniqueness of fixed points in given domains. 
We mentioned that this sufficient condition is equivalent to that of cone conditions for maps \cite{ZCov}.
\par
As demonstrations of applicability of our procedures, we have shown several validation examples of Lyapunov functions with computer assistance.
Adjustment of the number of division of domains or time steps of solvers, and choice of verification stage have potentials to extend Lyapunov domains not only around fixed points but also away from them.
We also observe that validation of Lyapunov domains becomes difficult in eigendirections whose associated eigenvalues have real parts close to zero (for flows) or moduli close to one (for maps).
These observations will give us several guiding principles for studying asymptotic behavior around fixed points in terms of Lyapunov functions or cones with computer assistance.

\bigskip
The effectiveness of cones with computer assistance are already seen in various fields (e.g. \cite{C, KWZ2007, Mat2}).
As mentioned above, criteria in Stage 1 are equivalent to a sufficient condition of cone conditions.
Additionally, Lyapunov domains admit re-parameterization of trajectories in terms of values of Lyapunov functions, called Lyapunov tracing.
This technique indicates that the parameter determining solution trajectories can be chosen both in finite and infinite parameter ranges so that dynamical information for appropriate systems is kept.
It enables us to track asymptotic behavior with suitable parameters in Lyapunov domains so that various phenomena including singular ones such as blow-up of solutions  (e.g. \cite{TMSTMO2016}) can be treated by using classical approaches of dynamical systems.

Our studies give us a comprehensive understanding of Lyapunov functions and cones, which will lead to studies of asymptotic behavior of dynamical systems around fixed points from a variety of viewpoints.

\bigskip
We end this paper providing a further direction of this research.

\bigskip
\begin{itemize}
\item[{\bf Lyapunov functions for non-hyperbolic equilibria.}] 
\end{itemize}
Our validation criteria for Lyapunov domains rely on hyperbolicity of equilibria. 
In other words, our current method cannot be applied to studies of dynamics around non-hyperbolic equilibria such as fold points, turning points, symmetry-breaking bifurcation points, and center points in Hamiltonian systems.
Dynamics around such points are essentially determined by higher order terms of vector fields around equilibria. 
A holistic discussion of higher-order terms for vector fields or maps would help us with constructing Lyapunov-type functions for describing  behavior of local dynamical systems, while it depends on normal forms of non-hyperbolic equilibria.

\section*{Acknowledgements}
This work was supported by CREST, JST.

% You may incorporate your references as follows in your main tex file.
% Using BibTex is not recommended but can be handled.

\medskip
% The data information below will be filled by AIMS editorial staff
Received xxxx 20xx; revised xxxx 20xx.
\medskip

\end{document}